\documentclass[a4paper]{amsart}

\pdfoutput=1

\usepackage{subfiles}

\usepackage{amsmath,amsthm,amsfonts,amssymb,mathrsfs}
\usepackage{color}
\usepackage{aliascnt}
\usepackage{bbm}
\usepackage{enumerate}
\usepackage{enumitem}
\usepackage[colorlinks=true,allcolors=blue!80!black, bookmarksdepth=3]{hyperref}
\usepackage[capitalize,nameinlink,noabbrev]{cleveref}
\usepackage{bbold}
\usepackage{mathtools}
\usepackage{faktor}
\usepackage{comment}
\usepackage{stmaryrd}

\usepackage{tikz}
\usetikzlibrary{shapes.geometric}
\usetikzlibrary{patterns}
\tikzset{
v/.style={draw, fill, circle, minimum size=1.5mm, inner sep=0},
b/.style={draw , regular polygon,regular polygon sides=4, minimum size=1.5mm, inner sep=.5mm},
e/.style={very thick},
vs/.style={draw, fill, circle, minimum size=1mm, inner sep=0},
bs/.style={draw,  regular polygon,regular polygon sides=4, minimum size=2mm, inner sep=0mm},
es/.style={thick}
}
\newlength{\nodeheight}
\newlength{\nodewidth}
\usetikzlibrary{arrows,matrix,positioning}
\usepackage{tikz-cd}

\usepackage{tikz-cd}

\usepackage[backend=bibtex, style=alphabetic, url=false, maxbibnames=99, sorting=nyt]{biblatex}
\RequirePackage{doi}

\numberwithin{thmcounter}{section}
\newaliascnt{thmauto}{thmcounter}

\newaliascnt{Defauto}{thmcounter}

\newaliascnt{exauto}{thmcounter}

\newaliascnt{lemauto}{thmcounter}

\newaliascnt{propauto}{thmcounter}

\newaliascnt{corauto}{thmcounter}

\newaliascnt{remauto}{thmcounter}

\newaliascnt{clmauto}{thmcounter}

\newaliascnt{quesauto}{thmcounter}

\newtheorem{theorem}[thmauto]{Theorem}
\newtheorem{lemma}[lemauto]{Lemma}
\newtheorem{proposition}[propauto]{Proposition}
\newtheorem{corollary}[corauto]{Corollary}

\newtheorem{question}[quesauto]{Question}

\theoremstyle{definition}
\newtheorem{definition}[Defauto]{Definition}
\newtheorem{example}[exauto]{Example}
\newtheorem{remark}[remauto]{Remark}

\newcommand{\fakeenv}{} 
\newenvironment{restate}[2]  
{
  \renewcommand{\fakeenv}{#2}   
                                
  \theoremstyle{plain}
  \newtheorem*{\fakeenv}{#1~\ref{#2}} 
  \begin{\fakeenv}  
}
{ \end{\fakeenv} }

\newenvironment{rerestate}[3]  
{
  \renewcommand{\fakeenv}{#3}   
  \theoremstyle{plain}
  \newtheorem*{\fakeenv}{#1~\ref{#2}}  
  \begin{\fakeenv}  
}
{ \end{\fakeenv} }

\numberwithin{equation}{section}

\DeclareMathOperator{\im}{im}
\DeclareMathOperator{\coker}{coker}

\DeclareMathOperator{\Tor}{Tor}
\DeclareMathOperator{\Ext}{Ext}
\providecommand{\id}{\ensuremath\mathrm{id}}

\renewcommand{\Bar}{\mathrm{Bar}}

\DeclareMathOperator{\TL}{TL}

\DeclareMathOperator{\cell}{S}

\newcommand{\InnermostCupComplex}{\mathrm{Inn}}
\newcommand{\SubmaxInnermostCupComplex}{\mathrm{M}}
\newcommand{\AugmentedInnermostCupComplex}{\mathrm{Inn}^\text{aug}}
\newcommand{\Z}{\mathbb{Z}}
\newcommand{\OutermostCupComplex}{\mathrm{Out}}

\newcommand{\CPL}{L}
\newcommand{\derivedCell}{\mathrm{D}\mathrm{S}}
\newcommand{\derivedOutermost}{\mathrm{D}\mathrm{Out}}
\newcommand{\derivedInnermost}{\mathrm{D}\mathrm{Inn}}

\newcommand{\MayerVie}{V}
\newcommand{\Lmax}{L_\text{max}}

\definecolor{indiguy}{RGB}{195,179,227}

\title[Even Temperley--Lieb algebras and the dga of planar loops]{Even Temperley--Lieb algebras\\and the dga of planar loops}

\author{Rachael Boyd}
\address{School of Mathematics and Statistics, University of Glasgow, Glasgow G12 8QQ, UK}
\email{rachael.boyd@glasgow.ac.uk}
\author{Guy Boyde}
\address{Department of Mathematics, Vrije Universiteit Amsterdam, De Boelelaan 1111, 1081 HV Amsterdam, The Netherlands}
\email{g.boyde@vu.nl}
\author{Oscar Randal-Williams}
\thanks{}
\address{Centre for Mathematical Sciences, Wilberforce Road, Cambridge CB3 0WB, UK}
\email{or257@cam.ac.uk}
\author{Robin J. Sroka}
\thanks{}
\address{Mathematisches Institut, Universität Münster, Einsteinstrasse 62, 48149 Münster, Germany}
\email{robinjsroka@uni-muenster.de}


\begin{document}

\begin{abstract}
    We show that the homology of a Temperley--Lieb algebra on an even number of strands has a rich algebraic structure and is highly nontrivial in general. This is achieved by proving that it is entirely governed by a differential graded algebra: the differential graded algebra of planar loops. We provide a small model for this dga, and use it to obtain consequences on homology.
\end{abstract}

\maketitle
\setcounter{tocdepth}{1}
\tableofcontents

\section{Introduction}
Given $R$ a unital commutative ring, $a \in R$ and $n \geq 0$ a natural number, the Temperley--Lieb algebra is a diagrammatic algebra denoted $\TL_n(R,a)$. The goal of this paper is to calculate the homology groups $\Tor_*^{\TL_n(R,a)}(R,R)$ of this algebra when $n$ is even, complementing the calculation \cite{Sroka} of the last-named author which treated the case $n$ odd. While the homology turned out to be trivial in positive degrees when $n$ is odd, this work shows that it has a rich structure when $n$ is even.

The core of the calculation is to describe the homology of an associated object, which we call the differential graded algebra of planar loops. This object can be introduced in a stand-alone way, so we first do so and describe our results in this direction, then later explain their connection to Temperley--Lieb algebras. 

\subsection{The differential graded algebra of planar loops}
We introduce the differential graded $R$-algebra (dga) of planar loops $\CPL(2n) = \CPL(2n;R,a)$. Its chain module in degree $p$, $\CPL(2n)_p$, is defined to be the free $R$-module on the set of all \emph{systems of planar loops (of height $2n$) pinned by $p$ bars}, described as follows.
\begin{itemize}
    \item Place $p$ vertical lines (\emph{bars}) in the plane. Mark $2n$ points on each bar: we call these \emph{nodes}.
    \item A basis element of $\CPL(2n)_p$ is an isotopy class of a number of disjoint circles (\emph{loops}) in the plane. Then we require that each loop contains at least two nodes, and that every node is in precisely one loop, see \cref{fig:whatIsB}. We think of the bars as `pinning' the loops.
\end{itemize} 
The differential of $\CPL(2n)$ is given as an alternating sum of face maps. The $k$-th face map removes the $k$-th bar: we take the resulting isotopy class (now with $p-1$ bars), and for each loop which is completely `unpinned' by the removal of this bar, we discard the loop and multiply by $a \in R$ (\cref{fig:whatIsB}, Part $\mathrm{(a)}$). Juxtaposing two such pictures defines a graded multiplication which gives $\CPL(2n)$ the structure of a dga (\cref{fig:whatIsB}, Part $\mathrm{(b)}$). 

The dga $\CPL(2n)$ also arises from a bar construction on the Temperley--Lieb algebra $\TL_{2n}(R,a)$, see \cref{definition: dga of planar loops}, and it is this algebraic description that we will work with in the main body of this article.

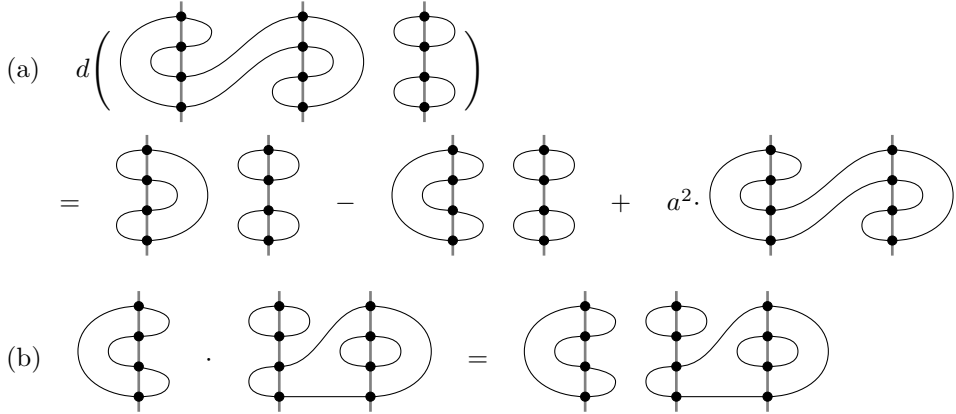
\begin{figure}[h!]
\begin{flushleft}
$\mathrm{(a)}$
\quad
$d \Biggl($
    \begin{tikzpicture}[scale=0.4, baseline=(base)]
        \def\widthscale{4};

        \coordinate (base) at (0,2);
        \draw[gray,line width = 1](0,0.5)--(0,4.5);
        \draw[gray,line width = 1](\widthscale,0.5)--(\widthscale,4.5);
        \draw[gray,line width = 1](2*\widthscale,0.5)--(2*\widthscale,4.5);

        \foreach \y in {1,2, 3,4}{
            \foreach \x in {0,1,2}{
                \draw[fill=black]  (\widthscale*\x,\y) circle [radius=0.15];
                \draw (\widthscale*\x,\y) node[right] {};
            }
            
        }

        \draw (0,2) to[out=180,in=-90] (-1,2.5) to[out=90,in=180] (0,3);
        \draw (0,3) to[out=0,in=-90] (1,3.5) to[out=90,in=180] (0,4);
        \draw (0,1) to[out=180,in=-90] (-2,2.5) to[out=90,in=180] (0,4);

        \draw (0,1) to[out=0,in=180] (\widthscale,3);
        \draw (0,2) to[out=0,in=180] (\widthscale,4);
        
        \draw (\widthscale,2) to[out=0,in=-90] (\widthscale+1,2.5) to[out=90,in=0] (\widthscale,3);
        \draw (\widthscale,1) to[out=180,in=-90] (\widthscale-1,1.5) to[out=90,in=180] (\widthscale,2);
        \draw (\widthscale,1) to[out=0,in=-90] (\widthscale+2,2.5) to[out=90,in=0] (\widthscale,4);

        \draw (2*\widthscale,1) to[out=180,in=-90] (2*\widthscale-1,1.5) to[out=90,in=180] (2*\widthscale,2);
        \draw (2*\widthscale,1) to[out=0,in=-90] (2*\widthscale+1,1.5) to[out=90,in=0] (2*\widthscale,2);

        \draw (2*\widthscale,3) to[out=180,in=-90] (2*\widthscale-1,3.5) to[out=90,in=180] (2*\widthscale,4);
        \draw (2*\widthscale,3) to[out=0,in=-90] (2*\widthscale+1,3.5) to[out=90,in=0] (2*\widthscale,4);
    \end{tikzpicture}
$\Biggr) $
\\
\vspace{0.1cm}
\quad
\quad
$=$
\quad
    \begin{tikzpicture}[scale=0.4, baseline=(base)]
        \def\widthscale{4};

        \coordinate (base) at (0,2);
        \draw[gray,line width = 1](\widthscale,0.5)--(\widthscale,4.5);
        \draw[gray,line width = 1](2*\widthscale,0.5)--(2*\widthscale,4.5);

        \foreach \y in {1,2, 3,4}{
            \foreach \x in {1,2}{
                \draw[fill=black]  (\widthscale*\x,\y) circle [radius=0.15];
                \draw (\widthscale*\x,\y) node[right] {};
            }
            
        }

        \draw (\widthscale,3) to[out=180,in=-90] (\widthscale-1,3.5) to[out=90,in=180] (\widthscale,4);
        
        \draw (\widthscale,2) to[out=0,in=-90] (\widthscale+1,2.5) to[out=90,in=0] (\widthscale,3);
        \draw (\widthscale,1) to[out=180,in=-90] (\widthscale-1,1.5) to[out=90,in=180] (\widthscale,2);
        \draw (\widthscale,1) to[out=0,in=-90] (\widthscale+2,2.5) to[out=90,in=0] (\widthscale,4);

        \draw (2*\widthscale,1) to[out=180,in=-90] (2*\widthscale-1,1.5) to[out=90,in=180] (2*\widthscale,2);
        \draw (2*\widthscale,1) to[out=0,in=-90] (2*\widthscale+1,1.5) to[out=90,in=0] (2*\widthscale,2);

        \draw (2*\widthscale,3) to[out=180,in=-90] (2*\widthscale-1,3.5) to[out=90,in=180] (2*\widthscale,4);
        \draw (2*\widthscale,3) to[out=0,in=-90] (2*\widthscale+1,3.5) to[out=90,in=0] (2*\widthscale,4);
    \end{tikzpicture}
\quad
$- $
\quad
    \begin{tikzpicture}[scale=0.4, baseline=(base)]
        \def\widthscale{3};

        \coordinate (base) at (0,2);
        \draw[gray,line width = 1](0,0.5)--(0,4.5);
        \draw[gray,line width = 1](\widthscale,0.5)--(\widthscale,4.5);

        \foreach \y in {1,2, 3,4}{
            \foreach \x in {0,1}{
                \draw[fill=black]  (\widthscale*\x,\y) circle [radius=0.15];
                \draw (\widthscale*\x,\y) node[right] {};
            }
            
        }

        \draw (0,2) to[out=180,in=-90] (-1,2.5) to[out=90,in=180] (0,3);
        \draw (0,3) to[out=0,in=-90] (1,3.5) to[out=90,in=180] (0,4);
        \draw (0,1) to[out=180,in=-90] (-2,2.5) to[out=90,in=180] (0,4);
        \draw (0,1) to[out=0,in=-90] (1,1.5) to[out=90,in=180] (0,2);

        \draw (\widthscale,1) to[out=180,in=-90] (\widthscale-1,1.5) to[out=90,in=180] (\widthscale,2);
        \draw (\widthscale,1) to[out=0,in=-90] (\widthscale+1,1.5) to[out=90,in=0] (\widthscale,2);

        \draw (\widthscale,3) to[out=180,in=-90] (\widthscale-1,3.5) to[out=90,in=180] (\widthscale,4);
        \draw (\widthscale,3) to[out=0,in=-90] (\widthscale+1,3.5) to[out=90,in=0] (\widthscale,4);
    \end{tikzpicture}
\quad
$+$ 
\quad
$a^2 \cdot$
    \begin{tikzpicture}[scale=0.4, baseline=(base)]
        
        \def\widthscale{4};

        \coordinate (base) at (0,2);
        \draw[gray,line width = 1](0,0.5)--(0,4.5);
        \draw[gray,line width = 1](\widthscale,0.5)--(\widthscale,4.5);

        \foreach \y in {1,2, 3,4}{
            \foreach \x in {0,1}{
                \draw[fill=black]  (\widthscale*\x,\y) circle [radius=0.15];
                \draw (\widthscale*\x,\y) node[right] {};
            }
            
        }

        \draw (0,2) to[out=180,in=-90] (-1,2.5) to[out=90,in=180] (0,3);
        \draw (0,3) to[out=0,in=-90] (1,3.5) to[out=90,in=180] (0,4);
        \draw (0,1) to[out=180,in=-90] (-2,2.5) to[out=90,in=180] (0,4);

        \draw (0,1) to[out=0,in=180] (\widthscale,3);
        \draw (0,2) to[out=0,in=180] (\widthscale,4);
        
        \draw (\widthscale,2) to[out=0,in=-90] (\widthscale+1,2.5) to[out=90,in=0] (\widthscale,3);
        \draw (\widthscale,1) to[out=180,in=-90] (\widthscale-1,1.5) to[out=90,in=180] (\widthscale,2);
        \draw (\widthscale,1) to[out=0,in=-90] (\widthscale+2,2.5) to[out=90,in=0] (\widthscale,4);
    \end{tikzpicture}
\\
\vspace{0.4cm}
$\mathrm{(b)}$
\quad
    \begin{tikzpicture}[scale=0.4, baseline=(base)]
        
        \def\widthscale{3};

        \coordinate (base) at (0,2);
        \draw[gray,line width = 1](0,0.5)--(0,4.5);

        \foreach \y in {1,2, 3,4}{
            \foreach \x in {0}{
                \draw[fill=black]  (\widthscale*\x,\y) circle [radius=0.15];
                \draw (\widthscale*\x,\y) node[right] {};
            }
            
        }

        \draw (0,2) to[out=180,in=-90] (-1,2.5) to[out=90,in=180] (0,3);
        \draw (0,3) to[out=0,in=-90] (1,3.5) to[out=90,in=180] (0,4);
        \draw (0,1) to[out=180,in=-90] (-2,2.5) to[out=90,in=180] (0,4);
        \draw (0,1) to[out=0,in=-90] (1,1.5) to[out=90,in=180] (0,2);
    \end{tikzpicture}
\quad
$\cdot$
\quad
     \begin{tikzpicture}[scale=0.4, baseline=(base)]
        \def\widthscale{3};

        \coordinate (base) at (0,2);
        \draw[gray,line width = 1](\widthscale,0.5)--(\widthscale,4.5);
        \draw[gray,line width = 1](2*\widthscale,0.5)--(2*\widthscale,4.5);

        \foreach \y in {1,2, 3,4}{
            \foreach \x in {1,2}{
                \draw[fill=black]  (\widthscale*\x,\y) circle [radius=0.15];
                \draw (\widthscale*\x,\y) node[right] {};
            }
            
        }

        \draw (\widthscale,3) to[out=180,in=-90] (\widthscale-1,3.5) to[out=90,in=180] (\widthscale,4);
        \draw (\widthscale,1) to[out=180,in=-90] (\widthscale-1,1.5) to[out=90,in=180] (\widthscale,2);
        \draw (\widthscale,3) to[out=0,in=-90] (\widthscale+1,3.5) to[out=90,in=0] (\widthscale,4);

        \draw (\widthscale,1) to[out=0,in=180] (2*\widthscale,1);
        \draw (\widthscale,2) to[out=0,in=180] (2*\widthscale,4);

        \draw (2*\widthscale,2) to[out=0,in=-90] (2*\widthscale+1,2.5) to[out=90,in=0] (2*\widthscale,3);
        \draw (2*\widthscale,2) to[out=180,in=-90] (2*\widthscale-1,2.5) to[out=90,in=180] (2*\widthscale,3);
        \draw (2*\widthscale,1) to[out=0,in=-90] (2*\widthscale+2,2.5) to[out=90,in=0] (2*\widthscale,4);
    \end{tikzpicture}
\quad
$=$
\quad
        \begin{tikzpicture}[scale=0.4, baseline=(base)]
       
        \def\widthscale{3};

        \coordinate (base) at (0,2);
        \draw[gray,line width = 1](0,0.5)--(0,4.5);
        \draw[gray,line width = 1](\widthscale,0.5)--(\widthscale,4.5);
        \draw[gray,line width = 1](2*\widthscale,0.5)--(2*\widthscale,4.5);

        \foreach \y in {1,2, 3,4}{
            \foreach \x in {0,1,2}{
                \draw[fill=black]  (\widthscale*\x,\y) circle [radius=0.15];
                \draw (\widthscale*\x,\y) node[right] {};
            }
            
        }

        \draw (0,2) to[out=180,in=-90] (-1,2.5) to[out=90,in=180] (0,3);
        \draw (0,3) to[out=0,in=-90] (1,3.5) to[out=90,in=180] (0,4);
        \draw (0,1) to[out=180,in=-90] (-2,2.5) to[out=90,in=180] (0,4);
        \draw (0,1) to[out=0,in=-90] (1,1.5) to[out=90,in=180] (0,2);
        
        \draw (\widthscale,3) to[out=180,in=-90] (\widthscale-1,3.5) to[out=90,in=180] (\widthscale,4);
        \draw (\widthscale,1) to[out=180,in=-90] (\widthscale-1,1.5) to[out=90,in=180] (\widthscale,2);
        \draw (\widthscale,3) to[out=0,in=-90] (\widthscale+1,3.5) to[out=90,in=0] (\widthscale,4);

        \draw (\widthscale,1) to[out=0,in=180] (2*\widthscale,1);
        \draw (\widthscale,2) to[out=0,in=180] (2*\widthscale,4);

        \draw (2*\widthscale,2) to[out=0,in=-90] (2*\widthscale+1,2.5) to[out=90,in=0] (2*\widthscale,3);
        \draw (2*\widthscale,2) to[out=180,in=-90] (2*\widthscale-1,2.5) to[out=90,in=180] (2*\widthscale,3);
        \draw (2*\widthscale,1) to[out=0,in=-90] (2*\widthscale+2,2.5) to[out=90,in=0] (2*\widthscale,4);
    \end{tikzpicture}
\end{flushleft}

    \caption{With $2n=4$: $\mathrm{(a)}$ a typical basis element of $\CPL(4)_3$, with the differential $d$ applied (`barwise') to that element, and $\mathrm{(b)}$ the (`juxtaposition') product.}
    \label{fig:whatIsB}
\end{figure}

Our main theorem gives a small model for the dga of planar loops. To state it we introduce the notation $\Phi \in \CPL(2n;R,a)_1$ for the element shown in \cref{fig: PhiInIntro}. Its key feature is that the underlying diagram consists of a single loop, so it satisfies $d(\Phi)=a$, and is a cycle if and only if $a=0$. 

    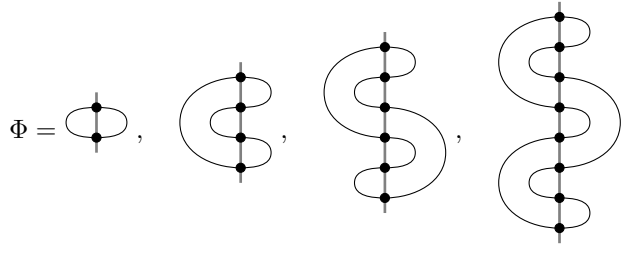
\begin{figure}[h!]
    $\Phi =$
        \begin{tikzpicture}[scale=0.4, baseline=(base)]
        \coordinate (base) at (0,1);
        
        \draw[gray,line width = 1](6,0.5)--(6,2.5);

        \foreach \y in {1,2}{
            \draw[fill=black]  (6,\y) circle [radius=0.15];
            \draw (6,\y) node[right] {};
        }
        
        \draw (6,1) to[out=0,in=-90] (7,1.5) to[out=90,in=0] (6,2);
        \draw (6,1) to[out=180,in=-90] (5,1.5) to[out=90,in=180] (6,2);
    \end{tikzpicture}
    ,
    \quad
        \begin{tikzpicture}[scale=0.4, baseline=(base)]
        \coordinate (base) at (0,2);
        
        \draw[gray,line width = 1](6,0.5)--(6,4.5);

        \foreach \y in {1,2, 3,4}{
            \draw[fill=black]  (6,\y) circle [radius=0.15];
            \draw (6,\y) node[right] {};
        }

                 \foreach \x in {1}{
        \draw (6,\x) to[out=180,in=-90] (4,\x + 1.5) to[out=90,in=180] (6,\x + 3);
        \draw (6,\x+1) to[out=180,in=-90] (5,\x + 1.5) to[out=90,in=180] (6,\x + 2);
        }

        \draw (6,1) to[out=0,in=-90] (7,1.5) to[out=90,in=0] (6,2);
        
        \draw (6,3) to[out=0,in=-90] (7,3.5) to[out=90,in=0] (6,4);
    \end{tikzpicture}
    ,
    \quad
        \begin{tikzpicture}[scale=0.4, baseline=(base)]
        \coordinate (base) at (0,3);
        
        \draw[gray,line width = 1](6,0.5)--(6,6.5);

        \foreach \y in {1,2, 3,4,5,6}{
            \draw[fill=black]  (6,\y) circle [radius=0.15];
            \draw (6,\y) node[right] {};
        }

                 \foreach \x in {3}{
        \draw (6,\x) to[out=180,in=-90] (4,\x + 1.5) to[out=90,in=180] (6,\x + 3);
        \draw (6,\x+1) to[out=180,in=-90] (5,\x + 1.5) to[out=90,in=180] (6,\x + 2);
        }

        \draw (6,1) to[out=180,in=-90] (5,1.5) to[out=90,in=180] (6,2);

                 \foreach \x in {1}{
        \draw (6,\x) to[out=0,in=-90] (8,\x + 1.5) to[out=90,in=0] (6,\x + 3);
        \draw (6,\x+1) to[out=0,in=-90] (7,\x + 1.5) to[out=90,in=0] (6,\x + 2);
        }
        
        \draw (6,5) to[out=0,in=-90] (7,5.5) to[out=90,in=0] (6,6);
    \end{tikzpicture}
    ,
    \quad
        \begin{tikzpicture}[scale=0.4, baseline=(base)]
        \coordinate (base) at (0,4);
        
        \draw[gray,line width = 1](6,0.5)--(6,8.5);

        \foreach \y in {1,2, 3,4,5,6,7,8}{
            \draw[fill=black]  (6,\y) circle [radius=0.15];
            \draw (6,\y) node[right] {};
        }

                 \foreach \x in {1,5}{
        \draw (6,\x) to[out=180,in=-90] (4,\x + 1.5) to[out=90,in=180] (6,\x + 3);
        \draw (6,\x+1) to[out=180,in=-90] (5,\x + 1.5) to[out=90,in=180] (6,\x + 2);
        }

         \foreach \x in {1,4,7}{
        \draw (6,\x) to[out=0,in=-90] (7,\x + 0.5) to[out=90,in=0] (6,\x + 1);
        }
        
        \draw (6,3) to[out=0,in=-90] (8,3 + 1.5) to[out=90,in=0] (6,3 + 3);
    \end{tikzpicture}
    \caption{The element $\Phi \in \CPL(2n)_1$, shown (left-to-right) for $2n=2,4,6,8$.}\label{fig: PhiInIntro}
\end{figure}

\begin{theorem}\label{thm:Main}
There is a weak equivalence
$$(T_R[x_1, x_3, \ldots, x_{2 n-1}], d) \overset{\sim}\to \CPL(2n; R,a)$$
of differential graded $R$-algebras, where the left-hand side is the tensor algebra on the generators $x_i$ of degree $i$ equipped with the differential given by 
$$d(x_1)=a \text{ and } d(x_{2i-1}) = \sum_{\substack{j+k=i \\ j,k>0}} \binom{i}{j} x_{2j-1}x_{2k-1} \text{ for } 2i-1>1.$$
Under this equivalence $x_1$ is mapped to $\Phi$.
\end{theorem}

This is a quasi-free (hence cofibrant) model for $\CPL(2n;R,a)$, and is minimal in any reasonable sense. The model is completely explicit and can be directly used to calculate $H_*(\CPL(2n;R,a))$ to some extent. However, as we shall see in \cref{subsection:calculations-the-parameter-zero-case}, the most powerful computational tool is to observe that in the case $a=0$ this model is simply the cobar construction of an explicit coalgebra over $R$. The $R$-linear dual of this coalgebra is the $R$-algebra
$$A_n := \Gamma_R[y]/(y^{[i]}, i \geq n+1),$$
i.e.\ a truncated divided power algebra. This leads to an algebra isomorphism
$$H_*(\CPL(2n;R,0)) \cong \Ext^*_{A_n}(R,R).$$
Hence, the homology of $\CPL(2n;R,0)$ may be calculated by a minimal resolution of $R$ as an $A_n$-module, instead of the cobar complex $(T_R[x_1, x_3, \ldots, x_{2 n-1}], d)$. This renders calculations much more tractable. In \cref{sec:ChangingCoeff} we describe some simple tools for reducing calculations over $(R,a)$ to calculations over $(\mathbb{Z}[a],a)$, and then to calculations over $(\mathbb{Z}, 0)$, so treating the case $a=0$ is not a major loss of generality. 

We include the following selection of explicit calculations obtained using the model, the observation above, and several other techniques. We will justify these calculations in \cref{sec:Calc}, where we give slightly more precise statements. 

\begin{corollary}\label{cor:MainCalculations}\mbox{}
\begin{enumerate}
\item $H_0(\CPL(2n;R,a)) = R/a$, so the $R$-module structure on $H_i(\CPL(2n;R,a))$ descends to $R/a$, i.e.\ the homology groups are all annihilated by $a$.

\item For $2n=2$ there is an isomorphism
$$H_{*}(\CPL(2; R, 0)) \cong R[\Phi]$$
of graded algebras, with $|\Phi| = 1$.

\item If $n! \in R^\times$ and $2n \geq 4$ then there is an isomorphism
$$H_{*}(\CPL(2n; R, 0)) \cong R[\Phi,  \alpha]/(\Phi^2)$$
of graded algebras, with $|\Phi| = 1$ and $|\alpha| = 2n$. Furthermore, the class $\alpha$ may be described as the $(n + 1)$-fold Massey product $\langle \Phi, \Phi, \ldots, \Phi \rangle$.

\item If $a \in R$ is not a zerodivisor and $n! \in (R/a)^\times$ then there is an isomorphism 
$$H_{*}(\CPL(2n; R, a)) \cong R/a[\beta]$$
of graded algebras, with $|\beta|=2n$. The class $\beta$ is uniquely determined by the property that under the change-of-rings map for $(R,a) \to (R/a,0)$ it maps to the class $\alpha \in H_{2n}(\CPL(2n ; R/a, 0))$ described in (c) (or to $\Phi^2$ if $2n=2$).

\item $H_*(\CPL(4;\mathbb{Z},0))$ only has 2-torsion, which is all simple. Modulo torsion there is an isomorphism
$$H_*(\CPL(4;\mathbb{Z},0))/\text{torsion} \cong \mathbb{Z}[\Phi, \gamma]/(\Phi^2)$$
of graded algebras, with $|\Phi|=1$ and $|\gamma|=4$. We may take $\gamma := \langle \Phi, 2\Phi, \Phi\rangle$. The Poincar{\'e} series for the $\mathbb{Z}/2$-dimension of the simple 2-torsion is
$$\frac{t^2}{(1-t - t^3 ) (1 - t^4)} = t^2 + t^3 + t^4 + 2 t^5 + 4 t^6 + 5 t^7 + 7 t^8
+ \cdots.$$

\end{enumerate}
\end{corollary}

\subsection{Relation to the homology of Temperley--Lieb algebras}

The algebra of planar loops arose naturally while we were investigating the homology, by which we mean $\Tor_*^{\TL_{2n}(R,a)}(R,R)$, of the Temperley--Lieb algebra $\TL_{2n}=\TL_{2n}(R,a)$. Recall that these algebras are free $R$-modules on a basis of Temperley--Lieb $(2n,2n)$-diagrams. These diagrams correspond to planar matchings between two bars each with $2n$ nodes, with a multiplication which concatenates diagrams and replaces loops with factors of $a$. An example multiplication is shown in \cref{fig: TLalgsInIntro}.

\begin{figure}[h]

        \begin{tikzpicture}[scale=0.4, baseline=(base)]
        \def\widthscale{4};

        \coordinate (base) at (0,2);
        \draw[gray,line width = 1](0,0.5)--(0,4.5);
        \draw[gray,line width = 1](\widthscale,0.5)--(\widthscale,4.5);

        \foreach \y in {1,2, 3,4}{
            \foreach \x in {0,1}{
                \draw[fill=black]  (\widthscale*\x,\y) circle [radius=0.15];
                \draw (\widthscale*\x,\y) node[right] {};
            }
            
        }

        \draw (0,3) to[out=0,in=-90] (1,3.5) to[out=90,in=0] (0,4);
        
        \draw (0,1) to[out=0,in=180] (\widthscale,1);
        \draw (0,2) to[out=0,in=180] (\widthscale,4);

        \draw (\widthscale,2) to[out=180,in=-90] (\widthscale-1,2.5) to[out=90,in=180] (\widthscale,3);
        \end{tikzpicture}
        \
        $\cdot$
        \
        \begin{tikzpicture}[scale=0.4, baseline=(base)]
        \def\widthscale{4};

        \coordinate (base) at (0,2);
        \draw[gray,line width = 1](0,0.5)--(0,4.5);
        \draw[gray,line width = 1](\widthscale,0.5)--(\widthscale,4.5);

        \foreach \y in {1,2, 3,4}{
            \foreach \x in {0,1}{
                \draw[fill=black]  (\widthscale*\x,\y) circle [radius=0.15];
                \draw (\widthscale*\x,\y) node[right] {};
            }
            
        }

        \draw (0,2) to[out=0,in=-90] (1,2.5) to[out=90,in=0] (0,3);
        \draw (0,1) to[out=0,in=-90] (2,2.5) to[out=90,in=0] (0,4);
        
        \draw (\widthscale,1) to[out=180,in=-90] (\widthscale-1,1.5) to[out=90,in=180] (\widthscale,2);
        \draw (\widthscale,3) to[out=180,in=-90] (\widthscale-1,3.5) to[out=90,in=180] (\widthscale,4);
        \end{tikzpicture}
     \
     $=$
     \
        \begin{tikzpicture}[scale=0.4, baseline=(base)]
        \def\widthscale{3.5};

        \coordinate (base) at (0,2);
        \draw[gray,line width = 1](0,0.5)--(0,4.5);
        \draw[gray,line width = 1](2*\widthscale,0.5)--(2*\widthscale,4.5);

        \foreach \y in {1,2, 3,4}{
            \foreach \x in {0,2}{
                \draw[fill=black]  (\widthscale*\x,\y) circle [radius=0.15];
                \draw (\widthscale*\x,\y) node[right] {};
            }
            
        }

        \draw (0,3) to[out=0,in=-90] (1,3.5) to[out=90,in=0] (0,4);
        
        \draw (0,1) to[out=0,in=180] (\widthscale,1);
        \draw (0,2) to[out=0,in=180] (\widthscale,4);

        \draw (\widthscale,1) to[out=0,in=-90] (\widthscale+2,2.5) to[out=90,in=0] (\widthscale,4);
        
        \draw (2*\widthscale,1) to[out=180,in=-90] (2*\widthscale-1,1.5) to[out=90,in=180] (2*\widthscale,2);
        \draw (2*\widthscale,3) to[out=180,in=-90] (2*\widthscale-1,3.5) to[out=90,in=180] (2*\widthscale,4);

        \draw (\widthscale,2.5) circle (0.7);
        \end{tikzpicture}
     \
     $= $
     \ 
     $a \cdot$
        \begin{tikzpicture}[scale=0.4, baseline=(base)]
        \def\widthscale{3.5};

        \coordinate (base) at (0,2);
        \draw[gray,line width = 1](0,0.5)--(0,4.5);
        \draw[gray,line width = 1](\widthscale,0.5)--(\widthscale,4.5);

        \foreach \y in {1,2, 3,4}{
            \foreach \x in {0,1}{
                \draw[fill=black]  (\widthscale*\x,\y) circle [radius=0.15];
                \draw (\widthscale*\x,\y) node[right] {};
            }
            
        }

        \draw (0,3) to[out=0,in=-90] (1,3.5) to[out=90,in=0] (0,4);
        \draw (0,1) to[out=0,in=-90] (1,1.5) to[out=90,in=0] (0,2);
        
        \draw (\widthscale,1) to[out=180,in=-90] (\widthscale-1,1.5) to[out=90,in=180] (\widthscale,2);
        \draw (\widthscale,3) to[out=180,in=-90] (\widthscale-1,3.5) to[out=90,in=180] (\widthscale,4);

        \end{tikzpicture}

    \caption{Multiplication in the Temperley--Lieb algebra $\TL_4$.}\label{fig: TLalgsInIntro}
\end{figure}
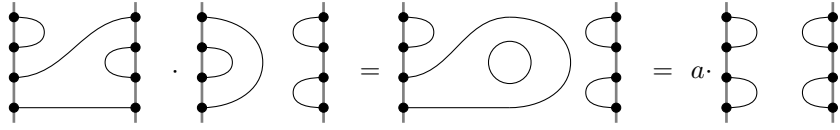

These algebras, and some of their modules, already implicitly appeared in our description of the complex of planar loops: we will define them formally in \cref{subsection: TLcat} as endomorphism objects in the Temperley--Lieb category $\TL$. Our formal definition of the differential graded algebra of planar loops $\CPL(2n)$ uses $\TL_{2n}$ explicitly.

The relationship between $\Tor_*^{\TL_{2n}}(R,R)$ and the homology of $\CPL(2n)$ is as follows.

\begin{theorem} \label{theorem: shift}
There are $R$-module isomorphisms
$$\Tor_i^{\TL_{2n}(R,a)}(R,R) \cong \begin{cases}
R & i=0\\
0 & 0 < i < 2n-1\\
H_{i-2n+1}(\CPL(2n;R,a)) &  2n-1 \leq i.
\end{cases}$$
\end{theorem}

Combined with \cref{cor:MainCalculations} we obtain the following example calculations:
\begin{enumerate}
\item For $i>0$, $\Tor_i^{\TL_{2n}(R,a)}(R,R)$ is $a$-torsion.
\item[(b)+(c)] If $n! \in R^\times$ then
$$\Tor_i^{\TL_{2n}(R,0)}(R,R) \cong \begin{cases}
R & i\geq 0, i \equiv -1,0 \mod 2n\\
0 & \text{else}.
\end{cases}$$
\setcounter{enumi}{3}
\item If $a \in R$ is not a zerodivisor and $n! \in (R/a)^\times$ then
$$\Tor_i^{\TL_{2n}(R,a)}(R,R) \cong \begin{cases}
R & i=0\\
R/a & i\geq 0, i \equiv -1 \mod 2n\\
0 & \text{else}.
\end{cases}$$
\item $\Tor_*^{\TL_{4}(\mathbb{Z},0)}(\mathbb{Z},\mathbb{Z})$ only has 2-torsion, which is all simple. The Poincar{\'e} series for its ranks is $\tfrac{1+ t^3}{1-t^4}$. The Poincar{\'e} series for the $\mathbb{Z}/2$-dimension of the simple 2-torsion is
$$\frac{t^5}{(1-t - t^3 ) (1 - t^4)} = t^5 + t^6 + t^7 + 2 t^8 + 4 t^9 + 5 t^{10} + \cdots.$$
\end{enumerate}

To prove \cref{thm:Main} we build a resolution of the $L(2n;R,0)$-module $R$ satisfying strong algebraic properties stated in our main technical theorem, \cref{theorem:derived-outermost-cup-complex-resolution}. This theorem relies on two auxiliary complexes, which we show are acyclic: the complex of innermost cups (and its subcomplex of submaximal cups), and the complex of outermost cups. We build derived copies, which allow us to prove \cref{theorem:derived-outermost-cup-complex-resolution} (and, along the way, \cref{theorem: shift}). These constructions are deferred to the second half of the paper.

\subsection{Connection to previous works}

Temperley--Lieb algebras $\TL_n(R,a)$ were introduced in mathematical physics \cite{TemperleyLieb} and later rediscovered by Jones, who used them to define the Jones polynomial \cite{Jones}. The diagrammatic interpretation used throughout this work is due to Kauffman \cite{KauffmanState, KauffmanInvariant}. Temperley--Lieb algebras are also interesting in representation theory, for instance as examples of cellular algebras studied by Graham--Lehrer and Ridout--Saint-Aubin \cite{GrahamLehrer, RidoutStAubin}. While we do not explicitly use the theory of cellular algebras, we note that the complex of outermost cups is a simplicial object constructed from the cell modules associated to a Temperley--Lieb algebra.

As mentioned above, our calculations contrast sharply with the situation when the number of strands is odd: here the last-named author showed that the homology vanishes \cite{Sroka}.

The study of the homology of Temperley--Lieb algebras, $\Tor_*^{\TL_{n}(R,a)}(R,R)$, was initiated in joint work of the first-named author and Hepworth \cite{BoydHepworthStability}. The main results of \cite{BoydHepworthStability} include vanishing of the homology in all positive degrees if the parameter $a$ is a unit in $R$ \cite[Theorem A]{BoydHepworthStability}, as well as a slope-1 vanishing line without restriction on the parameter \cite[Theorem B]{BoydHepworthStability} (which used an observation of the third-named author \cite{RandalWilliams}). Combining our \cref{theorem: shift} with Part (a) of \cref{cor:MainCalculations}, we recover these statements for Temperley--Lieb algebras on an even number of strands. Using the model (\cref{thm:Main}) one can then identify the first two nontrivial homology groups, extending \cite[Theorem C]{BoydHepworthStability}: $\Tor_{2n-1}^{\TL_{2n}(R,a)}(R,R) = R/a$ and $\Tor_{2n}^{\TL_{2n}(R,a)}(R,R) = R_a = \ker(R \xrightarrow{a} R)$.

The initial investigations by the first-named author and Hepworth, \cite{HepworthIH} and \cite{BoydHepworthStability}, led to research exploring the homology of various other augmented algebras: In joint work with Patzt they also investigated Brauer and Partition algebras \cite{BoydHepworthPatztBrauer, BoydHepworthPatztPartition}; Hepworth's result on Iwahori--Hecke algebras \cite{HepworthIH} were extended in the PhD thesis of the last-named author \cite{SrokaThesis} and in work of Moselle \cite{Moselle}; the second-named author investigated Jones annular, rook-Brauer and extended the results for partition algebras \cite{BoydeOne, BoydeTwo}; more recently, Graves explored generalized rook-Brauer algebras \cite{GravesOne}; Fisher--Graves investigated blob algebras, walled Brauer algebras and Tanabe algebras \cite{FisherGravesOne, FisherGravesTwo}; and Cranch--Graves studied coloured partition algebras \cite{CranchGravesOne}.

We expect that many of the new ideas contained in this work have analogues in these contexts.

\subsection{Remaining questions} Many questions about $\CPL(2n)$ and $\TL_{2n}$ remain open. Here we suggest three. First, the dga $\CPL(2n)$ admits two involutions given by reflecting diagrams left-to-right and top-to-bottom. The effect of these involutions on the model is not obvious from our methods.

\begin{question} \label{question: involution}  Describe the effect of these involutions on the model.
\end{question}

Caution: a priori, the endomorphisms of the model corresponding to these involutions need only be involutions up to (coherent) chain homotopy.

\begin{question} \label{question: diagrammatics} Interpret the generators $x_i$ of the model diagrammatically.
\end{question}

In forthcoming work of the second-named author, these first two questions will be answered in the special case $2n=4$, $a=0$, by direct calculation.

Lastly, from our results, using a universal coefficient spectral sequence, one can sometimes calculate $\Ext^*_{TL_{2n}}(R,R)$ as an $R$-module. The composition product is however not visible by our methods.

\begin{question} \label{question: Ext} Determine $\Ext^*_{TL_{2n}}(R,R)$ as an $R$-algebra.
\end{question}

For example, when $a \in R$ is not a zerodivisor and $n! \in (R/a)^\times$, the universal coefficient spectral sequence shows that
$$\Ext^i_{TL_{2n}}(R,R) \cong \begin{cases}
    R & i=0\\
    R/a &i>0, i \equiv 0 \mod 2n\\
    0 & \text{else}.
\end{cases}$$
Is this $R[x]/(ax)$ as an $R$-algebra?

\subsection{Overview}

In \cref{section: TL category and algebra of planar loops}, we introduce the Temperley--Lieb category, give a second definition of the dga of planar loops $\CPL(2n)$ and observe some extra structure on it. In \cref{section: small models},  we prove \cref{thm:Main}, which provides a small model for the dga $\CPL(2n)$ assuming our main technical result (\cref{theorem:derived-outermost-cup-complex-resolution}). In the subsequent \cref{sec:Calc}, we show that this readily leads to homology calculations, establishing in particular \cref{cor:MainCalculations}. The remaining sections are then devoted to the proof of \cref{theorem:derived-outermost-cup-complex-resolution} as well as to explaining the relation between $\CPL(2n)$ and the homology of Temperley--Lieb algebras in \cref{theorem: shift}.

\subsection{Acknowledgments}
GB would like to thank Richard Hepworth for several insightful conversations, and especially for suggesting that one could try to push calculations of Temperley--Lieb homology further by thinking about the diagrammatics of bar constructions.

RB was supported by EPSRC Fellowship EP/V043323/1 and EP/V043323/2. 

GB was supported by the ERC (grant no.~950048), and by the INI, Cambridge, during the programme ``Equivariant homotopy theory in context'' (EPSRC grant EP/Z000580/1) where work on this paper was undertaken.

ORW was partially supported by the ERC under the European Union’s Horizon 2020 research and innovation programme (grant agreement No.~756444) and by the Danish National Research Foundation through the Copenhagen Centre for Geometry and Topology (DNRF151).

RJS was supported by NSERC Discovery Grant A4000 as well as by the German Research Foundation through SFB 1442 -- 427320536, Geometry: Deformations and Rigidity, and under Germany's Excellence Strategy EXC 2044 -- 390685587, Mathematics M\"unster: Dynamics--Geometry--Structure.

Finally, we would like to thank the universities of Cambridge, Glasgow, Münster and Utrecht for support and hospitality during research visits.

\section{The Temperley--Lieb category and the dga of planar loops}
\label{section: TL category and algebra of planar loops}

In this section, we introduce the Temperley--Lieb category, give a second formal definition of the dga of planar loops $\CPL(2n)=\CPL(2n;R,a)$,  record some extra structure on it, and state our main technical result, \cref{theorem:derived-outermost-cup-complex-resolution}.

\subsection{The Temperley--Lieb category and Temperley--Lieb algebras}
\label{subsection: TLcat}

Let $m$ and $n$ be two non-negative integers. A \emph{Temperley--Lieb} $(m,n)$\emph{-diagram} is a \emph{planar diagram} consisting of two vertical bars with $m$ nodes on the left bar and $n$ nodes on the right bar, such that $m+n$ is even. Nodes are labelled from bottom to top. There are $(m+n)/2$ arcs connecting the $m+n$ nodes in pairs, which do not intersect, and which lie between the left and right bars. Moreover, the arcs are nonoriented and considered up to isotopy (relative to the end points). We refer the reader to \cite[Section 5.7.4]{KasselTuraev} for additional details. An example is shown in \cref{fig: TL53}.  We denote by $\TL(m,n)$ the free $R$-module on all Temperley--Lieb $(m,n)$-diagrams. Since we will later specify to $m$ and $n$ both being even, all of our figures are in this case.

\begin{figure}[ht]
    \begin{tikzpicture}[scale=0.4, baseline=(base)]
        \coordinate (base) at (0,3.25);
        \draw[gray,line width = 1](0,0.5)--(0,6.5);
        \draw[gray,line width = 1](6,0.5)--(6,6.5);
        \foreach \x in {1,2, 3,4,5,6}{
            \draw[fill=black] (0,\x) circle [radius=0.15] ;
            \draw (0,\x) node[left] {$\scriptstyle{\x}$};
        } 
        \foreach \y in {1,2, 3,4}{
            \draw[fill=black]  (6,\y+1) circle [radius=0.15];
            \draw (6,\y+1) node[right] {$\scriptstyle{\y}$};
        } 
        \draw (0,1) to[out=0,in=-90] (1,1.5) to[out=90,in=0] (0,2);
        \draw (0,4) to[out=0,in=-90] (1,4.5) to[out=90,in=0] (0,5);
        
        \draw (6,3) to[out=180,in=-90] (5,3.5) to[out=90,in=180] (6,4);
        \draw (0,3) .. controls (2,3) and (4,1.5) .. (6,2);
        \draw (0,6) .. controls (2,6) and (4,4.5) .. (6,5);
    \end{tikzpicture}
    \caption{A planar diagram in $\TL(6,4)$. In future diagrams we will sometimes suppress the node labels.}\label{fig: TL53}
    \end{figure}
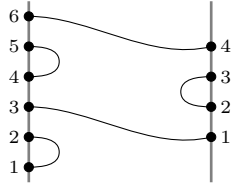

If $m=n=0$, then there is a unique Temperley--Lieb $(0,0)$-diagram, the empty diagram. It is denoted $\emptyset$, so $\TL(0,0)=R\{\emptyset\}$.

\begin{definition}\label{def: TL category}
    Let $R$ be unital commutative ring and let $a \in R$ be a parameter. The \emph{Temperley--Lieb category} $\TL=\TL(R,a)$ is the $R$-linear category with one object $n$ for every number $n \geq 0$. The morphism $R$-modules are given by $\TL(m,n)$ (note these are zero when $m+n$ is odd). Given a $(l,m)$-diagram $D$ and $(m,n)$-diagram $D'$, the composition law is
    \begin{eqnarray*}
        \TL(l, m) \otimes_R \TL(m,n) &\to& \TL(l,n)\\
        D \otimes D' &\mapsto& a^k [D \circ D'] 
    \end{eqnarray*} 
    where $[D \circ D']$ is the $(l,n)$-diagram obtained by identifying the right hand bar and nodes of $D$ with the left hand bar and nodes of $D'$, and joining the arcs together. Any closed loops created in this process are replaced by a factor of $a$, so in the equation above $k$ is the number of closed loops created; see \cref{fig:composition-in-tl}.
\end{definition}
\begin{figure}

\begin{tikzpicture}[scale=0.4, baseline=(base)]
        \coordinate (base) at (0,3.25);
        \draw[gray,line width = 1](0,0.5)--(0,6.5);
        \draw[gray,line width = 1](6,0.5)--(6,6.5);
        \foreach \x in {1,2, 3,4,5,6}{
            \draw[fill=black]  (6,\x) circle [radius=0.15] ;
        } 
        \draw[fill=black] (0,3) circle [radius=0.15] ;
        \draw[fill=black] (0,4) circle [radius=0.15] ;
    
        \draw (6,4) to[out=180,in=-90] (5,4.5) to[out=90,in=180] (6,5);
        \draw (6,2) to[out=180,in=-90] (5,2.5) to[out=90,in=180] (6,3);
        \draw (0,3) .. controls (2,3) and (4,1) .. (6,1);
         \draw (0,4) .. controls (2,4) and (4,6) .. (6,6);
    \end{tikzpicture}
   \
    $\cdot$
    \
    \begin{tikzpicture}[scale=0.4, baseline=(base)]
        \coordinate (base) at (0,3.25);
        \draw[gray,line width = 1](6,0.5)--(6,6.5);
        \draw[gray,line width = 1](12,0.5)--(12,6.5);

        \foreach \x in {1,2, 3,4,5,6}{
            \draw[fill=black]  (6,\x) circle [radius=0.15] ;
        } 
        
        \foreach \x in {1,2,3,4}{
            \draw[fill=black]  (12,\x+1) circle [radius=0.15];
        } 
    
        \draw (6,2) to[out=0,in=-90] (7,2.5) to[out=90,in=0] (6,3);
       
        \draw (12,3) to[out=180,in=-90] (11,3.5) to[out=90,in=180] (12,4);
        \draw (6,1) to[out=0,in=-90] (8,2.5) to[out=90,in=0] (6,4);
        \draw (6,5) .. controls (9,5) and (9,2) .. (12,2);
        \draw (6,6) .. controls (9,6) and (9,5) .. (12,5);
    \end{tikzpicture}
    \ \
    $=$
    \ \
    \begin{tikzpicture}[scale=0.4, baseline=(base)]
        \coordinate (base) at (0,3.25);
        \draw[gray,line width = 1](0,1.5)--(0,5.5);
        \draw[gray,line width = 1](6,1.5)--(6,5.5);
         \draw[fill=black] (0,3) circle [radius=0.15];
        \draw[fill=black] (0,4) circle [radius=0.15] ;
        \foreach \x in {1,2, 3,4}{
            \draw[fill=black] (6,\x+1) circle [radius=0.15];
        } 
        \draw (0,3) .. controls (3,3.) and (3,1.5) .. (6,2);
        \draw (0,4) .. controls (3,4) and (3,5.5) .. (6,5);
        \draw (6,3) to[out=180,in=-90] (5,3.5) to[out=90,in=180] (6,4);
        \draw (3,3.5) circle (0.7);
    \end{tikzpicture}
    \ \ 
    $=$
    \ \ 
    $a\cdot$
    \begin{tikzpicture}[scale=0.4, baseline=(base)]
    \def\widthscale{5}
        \coordinate (base) at (0,3.25);
        \draw[gray,line width = 1](0,1.5)--(0,5.5);
        \draw[gray,line width = 1](\widthscale,1.5)--(\widthscale,5.5);
            \draw[fill=black] (0,3) circle [radius=0.15];
            \draw[fill=black] (0,4) circle [radius=0.15] ;
        \foreach \x in {1,2, 3,4}{
            \draw[fill=black] [radius=0.15] (\widthscale,\x+1) circle [radius=0.15];
        } 
        \draw (0,3) to[out=0,in=180] (\widthscale,2);
        \draw (0,4) to[out=0,in=180] (\widthscale,5);
        \draw (\widthscale,3) to[out=180,in=-90] (\widthscale-1,3.5) to[out=90,in=180] (\widthscale,4);
    \end{tikzpicture}
    \caption{Composition of morphisms $\TL(2, 6) \otimes_R \TL(6,4) \to \TL(2,4)$.}
    \label{fig:composition-in-tl}
\end{figure}
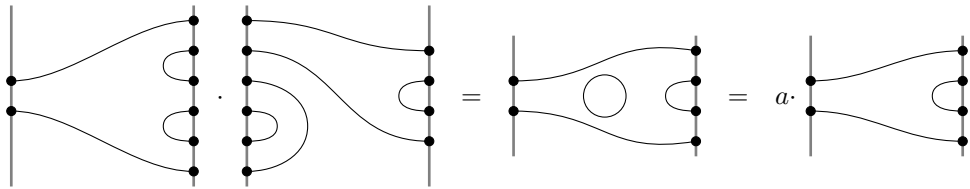

The composition law equips the endomorphism modules $\TL(n,n)$ with the structure of an $R$-algebra. These are the Temperley--Lieb algebras that we are interested in.

\begin{definition}
    The \emph{Temperley--Lieb algebra on $n$ strands with parameter $a \in R$} is $
    \TL_{n}(R,a)=\TL_n = \TL(n,n).
    $
\end{definition}

Similarly, the composition law in $\TL$ equips every morphism module $\TL(m,n)$ with the structure of a $(\TL_{m}, \TL_{n})$-bimodule. These modules will form the basic buildings blocks of our analysis of the family of Temperley--Lieb algebras on an even number of strands, $\{\TL_{2n}\}_{n \in \mathbb{N}}$.

We now introduce some named diagrams which we will frequently use throughout this paper. By a \emph{left (right) cup} in a diagram we mean a connection between two nodes on the left (respectively right).
\begin{definition}\label{defn: diagrams Lk Rk and Lmax}
For $1 \leq k \leq 2i-1$ let $L_k \in \TL(2i,2i-2)$ be the diagram which has a single left cup connecting the nodes $k$ and $k+1$ (see \cref{fig: R2}).
\begin{figure}[ht]
$L_2=$ 
    \begin{tikzpicture}[scale=0.4, baseline=(base)]
        \coordinate (base) at (0,3.25);
        \draw[gray,line width = 1](0,0.5)--(0,6.5);
        \draw[gray,line width = 1](6,0.5)--(6,6.5);
        \foreach \x in {1,2, 3,4,5,6}{
            \draw[fill=black] (0,\x) circle [radius=0.15] ;
            \draw (0,\x) node[left] {$\scriptstyle{\x}$};
        } 
        \foreach \y in {1,2, 3,4}{
            \draw[fill=black]  (6,\y+1) circle [radius=0.15];
            \draw (6,\y+1) node[right] {$\scriptstyle{\y}$};
        } 
        \draw (0,2) to[out=0,in=-90] (1,2.5) to[out=90,in=0] (0,3);
        \draw (6,2) to[out=180,in=0] (0,1);
         \foreach \x in {3,4,5}{
        \draw (6,\x) to[out=180,in=0]  (0,\x +1);
        }
    \end{tikzpicture}
$,$
\quad
$\Lmax =$
    \begin{tikzpicture}[scale=0.4, baseline=(base)]
        \coordinate (base) at (0,3.25);
        \draw[gray,line width = 1](0,0.5)--(0,6.5);
        
        \foreach \x in {1,2, 3,4,5,6}{
            \draw[fill=black] (0,\x) circle [radius=0.15] ;
            \draw (0,\x) node[left] {$\scriptstyle{\x}$};
        } 
        \draw (0,1) to[out=0,in=-90] (1,1.5) to[out=90,in=0] (0,2);
        \draw (0,3) to[out=0,in=-90] (1,3.5) to[out=90,in=0] (0,4);
        \draw (0,5) to[out=0,in=-90] (1,5.5) to[out=90,in=0] (0,6);
    \end{tikzpicture}
    \caption{The diagrams $L_2$ in $\TL(6,4)$ and $\Lmax \in \TL(6,0)$.}\label{fig: R2}
\end{figure}
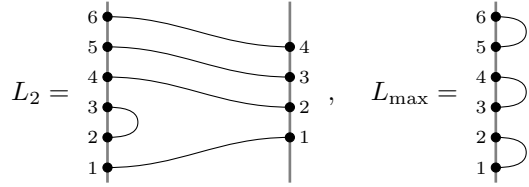
Let $R_k$ be the `mirror image' of $L_k$: the diagram in $\TL(2i-2,2i)$ having a single right cup connecting $k$ and $k+1$. Let $\Lmax \in \TL(2i,0)$ be the diagram with a connection from $2k+1$ to $2k+2$ for $0 \leq k < i$.    
\end{definition}

\subsection{The algebra of planar loops}

In this section, we give a second, more formal, definition of the differential graded algebra of planar loops $\CPL(2n, R, a)$, which we described pictorially in the introduction, in terms of an augmented two-sided bar construction arising from the Temperley--Lieb category $\TL$.

\begin{definition} \label{definition: dga of planar loops} Let $n \geq 0$. 
The \emph{dga of planar loops} $\CPL(2n) = \CPL(2n; R, a)$ is defined as follows. If $n = 0$ we set $\CPL(2n) = \TL(0,0) = R\{\emptyset\}$, concentrated in degree 0. For $n > 0$, we define $\CPL(2n)$ as the cone of the map $$\Bar(\TL(0,2n), \TL_{2n}, \TL(2n,0))  \to \TL(0,2n) \otimes_{\TL_{2n}} \TL(2n,0) \to \TL(0,0)$$ which first uses the canonical map from the bar construction to its zeroth homology and then the composition law in $\TL$. 
Explicitly, in degree $q > 0$ the underlying $R$-module is
$$\CPL(2n)_q = \TL(0,2n) \otimes_R \TL(2n,2n)^{\otimes q-1} \otimes_R \TL(2n,0)$$
and in degree $q = 0$ it is given by $\CPL(2n)_0 = \TL(0, 0)$.
The differential of $\CPL(2n)$ is induced by the composition maps in $\TL$. Writing bars to represent the tensor products $\otimes_R$, it is given by an alternating sum of face maps
$$d_q(D_0 | \dots | D_{q}) = \sum_{i=0}^{q-1} (-1)^{i} (D_0 | \dots | D_i D_{i+1} | \dots | D_{q}).$$
This is a differential graded algebra, with multiplication induced by the composition map $\TL(2n,0) \otimes_R \TL(0,2n) \to \TL_{2n}$, and unit $\emptyset \in \TL(0,0)=L(2n)_0$; see \cref{fig:whatIsB}.
\end{definition}

In each degree, $\CPL(2n)$ comes with a canonical basis, given by the tensors, where each tensor factor is taken to have its standard basis of diagrams. In degree $q$, such a tensor is precisely a system of planar loops of height $2n$ pinned by $q$ bars, as defined in the introduction. 

The second part of the construction produces dg-modules over this dga.

\begin{definition}\label{definition: dg modules general}
Suppose we are given the data $(X,\phi, \psi)$ of a left $\TL_{2n}$-module $X$, an $R$-module map $\phi \colon \TL(0,2n) \otimes_{\TL_{2n}} X \to Y$, and a $\TL_{2n}$-module map $\psi \colon \TL(2n,0) \otimes_R Y \to X$ such that the diagrams
    \begin{equation*}
\begin{tikzcd}
\TL(0,2n) \otimes_R \TL(2n,0) \otimes_R Y \dar{\TL(0,2n) \otimes \psi} \arrow[r] & R \otimes_R Y \arrow[d, equals] \\
\TL(0,2n) \otimes_R X \rar{\phi}  & Y
\end{tikzcd}
\end{equation*}
and
    \begin{equation*}
\begin{tikzcd}
\TL(2n,0) \otimes_R \TL(0,2n) \otimes_{\TL_{2n}} X \arrow[r] \dar{\TL(2n,0) \otimes \phi} & \TL_{2n} \otimes_{\TL_{2n}} X \arrow[d, equals]\\
\TL(2n,0) \otimes_R Y \rar{\psi} & X
\end{tikzcd}
\end{equation*}
commute, where unlabelled arrows are induced by composition in $\TL$.

We may associate to this the chain complex of $R$-modules
$${D}(X, \phi, \psi) := \mathrm{Cone}(\Bar(\TL(0,2n), \TL_{2n}, X) \to \TL(0,2n) \otimes_{\TL_{2n}} X \overset{\phi}\to Y),$$
where the first map is the canonical map from the bar construction to its zeroth homology. Explicitly, in degree $q>0$ the underlying $R$-module is
$${D}(X, \phi, \psi)_q = \TL(0,2n) \otimes_R \TL_{2n}^{\otimes q-1} \otimes_R X$$
and in degree $q=0$ it is given by ${D}(X, \phi, \psi)_0 = Y$. The differential is induced by multiplication in $\TL_{2n}$, its right and left actions on $\TL(0,2n)$ and $X$, and $\phi$.

The chain complex ${D}(X, \phi, \psi)$ has the canonical structure of a left $L(2n)$-module: the $\CPL(2n)$ action on positive-degree classes is induced by the multiplication map $\TL(2n,0) \otimes_R \TL(0,2n) \to \TL_{2n}$ in \cref{definition: dga of planar loops}, and its action on degree zero classes by $\psi$. This construction defines a functor from the category of such triples $(X,\phi,\psi)$ to the category of left $L(2n)$-modules.
\end{definition}

An example of this construction is the following.

\begin{definition} \label{def: dg modules L02n2i} Let $n,i \geq 0$. The module $\CPL(0,2n,2i)=\CPL(0,2n,2i;R,a)$ is the left $\CPL(2n)=\CPL(0,2n,0)$-module obtained from the construction of \cref{definition: dg modules general} with $X = \TL(2n,2i)$, $Y = \TL(0,2i)$, and the maps $$\phi \colon \TL(0,2n) \otimes_{\TL_{2n}} \TL(2n,2i) \to \TL(0,2i)$$
and
$$\psi \colon \TL(2n,0) \otimes_{R} \TL(0,2i) \to \TL(2n,2i)$$
given by composition in $\TL$. Commutativity of the required diagrams follows from associativity of composition in $\TL$.
\end{definition}

Note that $\CPL(0,2n,2i)$ has a canonical $R$-basis similar to that given for $\CPL(2n)$.

\subsection{The element $\Phi$}\label{sec:Phi}

In this section we introduce the element $\Phi \in \CPL(2n)_1$. 

    \begin{figure}[ht]
    $\Phi_l =$
        \begin{tikzpicture}[scale=0.4, baseline=(base)]
        \coordinate (base) at (0,3.75);
        
        \draw[gray,line width = 1](6,0.5)--(6,8.5);

        \foreach \y in {1,2, 3,4,5,6,7,8}{
            \draw[fill=black]  (6,\y) circle [radius=0.15];
            \draw (6,\y) node[right] {$\scriptstyle{\y}$};
        }

                 \foreach \x in {1,5}{
        \draw (6,\x) to[out=180,in=-90] (4,\x + 1.5) to[out=90,in=180] (6,\x + 3);
        \draw (6,\x+1) to[out=180,in=-90] (5,\x + 1.5) to[out=90,in=180] (6,\x + 2);
        }

    \end{tikzpicture}
    ,
    \begin{tikzpicture}[scale=0.4, baseline=(base)]
        \coordinate (base) at (0,3.75);
        \draw[gray,line width = 1](0,0.5)--(0,8.5);
        \draw[gray,line width = 1](6,0.5)--(6,8.5);

        \foreach \x in {1,2, 3,4,5,6,7,8}{
            \draw[fill=black]  (0,\x) circle [radius=0.15];
            \draw (0,\x) node[left] {$\scriptstyle{\x}$};
        } 

        \foreach \x in {1,2}{
            \draw[fill=black]  (6,\x+3) circle [radius=0.15];
            \draw (6,\x+3) node[right] {$\scriptstyle{\x}$};
        } 

        \foreach \x in {1,2}{
        \draw (0,\x+6) to[out=0,in=180] (6,\x + 3);
        }

         \foreach \x in {1,4}{
        \draw (0,\x) to[out=0,in=-90] (1,\x + 0.5) to[out=90,in=0] (0,\x + 1);
        }
        
        \draw (0,3) to[out=0,in=-90] (2,3 + 1.5) to[out=90,in=0] (0,3 + 3);
    \end{tikzpicture}
    $= \Phi_r'$
    ,
    \quad
    \quad
        $\Phi_l =$
        \begin{tikzpicture}[scale=0.4, baseline=(base)]
        \coordinate (base) at (0,4.75);
        
        \draw[gray,line width = 1](6,0.5)--(6,10.5);

        \foreach \y in {1,2, 3,4,5,6,7,8,9,10}{
            \draw[fill=black]  (6,\y) circle [radius=0.15];
            \draw (6,\y) node[right] {$\scriptstyle{\y}$};
        } 
        
         \foreach \x in {1,4, 8}{
        \draw (6,\x) to[out=180,in=-90] (5,\x + 0.5) to[out=90,in=180] (6,\x + 1);
        }
         \foreach \x in {3,7}{
        \draw (6,\x) to[out=180,in=-90] (4,\x  + 1.5) to[out=90,in=180] (6,\x + 3);
        }
    \end{tikzpicture}
    ,
    \begin{tikzpicture}[scale=0.4, baseline=(base)]
        \coordinate (base) at (0,4.75);
        \draw[gray,line width = 1](0,0.5)--(0,10.5);

        \foreach \x in {1,2, 3,4,5,6,7,8, 9, 10}{
            \draw[fill=black]  (0,\x) circle [radius=0.15];
            \draw (0,\x) node[left] {$\scriptstyle{\x}$};
        }

         \foreach \x in {2,6}{
        \draw (0,\x) to[out=0,in=-90] (1,\x + 0.5) to[out=90,in=0] (0,\x + 1);
        }

        \draw[dotted, thick] (0,9) to[out=0,in=-90] (1,9 + 0.5) to[out=90,in=0] (0,9 + 1);
        
         \foreach \x in {1,5}{
        \draw (0,\x) to[out=0,in=-90] (2,\x  + 1.5) to[out=90,in=0] (0,\x + 3);
        }
    \end{tikzpicture}
    $= \Phi_r$
    \caption{The diagrams $\Phi_l$ and $\Phi_r'$ for $2n=8$ (left) and $\Phi_l$ and $\Phi_r$ for $2n=10$ (right). In the latter case, the cup formed by the multiplication $\Phi_r = \Phi_r' \cdot L_1$ is dotted for emphasis.}\label{fig: bubbleWriting1}
\end{figure}
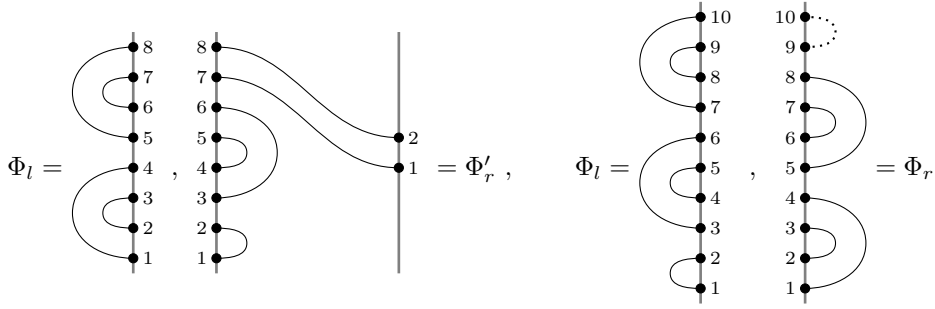
    
\begin{definition}\label{defn:phir and phil}
    Assume that $n \geq 1$. We take $\Phi_l \in \TL(0,2n)$ and $\Phi_r \in \TL(2n,0)$ to be as shown in \cref{fig: bubbleWriting1}. Precisely, let $\Phi_l \in \TL(2n,0)$ be the diagram which is defined as follows. If $2n \equiv 0 \ \mathrm{mod}(4)$, then for each $k \geq 1$ congruent to $1$ mod $4$ with $k+3 \leq 2n$, we connect $k$ to $k+3$ and $k+1$ to $k+2$ (producing a nested pair of cups). If $2n \equiv 2 \ \mathrm{mod}(4)$, then we `add a cup at the start and shift up', first connecting $1$ to $2$, then, for each $k \geq 3$ congruent to $3$ mod $4$ with $k+3 \leq 2n$, we connect $k$ to $k+3$ and $k+1$ to $k+2$ (again producing a nested pair of cups).

We now define the complementary diagram $\Phi_r \in \TL(2n,0)$. It will be useful to first define a `lift' $\Phi_r'$ to $\TL(2n,2)$ satisfying $\Phi_r' \cdot L_1 = \Phi_r$. To form $\Phi_r'$, we connect $2n$ on the left to 2 on the right and $2n-1$ on the left to $1$ on the right, then add consecutive pairs of nested cups, starting with connections $2n-5$ to $2n-2$ and $2n-4$ to $2n-3$, until there are fewer than four nodes left. If $2n \equiv 2 \ \mathrm{mod}(4)$ then we are done. If $2n \equiv 0 \ \mathrm{mod}(4)$ then we add a final cup from $1$ to $2$.

We then define $\Phi_r := \Phi_r' \cdot L_1$.
\end{definition} 

\begin{definition}\label{definition:Phi}
    We write $\Phi := \Phi_l \otimes \Phi_r \in \TL(0,2n) \otimes_{R} \TL(2n,0) = \CPL(2n)_1$.
\end{definition}

We record some elementary properties of these elements which we will need later.

\begin{lemma} \label{lemma: phiL - phiR'} Let $n \geq 1$. Then
$$\Phi_l \cdot \Phi_r' = R_1 \in \TL(0,2),$$
hence also
\[
\pushQED{\qed} 
\Phi_l \cdot \Phi_r = a \cdot \emptyset \in \TL(0,0).\qedhere
\popQED
\]    
\end{lemma}

As a consequence we have the following.

\begin{lemma} \label{lemma: wind sock}
    Let $n>0$ and $i>0$. For any diagram $z \in \TL(0,2i)$, there exists a diagram $x \in \TL(0,2n)$ and a diagram $y \in \TL(2n,2i)$ such that $xy = z$.
\end{lemma}
\begin{proof} Choose any cup in $z$ which is not contained in any other cup, and cut it open to form $z' \in \TL(2,2i)$ such that $R_1 \cdot z' = z$. Set $x = \Phi_l$ and $y = \Phi_r' \cdot z'$. Then, using \cref{lemma: phiL - phiR'}, we have
\[
xy = \Phi_l \cdot \Phi_r' \cdot z' = R_1 \cdot z' = z.\qedhere
\]    
\end{proof}

\subsection{Extra structure on the algebra: weight grading and hook maps}

We now discuss extra structure on the algebra of planar loops that will be crucial in proving our main results. First, in certain circumstances, the dga $\CPL(2n)$ (and its modules $\CPL(0,2n,2i)$) admit an extra grading, which we call \emph{weight}, given informally by counting loops (look at \cref{fig:whatIsB}).

Recall that $\CPL(2n)$ has a canonical $R$-basis given by \emph{systems of pinned planar loops}. Such a system $D$ is in particular an isotopy class of some number $\ell$ of loops (disjoint circles) in the plane. We will say that $D$ \emph{forms $\ell$ loops}. 

\begin{definition}[Weight] \label{def: weight} Consider the pair $(R,a)$. Suppose that the commutative ring $R$ is non-negatively graded (in the sense that the multiplication satisfies $R_j \otimes R_k \mapsto R_{j+k}$), and that $a \in R$ is homogeneous of degree 1. For clarity, we will call the grading of a homogenous element $r\in R$ the \emph{rank} of $r$.

In this situation, we define the \emph{weight grading} as follows. If $r \in R$ is homogeneous of rank $j$ and $D \in \CPL(2n)$ is a system of planar loops forming $\ell$ loops, then $r \cdot D$ has weight $j+\ell$. Extend this linearly to define the weight grading on $\CPL(2n)$.

To define the grading on $\CPL(0,2n,2i)$ for $i>0$, recall from \cref{defn: diagrams Lk Rk and Lmax} that $\Lmax \in \TL(2i,0)$ is the diagram with $i$ cups between adjacent pairs of nodes. For $i \geq 1$, the \emph{weight} of an element $r \cdot D$, where $r \in R$, and $D$ is a basis element $D = (D_0 | \dots | D_q) \in \CPL(0,2n,2i)$, is the weight of $r \cdot (D_0 | \dots | (D_q \Lmax)) \in \CPL(0,2n,0)$. (Our choice of $\Lmax$ amongst all diagrams in $\TL(2i,0)$ is important only once, in \cref{lem: weight on outermost}.)
\end{definition}

We will use that this extra structure is present in the universal case $(\Z[a], a)$, where $a$ has rank 1, and $(\mathbb{Z},0)$, where we think of $\mathbb{Z}$ as concentrated in rank zero and $a=0$ in rank 1 (so that $(\Z, 0)$ is the quotient of $(\Z[a],a)$ by $a$).

We verify that this indeed defines a grading.

\begin{lemma} \label{lemma: grading indeed} The differential $d$ of $\CPL(0,2n,2i)$ preserves weight.
\end{lemma}

\begin{proof}
    Recall that $d$ was defined as an alternating sum of face maps, and that the $k$-th face map can be thought of geometrically as removing the $k$-th bar and replacing unpinned loops with factors of $a$. Since $a$ is homogeneous of rank 1, this leaves the weight unchanged.
\end{proof}

When we wish to emphasise the weight grading, we will write it second, so that $\CPL(0,2n,2i)_{q,w}$ is spanned by elements of homological degree $q$ and weight $w$. In light of \cref{lemma: grading indeed}, the differential has bidegree $(-1,0)$. This makes $\CPL(2n)$ a differential bigraded algebra, and $\CPL(0,2n,2i)$ a differential bigraded module over it; concretely, we have a multiplication map
$$\CPL(2n)_{q_1,w_1} \otimes \CPL(2n)_{q_2,w_2} \to \CPL(2n)_{q_1 + q_2,w_1+ w_2},$$
and an action
$$\CPL(2n)_{q_1,w_1} \otimes \CPL(0,2n,2i)_{q_2,w_2} \to \CPL(0,2n,2i)_{q_1 + q_2,w_1+ w_2}.$$

The second piece of extra structure is a comparison map $h\colon \CPL(2n) \to \CPL(2n+2)$ that we will call the hook map and which is visualised in \cref{fig:hook map}. 

To give the formal definition, recall from \cref{defn: diagrams Lk Rk and Lmax} that $R_1 \in \TL(2n, 2n+2)$ is the diagram with a single right cup at position 1, and $L_2 \in \TL(2n+2, 2n)$ is the diagram with a single left cup at position 2.

\begin{definition} \label{def: hook map}
    The \emph{hook map} is the map $h\colon \CPL(0,2n,2i) \hookrightarrow \CPL(0,2n+2,2i)$ which in degree $q$ is given by 
    \begin{align*}
        \TL(0, 2n) \otimes \TL_{2n}^{\otimes q-1} \otimes \TL(2n, 2i) &\to \TL(0, 2n+2) \otimes \TL_{2n+2}^{\otimes q-1} \otimes \TL(2n+2, 2i)\\
        (D_0 | D_1 | \dots | D_q | D_{q+1}) &\mapsto (D_0 R_1 | L_2 D_1 R_1 | \dots | L_2  D_q R_1 | L_2 D_{q+1}). 
    \end{align*}
\end{definition}

\begin{figure}[h]
    \begin{tikzpicture}[scale=0.4, baseline=(base),xscale=-1]
        \def\widthscale{4.5};

        \coordinate (base) at (0,2);
        \draw[gray,line width = 1](0,0.5)--(0,4.5);
        \draw[gray,line width = 1](\widthscale,0.5)--(\widthscale,4.5);
        \draw[gray,line width = 1](2*\widthscale,0.5)--(2*\widthscale,4.5);

        \foreach \y in {1,2, 3,4}{
            \foreach \x in {0,1,2}{
                \draw[fill=black]  (\widthscale*\x,\y) circle [radius=0.15];
                \draw (\widthscale*\x,\y) node[right] {};
            }
            
        }

        \draw (0,1) to[out=180,in=-90] (-1,1.5) to[out=90,in=180] (0,2);
        \draw (0,3) to[out=180,in=-90] (-1,3.5) to[out=90,in=180] (0,4);
        \draw (0,2) to[out=0,in=-90] (1,2.5) to[out=90,in=0] (0,3);

        \draw (0,1) to[out=0,in=180] (\widthscale,3);
        \draw (0,4) to[out=0,in=180] (\widthscale,4);

        \draw (\widthscale,1) to[out=180,in=-90] (\widthscale-1,1.5) to[out=90,in=180] (\widthscale,2);
        \draw (\widthscale,1) to[out=0,in=-90] (\widthscale+1,1.5) to[out=90,in=0] (\widthscale,2);

        \draw (\widthscale,3) to[out=0,in=180] (2*\widthscale,1);
        \draw (\widthscale,4) to[out=0,in=180] (2*\widthscale,2);

        \draw (2*\widthscale,2) to[out=0,in=-90] (2*\widthscale+1,2.5) to[out=90,in=0] (2*\widthscale,3);
        \draw (2*\widthscale,3) to[out=180,in=-90] (2*\widthscale-1,3.5) to[out=90,in=180] (2*\widthscale,4);
        \draw (2*\widthscale,1) to[out=0,in=-90] (2*\widthscale+2,2.5) to[out=90,in=0] (2*\widthscale,4);
    \end{tikzpicture}
\quad
    $\longmapsto$
\quad
    \begin{tikzpicture}[scale=0.4, baseline=(base),xscale=-1]
        
        \def\widthscale{4.5};

        \coordinate (base) at (0,4);
        \draw[gray,line width = 1](0,2.5)--(0,6.5);
        \draw[gray,line width = 1](\widthscale,2.5)--(\widthscale,6.5);
        \draw[gray,line width = 1](2*\widthscale,2.5)--(2*\widthscale,6.5);

        \draw[red,line width = 1](0,0.5)--(0,2.5);
        \draw[red,line width = 1](\widthscale,0.5)--(\widthscale,2.5);
        \draw[red,line width = 1](2*\widthscale,0.5)--(2*\widthscale,2.5);

        \draw[dotted] (0,3) to[out=180,in=-90] (-1,3.5) to[out=90,in=180] (0,4);
        \draw (0,5) to[out=180,in=-90] (-1,5.5) to[out=90,in=180] (0,6);
        \draw (0,4) to[out=0,in=-90] (1,4.5) to[out=90,in=0] (0,5);

        \draw (0,3) to[out=0,in=180] (\widthscale,5);
        \draw (0,6) to[out=0,in=180] (\widthscale,6);

        \draw[dotted] (\widthscale,3) to[out=180,in=-90] (\widthscale-1,3.5) to[out=90,in=180] (\widthscale,4);
        \draw (\widthscale,3) to[out=0,in=-90] (\widthscale+1,3.5) to[out=90,in=0] (\widthscale,4);

        \draw[dotted] (\widthscale,5) to[out=0,in=180] (2*\widthscale,3);
        \draw (\widthscale,6) to[out=0,in=180] (2*\widthscale,4);

        \draw (2*\widthscale,4) to[out=0,in=-90] (2*\widthscale+1,4.5) to[out=90,in=0] (2*\widthscale,5);
        \draw (2*\widthscale,5) to[out=180,in=-90] (2*\widthscale-1,5.5) to[out=90,in=180] (2*\widthscale,6);
        \draw (2*\widthscale,3) to[out=0,in=-90] (2*\widthscale+2,4.5) to[out=90,in=0] (2*\widthscale,6);

        \draw[red] (0,1) to[out=180,in=-90] (-2,2.5) to[out=90,in=180] (0,4);
        \draw[red] (0,2) to[out=180,in=-90] (-1,2.5) to[out=90,in=180] (0,3);
        \draw[red] (0,1) to[out=0,in=-90] (1,1.5) to[out=90,in=0] (0,2);
        
        \draw[red] (\widthscale,1) to[out=180,in=-90] (\widthscale-2,2.5) to[out=90,in=180] (\widthscale,4);
        \draw[red] (\widthscale,2) to[out=180,in=-90] (\widthscale-1,2.5) to[out=90,in=180] (\widthscale,3);

        \draw[red] (\widthscale,1) to[out=0,in=-90] (\widthscale+1,1.5) to[out=90,in=0] (\widthscale,2);

        \draw[red] (\widthscale,5) to[out=0,in=180] (2*\widthscale,1);

        \draw[red] (2*\widthscale,2) to[out=180,in=-90] (2*\widthscale-1,2.5) to[out=90,in=180] (2*\widthscale,3);

        \draw[red] (2*\widthscale,1) to[out=0,in=-90] (2*\widthscale+1,1.5) to[out=90,in=0] (2*\widthscale,2);

        \fill[pattern = north west lines, pattern color = red] (0,3) to[out=180,in=-90] (-1,3.5) to[out=90,in=180] (0,4) to[out=180,in=90] (-2,2.5) to[out=-90,in=180] (0,1) to[out=0,in=-90] (1,1.5) to[out=90,in=0] (0,2) to[out=180,in=-90] (-1,2.5) to[out=90,in=180] (0,3);

        \fill[pattern = north west lines, pattern color = red] (\widthscale,3) to[out=180,in=-90] (\widthscale-1,3.5) to[out=90,in=180] (\widthscale,4) to[out=180,in=90] (\widthscale-2,2.5) to[out=-90,in=180] (\widthscale,1) to[out=0,in=-90] (\widthscale+1,1.5) to[out=90,in=0] (\widthscale,2) to[out=180,in=-90] (\widthscale-1,2.5) to[out=90,in=180] (\widthscale,3);

        \fill[pattern = north west lines, pattern color = red] (2*\widthscale,3) to[out=180,in=0] (\widthscale,5) to[out=0,in=180]  (2*\widthscale,1) to[out=0,in=-90] (2*\widthscale+1,1.5) to[out=90,in=0] (2*\widthscale,2) to[out=180,in=-90] (2*\widthscale-1,2.5) to[out=90,in=180] (2*\widthscale,3);

        \foreach \y in {1,2, 3,4,5,6}{
            \foreach \x in {0,1,2}{
                \draw[fill=black]  (\widthscale*\x,\y) circle [radius=0.15];
                \draw (\widthscale*\x,\y) node[right] {};
            }
            
        }
    \end{tikzpicture}

    \caption{The effect of the hook map $h\colon  \CPL(4) \to \CPL(6)$ on a basis element.}
    \label{fig:hook map}
\end{figure}
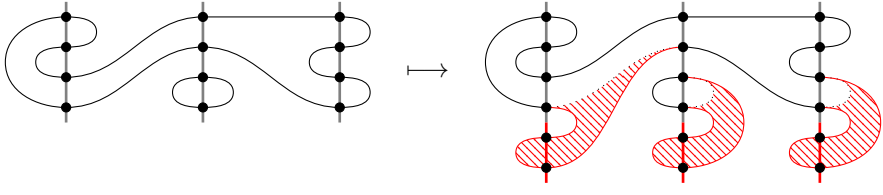

\begin{lemma}
    \label{lemma:hook-map}
        The hook map $h\colon \CPL(0,2n,2i) \hookrightarrow \CPL(0,2n+2,2i)$ is:
    \begin{enumerate}
        \item a map of dgas when $i=0$, and
        \item a map of dg-$L(2n)$-modules when $i \geq 0$.
    \end{enumerate}
    Further, $h$ is injective, and is compatible with the weight grading of \cref{def: weight} (when $(R,a)$ is such that the weight grading is defined).
\end{lemma}

\begin{proof} It is immediate from the definitions that $h$ gives a map of graded algebras/modules. It is injective because the multiplication maps $(L_2 \cdot-)$ and $(-\cdot R_1)$ are injective. The identity $R_1 \cdot L_2 = \id_{\TL_{2n}}$ implies that $h$ respects the face maps, so is a map of differential graded algebras/modules. This identity also implies that the number of loops formed by a basis element is not changed by applying $h$, so $h$ must respect the weight grading.
\end{proof}

\subsection{Change of coefficients}\label{sec:ChangingCoeff}
So far we have worked with the commutative ring $R$ and the parameter $a \in R$ fixed, but $\CPL(2n) = \CPL(2n;R,a)$ is functorial in this data and we will exploit this in several ways.

If $\phi\colon S \to R$ is a homomorphism of commutative rings, and $b\in S$ satisfies $a=\phi(b)$, then there is an isomorphism of dgas
$$\CPL(2n;R,a) = \CPL(2n;S,b) \otimes_S R.$$
We will often apply this to the homomorphism $\phi \colon\mathbb{Z}[a] \to R$ sending the formal symbol $a$ to $a \in R$, to reduce to studying the universal case $\CPL(2n; \mathbb{Z}[a], a)$.

If $x \in R$ is not a zerodivisor, then there is a short exact sequence of $R$-modules $0 \to R \overset{x \cdot -}\to R \overset{q}\to R/x \to 0$, where $q$ is the quotient map. Applying $\CPL(2n;R,a) \otimes_R -$ to this gives a short exact sequence of left $\CPL(2n;R,a)$-modules
$$\CPL(2n;R,a) \overset{x \cdot -}\longrightarrow \CPL(2n;R,a) \overset{q}\longrightarrow \CPL(2n;R/x,a \text{ mod } x).$$

If $a \in R$ is not a zerodivisor then we may apply this with $x=a$. But as $a = d(\Phi) \in \CPL(2n;R,a)_0$ is a boundary ($\Phi$ as in \cref{definition:Phi}), the map $a \cdot -$ is chain nullhomotopic, giving short exact sequences
$$0 \to H_i(\CPL(2n;R,a)) \overset{q}\to H_i(\CPL(2n;R/a,0)) \overset{\partial}\to H_{i-1}(\CPL(2n;R,a)) \to 0.$$
This identifies $H_*(\CPL(2n;R,a))$ with the kernel of the $a$-Bockstein operator $\beta_a := q \circ \partial$ on $H_*(\CPL(2n;R/a,0))$. It applies in particular to the universal case $(\mathbb{Z}[a], a)$, so reduces understanding $H_*(\CPL(2n;\mathbb{Z}[a], a))$ to understanding $H_*(\CPL(2n;\mathbb{Z},0))$ and the action of $\beta_a$.

Computationally, as $\CPL(2n;S,b)$ is a bounded below chain complex of free and hence flat $S$-modules, there is a Universal Coefficient spectral sequence \cite[Theorem 5.6.4]{Weibel}
$$E^2_{p,q} = \Tor_p^S(H_q(\CPL(2n;S,b)), R) \Longrightarrow H_{p+q}(\CPL(2n;R,a)).$$
We will not make use of this, but the reader may find it useful in making calculations over rings other than those treated in \cref{sec:Calc}.

\section{Small models}
\label{section: small models}

In this section we prove \cref{thm:Main}. We will first consider the parameter zero case $(R, a) = (\mathbb{Z}, 0)$ with $\mathbb{Z}$ concentrated in rank zero. This is the input for the universal case $(R, a) = (\mathbb{Z}[a], a)$, where $a$ has rank one, from which \cref{thm:Main} is deduced by change of coefficients. To prove the parameter zero case we will use the following technical result.

\begin{theorem}
    \label{theorem:derived-outermost-cup-complex-resolution}
    Let $R$ be concentrated in rank zero, let $a = 0$, and consider the differential bigraded algebra $\CPL(2n)=\CPL(2n; R, 0)$. Then the trivial module $R$ of $\CPL(2n)$ $($concentrated in bidegree $(0,0))$ admits a resolution
    $$0 \leftarrow R \leftarrow \derivedOutermost _0(2n) \leftarrow \derivedOutermost _1(2n) \leftarrow \derivedOutermost _2(2n) \leftarrow \cdots \leftarrow \derivedOutermost _n(2n) \leftarrow 0$$
    of left modules $\derivedOutermost_i(2n)$ over $\CPL(2n)$ with the following properties:
    \begin{enumerate}
        \item There are preferred equivalences $w_i\colon \Sigma^{i,i} \CPL(2n) \xrightarrow{\simeq} \derivedOutermost_i(2n)$ of differential bigraded modules over $\CPL(2n)$ such that 
        $$\Sigma^{i-1,i-1} \CPL(2n) \simeq \derivedOutermost_{i-1}(2n) \leftarrow \derivedOutermost_i(2n) \simeq \Sigma^{i,i} \CPL(2n)$$
        is given by right multiplication by $i\Phi$. \label{ThmPart1}
        \item The hook map $h\colon \CPL(2n) \to \CPL(2n+2)$ extends to a map of resolutions $\derivedOutermost_*(2n) \to \derivedOutermost_*(2n+2)$. The $i$-th term of this map corresponds to $\Sigma^{i,i}h$ under the equivalences $w_i$.\label{ThmPart2}
    \end{enumerate}
\end{theorem}

For now we will take this resolution as given: it will be constructed in \cref{def: derived outermost}, and is called the \emph{derived complex of outermost cups}. \cref{section: complexes} to \cref{section:proof of main technical thm} introduce the technical set up and prove this theorem.

It will often be convenient, given a chain complex $C_0 \leftarrow C_1 \leftarrow \dots \leftarrow C_n$ of chain complexes $C_i$, to write $\left[ C_0 \leftarrow C_1 \leftarrow \dots \leftarrow C_n \right]$ for the totalisation of the associated double complex. In particular, $C_0$ is a subcomplex with no degree shift.

\subsection{The argument for $(\mathbb{Z}, 0)$}
In this section we fix $(R,a)=(\Z,0)$. As $a=0$, the dga $L(2n)$ has a weight grading.

\subsubsection{The existence of a dga model} Our first goal is to establish the (bi)degrees of the generators of a quasi-free model for $L(2n)$, as follows. Later we will determine the differential on this model.

\begin{proposition}\label{prop:Indec}
There is an equivalence 
$$(T_\mathbb{Z}[x_1, x_3, \ldots, x_{2n-1}], d) \overset{\sim}\to \CPL(2n)$$
of differential bigraded algebras, where $|x_{2i-1}| = (2i-1, i)$, for some differential $d$.

Furthermore, the model for $\CPL(2n)$ may be obtained from the model for $\CPL(2n-2)$ by attaching a single $(2n-1, n)$-cell.
\end{proposition}

To prove this Proposition we will make use of a Hurewicz principle for dgas, which will allow us to obtain a model for $L(2n)$ by calculating the homology of the derived tensor product $\mathbb{Z} \otimes^\mathbb{L}_{\CPL(2n)} \mathbb{Z}$, which we will be able to do using \cref{theorem:derived-outermost-cup-complex-resolution}. Although this principle is well-known in some circles, in order to be self-contained we state and prove it below.

Say that a dga $A$ is \emph{connected} if it vanishes in negative degrees and in addition the unit map $1\colon \mathbb{Z} \to H_0(A)$ is an isomorphism; a connected dga has a unique augmentation $\epsilon\colon A \to H_0(A) \cong \mathbb{Z}$, using which we may form the derived tensor product $\mathbb{Z} \otimes^\mathbb{L}_{A} \mathbb{Z}$.

\begin{lemma}[Hurewicz principle] \label{lemma: hurewicz}
Let $A \to B$ be a map of connected dgas. Then there is an induced map
$$h\colon  H_i(B,A) \to H_{i+1}(\mathbb{Z} \otimes_B^\mathbb{L} \mathbb{Z}, \mathbb{Z} \otimes_A^\mathbb{L} \mathbb{Z}),$$
and if $H_i(B,A)=0$ for all $i<n$ then $h$ is an isomorphism for all $i \leq n$.
\end{lemma}
\begin{proof}
If $A$ is a connected dga with augmentation $\epsilon\colon A \to \mathbb{Z}$ then the commutative square
\begin{equation*}
\begin{tikzcd}[column sep=5ex]
A \arrow[r,equals] & A \otimes_A^\mathbb{L} A \dar{\epsilon \otimes_A A} \rar{A \otimes_A \epsilon}& A \otimes_A^\mathbb{L} \mathbb{Z} \dar{\epsilon \otimes_A A} \arrow[r,equals] & \mathbb{Z}\\
\mathbb{Z} \arrow[r, equals] & \mathbb{Z} \otimes_A^\mathbb{L} A \rar{A \otimes_A \epsilon} & \mathbb{Z} \otimes_A^\mathbb{L} \mathbb{Z}
\end{tikzcd}
\end{equation*}
yields a map $\mathrm{st}(A) \to \mathbb{Z} \otimes_A^\mathbb{L} \mathbb{Z}$ from the homotopy pushout of $\mathbb{Z} \overset{\epsilon}\leftarrow A \overset{\epsilon}\rightarrow \mathbb{Z}$. Writing $\bar{A}$ for the homotopy fibre of $\epsilon\colon A \to \mathbb{Z}$, there is a canonical equivalence $\mathrm{st}(A) \simeq \mathbb{Z} \oplus \Sigma \bar{A}$. This gives a map $h \colon H_i(\bar{A}) \to H_{i+1}(\mathbb{Z} \otimes_A^\mathbb{L} \mathbb{Z})$,  functorially in the connected dga $A$.

If $f \colon A \to B$ is a map of connected dgas such that $H_i(B,A)=0$ for all $i<n$ then we claim that the induced square
\begin{equation*}
\begin{tikzcd}[column sep=3ex]
\mathrm{st}(A) \rar \dar & \mathrm{st}(B) \dar\\
\mathbb{Z} \otimes_A^\mathbb{L} \mathbb{Z} \rar & \mathbb{Z} \otimes_B^\mathbb{L} \mathbb{Z}
\end{tikzcd}
\end{equation*}
is $(n+2)$-cocartesian. To see this, we may assume up to homotopy equivalence that $A$ and $B$ are levelwise free $\mathbb{Z}$-modules and then calculate the derived tensor product using the normalised bar resolution \cite[Exercise 8.6.4]{Weibel}
$$\mathbb{Z} \overset{\sim}\longleftarrow \left[ A \leftarrow A \otimes \bar{A} \leftarrow A \otimes \bar{A}^{\otimes 2} \leftarrow A \otimes \bar{A}^{\otimes 3} \leftarrow \dots \right],$$
giving
$$\mathbb{Z} \otimes_A^\mathbb{L} \mathbb{Z} \overset{\sim}\longleftarrow \left[ \mathbb{Z} \leftarrow  \bar{A} \leftarrow  \bar{A}^{\otimes 2} \leftarrow \bar{A}^{\otimes 3} \leftarrow \dots \right].$$
We recognise the bottom portion $[\mathbb{Z} \overset{0}\leftarrow  \bar{A}]$ as $\mathrm{st}(A)$. This shows that the total homotopy cofibre of this square has a filtration with associated graded 
$$\bigoplus_{p \geq 2} \Sigma^p \mathrm{Cofib}(\bar{f}^{\otimes p}\colon\bar{A}^{\otimes p} \to \bar{B}^{\otimes p}).$$ 
If the homotopy cofibre of $f$ is $(n-1)$-connected then as $\bar{A}$ and $\bar{B}$ are 0-connected the homotopy cofibre of $\bar{f}^{\otimes p}$ is $n$-connected for all $p \geq 2$, so all of these terms are $(n+2)$-connected. The spectral sequence for this filtration then proves the claim.

The $(n+2)$-cocartesianness of the square implies that the induced map
$$h\colon  H_i(B,A) \cong H_i(\bar{B}, \bar{A}) \cong H_{i+1}(\mathrm{st}(B),\mathrm{st}(A)) \to H_{i+1}(\mathbb{Z} \otimes_B^\mathbb{L} \mathbb{Z}, \mathbb{Z} \otimes_A^\mathbb{L} \mathbb{Z})$$
 between horizontal homotopy cofibres is an isomorphism for $i+1 < n+2$ (and also an epimorphism for $i+1=n+2$), i.e.~for $i \leq n$.
\end{proof}

\begin{proof}[Proof of \cref{prop:Indec}]
Before starting the argument proper, we calculate the homology of $\mathbb{Z} \otimes^\mathbb{L}_{\CPL(2n)} \mathbb{Z}$. We do so using the resolution
$$\mathbb{Z} \overset{\sim}\longleftarrow [\derivedOutermost_0(2n) \leftarrow \derivedOutermost_1(2n) \leftarrow \derivedOutermost_2(2n) \leftarrow \cdots \leftarrow \derivedOutermost_n(2n)]$$
of left $\CPL(2n)$-modules provided by \cref{theorem:derived-outermost-cup-complex-resolution}.
 By Part (a) of that Theorem there are equivalences $\derivedOutermost_i(2n) \simeq \Sigma^{i,i} \CPL(2n)$ of left $\CPL(2n)$-modules, and therefore equivalences $\mathbb{Z} \otimes^\mathbb{L}_{\CPL(2n)}\derivedOutermost_i(2n) \simeq \Sigma^{i,i}\mathbb{Z}$. The resulting hyperhomology spectral sequence therefore collapses by sparsity, giving
$$H_{d,w}(\mathbb{Z} \otimes^\mathbb{L}_{\CPL(2n)} \mathbb{Z}) \cong \begin{cases}
    \mathbb{Z} & (d,w) = (2i, i)\text{ for } 0 \leq i \leq n\\
    0 & \text{else}.
\end{cases}$$
We can calculate the effect on homology of the map $\mathbb{Z} \otimes^\mathbb{L}_{\CPL(2n-2)} \mathbb{Z} \to \mathbb{Z} \otimes^\mathbb{L}_{\CPL(2n)} \mathbb{Z}$ induced by the hook map $h \colon \CPL(2n-2) \to \CPL(2n)$ using the map $\derivedOutermost_*(2n-2) \to \derivedOutermost_*(2n)$ between resolutions given by \cref{theorem:derived-outermost-cup-complex-resolution} (b). As the map $\derivedOutermost_i(2n-2) \to \derivedOutermost_i(2n)$ corresponds to $\Sigma^{i,i}h \colon \Sigma^{i,i} \CPL(2n-2) \to \Sigma^{i,i} \CPL(2n)$, the induced map
$$\Sigma^{i,i}\mathbb{Z} \simeq \mathbb{Z} \otimes^\mathbb{L}_{\CPL(2n-2)}\derivedOutermost_i(2n-2) \to \mathbb{Z} \otimes^\mathbb{L}_{\CPL(2n)}\derivedOutermost_i(2n) \simeq \Sigma^{i,i}\mathbb{Z}$$
is the identity. It follows that on homology $\mathbb{Z} \otimes^\mathbb{L}_{\CPL(2n-2)} \mathbb{Z} \to \mathbb{Z} \otimes^\mathbb{L}_{\CPL(2n)} \mathbb{Z}$ is an isomorphism on to everything except the $\mathbb{Z}$ in bidegree $(2n, n)$.

We make a further preliminary calculation with quasi-free dgas. If $V$ is a free graded $\mathbb{Z}$-module and $(T_\mathbb{Z}[V], d)$ is a dga whose underlying algebra is the tensor algebra on $V$, then we may calculate $\mathbb{Z} \otimes_{(T_\mathbb{Z}[V], d)}^\mathbb{L} \mathbb{Z}$ using the quasi-free Koszul resolution
$$\mathbb{Z} \overset{\sim}\longleftarrow \left[(T_\mathbb{Z}[V], d) \leftarrow (T_\mathbb{Z}[V] \otimes_\mathbb{Z} V, d) \right].$$
Here $T_\mathbb{Z}[V] \otimes_\mathbb{Z} V$ is the augmentation ideal of $T_\mathbb{Z}[V]$ with the induced differential. This yields $\mathbb{Z} \otimes_{(T_\mathbb{Z}[V], d)}^\mathbb{L} \mathbb{Z} \simeq [\mathbb{Z} \overset{0}\leftarrow (V, d_V)] \simeq \mathbb{Z} \oplus \Sigma (V, d_V)$, where $d_V$ is the differential on $V$ determined by $d(v) = d_V(v) + \text{decomposables} \in T_\mathbb{Z}[V]$. In particular, if the differential $d$ is such that $d(v)$ is decomposable for all $v \in V$ then $\mathbb{Z} \otimes_{(T_\mathbb{Z}[V], d)}^\mathbb{L} \mathbb{Z} \simeq \mathbb{Z} \oplus \Sigma V$.

We now begin the proof of \cref{prop:Indec} proper. The claim in the proposition for $n=0$ is clear: recall that we defined $\CPL(0) = \mathbb{Z}$ which has a model with no generators. For $n>0$, suppose for an induction that a model $(T_\mathbb{Z}[x_1, x_3, \ldots, x_{2n-3}], d) \overset{\sim}\to \CPL(2n-2)$ has been obtained, and consider the map of dgas 
$$\phi \colon (T_\mathbb{Z}[x_1, x_3, \ldots, x_{2n-3}], d) \overset{\sim}\to \CPL(2n-2) \to \CPL(2n).$$ 
We have seen that the cofibre of the induced map
$$\mathbb{Z} \otimes_{\phi}^\mathbb{L} \mathbb{Z}\colon\mathbb{Z} \otimes_{(T_\mathbb{Z}[x_1, x_3, \ldots, x_{2n-3}], d)}^\mathbb{L} \mathbb{Z} \overset{\sim}\to \mathbb{Z} \otimes_{\CPL(2n-2)}^\mathbb{L} \mathbb{Z} \to \mathbb{Z} \otimes_{\CPL(2n)}^\mathbb{L} \mathbb{Z}$$
has homology just $\mathbb{Z}$ supported in bidegree $(2n,n)$, so we may lift a generator along the Hurewicz map 
\begin{align*}
&h \colon H_{2n-1,n}(\CPL(2n), (T_\mathbb{Z}[x_1, x_3, \ldots, x_{2n-3}], d)) \\
&\quad\quad\quad\overset{\sim}\to H_{2n,n}(\mathbb{Z} \otimes_{\CPL(2n)}^\mathbb{L} \mathbb{Z}, \mathbb{Z} \otimes_{(T_\mathbb{Z}[x_1, x_3, \ldots, x_{2n-3}], d)}^\mathbb{L} \mathbb{Z})\cong \mathbb{Z},
\end{align*}
which is an isomorphism by the Hurewicz principle of \cref{lemma: hurewicz}, to a class $[x_{2n-1}]$. A cycle $x_{2n-1}$ representing this class is the data of a cycle 
$$\delta x_{2n-1} \in (T_\mathbb{Z}[x_1, x_3, \ldots, x_{2n-3}], d)$$
of bidegree $(2n-2, n)$ together with an element $x_{2n-1} \in L(2n)$ such that $d(x_{2n-1}) = \phi(\delta x_{2n-1})$. This yields an extension
$$\phi\colon(T_\mathbb{Z}[x_1, x_3, \ldots, x_{2n-3}], d) \to (T_\mathbb{Z}[x_1, x_3, \ldots, x_{2n-3}] * T_\mathbb{Z}[x_{2n-1}], d) \overset{\phi'}\to L(2n)$$
where $\phi'$ sends the symbol $x_{2n-1}$ to the element $x_{2n-1} \in L(2n)$, and the differential satisfies $d(x_{2n-1}) = \delta x_{2n-1}$. By construction the induced map 
$$\mathbb{Z} \otimes_{\phi'}^\mathbb{L} \mathbb{Z}\colon\mathbb{Z} \oplus \Sigma \mathbb{Z}\{x_1, x_3, \ldots, x_{2n-1}\} \to \mathbb{Z} \otimes_{\CPL(2n)}^\mathbb{L} \mathbb{Z}$$ 
is an equivalence, so by the Hurewicz principle the map $\phi'$ is an equivalence.
\end{proof}

\subsubsection{Calculating the differential}

\cref{prop:Indec} shows that $\CPL(2n)$ is equivalent to a dga with one generator in each degree $1,3,\ldots,2n-1$, and by construction this equivalence is natural with respect to the hook map. However we are yet to identify the differential on this dga explicitly. We wish to show that it may be described by the following model.

\begin{definition}
We define a differential bigraded algebra by
$$M(2n) := (T_\mathbb{Z}[x_1, x_3, \ldots, x_{2n-1}], d_\text{model})$$
with $|x_{2i-1}| = (2i-1, i)$ and differential of bidegree $(-1,0)$ given by
$$d_\text{model}(x_{2i-1}) = \sum_{\substack{j+k=i \\ j,k>0}} \binom{i}{j} x_{2j-1}x_{2k-1}.$$
\end{definition}

We leave the reader to verify that $d_\text{model} \circ d_\text{model} =0$ (or, as mentioned in the introduction, to notice that $M(2n)$ is the cobar construction of a truncated divided power coalgebra). Our next goal is to establish some features of this model which will be useful in proving that it is equivalent to $\CPL(2n)$.

The element
$$z_0 := \sum_{\substack{j+k=n+1 \\ j,k>0}} \binom{n+1}{j} x_{2j-1}x_{2k-1} \in M(2n)_{2n,n+1}$$
is a cycle, so represents a class in $H_{2n,n+1}(M(2n))$. This is the class that we attach $x_{2n+1}$ along to form the next model $M(2n+2)$.

There is an algebra map $M(2n-2) \to M(2n)$ induced by inclusion, and on underlying graded algebras this makes $M(2n)$ into a free right $M(2n-2)$-module. The derived tensor product $M(2n) \otimes^\mathbb{L}_{M(2n-2)} \mathbb{Z}$ may therefore be computed as the literal tensor product $M(2n) \otimes_{M(2n-2)} \mathbb{Z}$ with its induced differential, and in degrees $\sim 2n$ this has the form
$$\cdots \leftarrow \mathbb{Z}\{x_{2n-1}\} \overset{0}\leftarrow \mathbb{Z}\{x_1 x_{2n-1}\} \overset{0}\leftarrow \mathbb{Z}\{x_1^2x_{2n-1}\} \leftarrow \mathbb{Z}\{x_3x_{2n-1}, x_1^3 x_{2n-1}\} \leftarrow \cdots$$
so that $H_{2n,n+1}(M(2n) \otimes^\mathbb{L}_{M(2n-2)} \mathbb{Z}) \cong \mathbb{Z}\{x_1 x_{2n-1}\}$.

\begin{proposition}\label{prop:DetectingQuot}
The map
$$q_*: H_{2n, n+1}(M(2n)) \rightarrow H_{2n,n+1}(M(2n) \otimes^\mathbb{L}_{M(2n-2)} \mathbb{Z}) \cong \mathbb{Z}\{x_1 x_{2n-1}\}$$
is injective, and sends $z_0$ to $(n+1) x_1 x_{2n-1}$.
\end{proposition}
\begin{proof}
We will prove this by identifying the domain. Note first that the word length of a monomial in $x_1, \dots , x_{2n-1}$ is encoded in its bigrading: if it has bidegree $(d,w)$ then it has word length $2w-d$.
This means that the bidegree $|z_0| = (2n,n+1)$ piece of $M(2n)$ is spanned by words of length 2. Checking bidegrees gives
$$M(2n)_{2n,n+1} = \mathbb{Z}\{x_1 x_{2n-1}, x_3 x_{2n-3}, \ldots x_{2n-3} x_3, x_{2n-1} x_1\}.$$

No nonzero class in this bidegree can be a boundary: if $u = d_\text{model}(v)$ with $|u| = (2n,n+1)$, then $|v|=(2n+1,n+1)$ so $v$ must have word length 1, i.e.~be a linear combination of the generators $x_1, \dots , x_{2n-1}$, but no $x_{2i-1}$ has this bidegree.

We claim that the submodule of cycles has rank precisely 1. If
$$z = \sum_{i=1}^{n} A_i x_{2i-1} x_{2n-(2i-1)}$$
is a cycle then applying $d_\text{model}$ gives
$$0 = \sum_{i=1}^{n} A_i \left(\sum_{\substack{j+k=i \\ j,k>0}} \binom{i}{j} x_{2j-1}x_{2k-1} x_{2n-(2i-1)} - x_{2i-1} \sum_{\substack{a+b=n-i+1 \\ a,b>0}} \binom{n-i+1}{a} x_{2a-1}x_{2b-1}\right).$$
The coefficient of $x_{2i-1} x_{2a-1} x_{2b-1}$ is $A_{i+a}\binom{i+a}{a} - A_i\binom{n-i+1}{a}$, which shows that we must have the proportionality
$$A_{1+a} = A_1 (1+a) \binom{n}{a} \text{ for } 1 \leq 1+a \leq n$$
between the coefficients $A_i$. Thus the cycles of bidegree $(2n,n)$ have rank at most 1. As $z_0$ is a non-zero cycle, they in fact have rank precisely 1. 

The map $q\colon M(2n) \to M(2n) \otimes^\mathbb{L}_{M(2n-2)} \mathbb{Z}$ sends the cycle $z_0$ to $(n+1) x_1 x_{2n-1} \neq 0 \in \mathbb{Z}\{x_1 x_{2n-1}\} \cong H_{2n,n+1}(M(2n) \otimes^\mathbb{L}_{M(2n-2)} \mathbb{Z})$, so induces an injection.
\end{proof}

\begin{remark}
The element $z_0$ does not generate $H_{2n, n+1}(M(2n))$, but is a multiple of the generator. It follows from the proof that the generator is obtained by dividing $z_0$ by the gcd of its coefficients, i.e.~$\gcd\{{n+1 \choose j} | 1 \leq j \leq n\}$. We will not need this.
\end{remark}

The remainder of this section is devoted to proving the following proposition.

\begin{proposition}\label{proposition: equivalence of dga with model} There is a weak equivalence 
$$M(2n) \overset{\sim}\to \CPL(2n)$$
of differential bigraded algebras, which sends $x_1$ to the class $\Phi$ of \cref{definition:Phi}. Furthermore, these equivalences are natural with respect to the hook map.
\end{proposition}

We will show that in the model for $\CPL(2n)$ provided by \cref{prop:Indec} we may choose the generators $x_i$ so that the differential $d$ is given by $d_\text{model}$. We will do so by strong induction on $n$. To start the induction, we do some early cases by hand.

\begin{example}[$2n=2$]
\label{exampe:model 2n=2}
There is a model of the form $(T_\mathbb{Z}[x_1], d) \simeq \CPL(2)$ provided by \cref{prop:Indec}. The differential is completely determined by $d(x_1)$, but $H_{0, *}(\CPL(2))=\mathbb{Z}$ so we must have $d(x_1)=0$. We normalise so that $x_1$ maps to $\Phi$, so in particular $[x_1] = \Phi \in H_{1,*}(\CPL(2))$.
\end{example}

\begin{corollary}\label{cor:PhiGenerates}
    For all $2n \geq 2$ we have $H_{1,*}(L(2n)) \cong \mathbb{Z}\{\Phi\}$.
\end{corollary}
\begin{proof}
The previous Example establishes this for $2n=2$, and by \cref{prop:Indec} a model for $L(2n)$ may be obtained from a model for $L(2)$ by attaching dga cells of homological degrees $\geq 3$. These can not change the first homology.
\end{proof}

\begin{example}[$2n=4$]
There is a model $(T_\mathbb{Z}[x_1, x_3], d) \simeq \CPL(4)$ provided by \cref{prop:Indec}, which may be obtained from the model $M(2)$ for $\CPL(2)$ by attaching a cell $x_3$. Thus we must determine $d(x_3) \in H_{2,1}(M(2)) = \mathbb{Z}\{x_1^2\}$.

By \cref{theorem:derived-outermost-cup-complex-resolution} we have the resolution $R \leftarrow \CPL(4) \leftarrow \derivedOutermost_1(4) \leftarrow \derivedOutermost_2(4)$ of $R$ by left $\CPL(4)$-modules. Under the equivalences $\derivedOutermost_1(4) \simeq \Sigma^{1,1} \CPL(4)$ and $\derivedOutermost_2(4) \simeq \Sigma^{2,2}\CPL(4)$ the middle map is identified with $- \cdot \Phi$ and the right-hand map is identified with $- \cdot 2\Phi$. In particular as the composite of these maps is null it follows that $2\Phi^2 = 0 \in H_{*,*}(\CPL(4))$, but by exactness we also know that $\Phi^2 \neq 0$. From this it follows that we must have $d(x_3) = \pm 2 x_1^2$. After perhaps changing the definition of $x_3$ by a sign, we have identified the model with $M(4)$.
\end{example}

Before we can prove \cref{proposition: equivalence of dga with model} we need an additional technical lemma.

\begin{lemma}\label{lem:CalcDivisbility}
We have
$$H_{2n,n+1}(\CPL(2n+2) \otimes^\mathbb{L}_{\CPL(2n-2)} \mathbb{Z}) \cong \mathbb{Z}/(n+1).$$
\end{lemma}

\begin{proof}
Using the resolution
$$\mathbb{Z} \overset{\sim}\longleftarrow [\derivedOutermost_0(2n-2) \leftarrow \cdots \leftarrow \derivedOutermost_{n-1}(2n-2)]$$
of $\mathbb{Z}$ as a left $\CPL(2n-2)$-module, and the fact that the canonical map 
$$\CPL(2n+2) \otimes^\mathbb{L}_{\CPL(2n-2)} \derivedOutermost_i(2n-2) \to \derivedOutermost_i(2n+2)$$
is an equivalence (using \cref{theorem:derived-outermost-cup-complex-resolution} twice), we see that $\CPL(2n+2) \otimes^\mathbb{L}_{\CPL(2n-2)} \mathbb{Z}$ may be described as
$$C' := \left[ \derivedOutermost_0(2n+2) \leftarrow \cdots \leftarrow \derivedOutermost_{n-1}(2n+2) \right].$$
We may replace this by
$$C'' := \Sigma^{-1}[ \mathbb{Z} \overset{\epsilon}\leftarrow \derivedOutermost_0(2n+2) \leftarrow \cdots \leftarrow \derivedOutermost_{n-1}(2n+2) ]$$
without changing its homology in bidegree $(2n, n+1)$. This is the (desuspension of the) initial portion of the resolution
$$\mathbb{Z} \overset{\epsilon}\leftarrow \derivedOutermost_0(2n+2) \leftarrow \cdots \leftarrow \derivedOutermost_{n+1}(2n+2),$$
which is acyclic, so $C''$ is equivalent to 
$$ C := \Sigma^{n-1} \left[\derivedOutermost_{n}(2n+2)\leftarrow \derivedOutermost_{n+1}(2n+2) \right].$$
By \cref{theorem:derived-outermost-cup-complex-resolution}, under the identifications $\derivedOutermost_n(2n+2) \simeq \Sigma^{n,n} \CPL(2n+2)$ and $\derivedOutermost_{n+1}(2n+2) \simeq \Sigma^{n+1,n+1} \CPL(2n+2)$, the $\CPL(2n+2)$-module map $\derivedOutermost_{n+1}(2n+2) \to \derivedOutermost_{n}(2n+2)$ is identified with right multiplication by $(n+1)\Phi$, so $C$ is equivalent to the mapping cone of
$$- \cdot (n+1)\Phi  \colon\Sigma^{2n, n+1} \CPL(2n+2) \to  \Sigma^{2n-1, n} \CPL(2n+2).$$
Using \cref{cor:PhiGenerates} we therefore have an exact sequence
\begin{equation*}
    \begin{tikzcd}
H_{0,0}( \CPL(2n+2)) \rar{- \cdot (n+1)\Phi} \arrow[d, equals] & H_{1,1}(\CPL(2n+2)) \rar \arrow[d, equals] & H_{2n,n+1}(C) \rar & 0\\
\mathbb{Z}\{1\} & \mathbb{Z}\{\Phi\}
    \end{tikzcd}
\end{equation*}
giving $H_{2n,n+1}(\CPL(2n+2) \otimes^\mathbb{L}_{\CPL(2n-2)} \mathbb{Z}) \cong H_{2n,n+1}(C) \cong \mathbb{Z}/(n+1)$ as claimed.
\end{proof}

\begin{proof}[Proof of \cref{proposition: equivalence of dga with model}]
Suppose for an induction that we have models
$$M(2n-2) \overset{\sim}\to \CPL(2n-2) \text{ and } M(2n) \overset{\sim}\to \CPL(2n)$$
which are natural with respect to the hook map. We know from \cref{prop:Indec} that $\CPL(2n+2)$ has a model obtained from $M(2n)$ by attaching a single cell $x_{2(n+1)-1}$ along some element
$$[d(x_{2(n+1)-1})] \in H_{2n, n+1}(M(2n)).$$
We must show that this agrees with $[d_\text{model}(x_{2(n+1)-1})] = [z_0]$ (up to a sign, as we may rechoose the cell to be $-x_{2(n+1)-1}$).

By \cref{prop:DetectingQuot} the map
$$q_*\colon H_{2n, n+1}(M(2n)) \to H_{2n, n+1}(M(2n) \otimes^\mathbb{L}_{M(2n-2)} \mathbb{Z}) \cong \mathbb{Z}\{x_1 x_{2n-1}\}$$
is injective, so it suffices to show that $[d(x_{2(n+1)-1})]$ and $z_0$ agree (up to a sign) under this map. Using \cref{prop:DetectingQuot} the class $z_0$ maps to $(n+1) x_1 x_{2n-1}$, so we must show that the same is true (up to a sign) of $[d(x_{2(n+1)-1})]$.

As a model for $\CPL(2n+2)$ is obtained from $M(2n) \simeq \CPL(2n)$ by attaching a single cell $x_{2(n+1)-1}$ of degree $2n+1$ and weight $n+1$, we have
\begin{equation}\label{eq:HomologyOfF}
    H_{d,*}(\CPL(2n+2)/\CPL(2n)) = \begin{cases}
 0 & d < 2n+1 \\
 \mathbb{Z}\{x_{2(n+1)-1}\} & d=2n+1,
\end{cases}
\end{equation}
where this generator has weight $n+1$, since the hook map is injective by \cref{lemma:hook-map}. The map
\begin{equation*}
\begin{tikzcd}[column sep=3ex]
\CPL(2n) \rar \dar{q} & \CPL(2n+2) \dar \rar & \CPL(2n+2)/\CPL(2n) \dar\\
 \CPL(2n) \otimes^\mathbb{L}_{\CPL(2n-2)} \mathbb{Z} \rar & \CPL(2n+2) \otimes^\mathbb{L}_{\CPL(2n-2)} \mathbb{Z} \rar & (\CPL(2n+2)/\CPL(2n)) \otimes^\mathbb{L}_{\CPL(2n-2)} \mathbb{Z} 
\end{tikzcd}
\end{equation*}
of homotopy cofibre sequences (short exact sequences of chain complexes are always homotopy cofibre sequences) yields a map of long exact sequences, a portion of which is
\begin{equation*}
\begin{tikzcd}[column sep=1.5ex]
\mathbb{Z}\{x_{2(n+1)-1}\} \rar{\sim} \dar{d} & H_{2n+1,n+1}((\CPL(2n+2)/\CPL(2n)) \otimes^\mathbb{L}_{\CPL(2n-2)} \mathbb{Z}) \dar{\partial}\\
H_{2n,n+1}(\CPL(2n)) \arrow[r, hook, "q_*"] \dar & H_{2n,n+1}(\CPL(2n) \otimes^\mathbb{L}_{\CPL(2n-2)} \mathbb{Z}) \arrow[r,equals] \dar & \mathbb{Z}\{x_1x_{2n-1}\}\\
H_{2n,n+1}(\CPL(2n+2))  \rar \dar & H_{2n,n+1}(\CPL(2n+2) \otimes^\mathbb{L}_{\CPL(2n-2)} \mathbb{Z}) \dar\\
0 \rar & 0
\end{tikzcd}
\end{equation*}

To show that $[d(x_{2(n+1)-1})]$ and $z_0$ agree up to a sign is therefore equivalent to showing that the group $H_{2n,n+1}(\CPL(2n+2) \otimes^\mathbb{L}_{\CPL(2n-2)} \mathbb{Z})$ has order $n+1$. This is true by \cref{lem:CalcDivisbility}.
\end{proof}

\subsection{The argument for $(\mathbb{Z}[a], a)$}
By \cref{proposition: equivalence of dga with model} we have a model $M(2n)$ for $L(2n;\Z,0)$. In this section we extend the model to build a model for the universal case $(\mathbb{Z}[a],a)$, in which $a$ has grading 1. We write $\CPL'(2n) := \CPL(2n; \mathbb{Z}[a], a)$, and let
$$M'(2n) := (T_{\mathbb{Z}[a]}[x'_1, x'_3, \ldots, x'_{2n-1}], d'_\text{model})$$
with $|x'_{2i-1}| = (2i-1, i)$ and differential given by $d'_\text{model}(x'_1)=a$
and
$$d'_\text{model}(x'_{2i-1}) = \sum_{\substack{j+k=i \\ j,k>0}} \binom{i}{j} x'_{2j-1}x'_{2k-1} \text{ for } 2i-1 > 1.$$
We again leave the reader to verify that $d'_\text{model} \circ d'_\text{model} =0$.

The goal of this section is then to establish the following universal case of the model, from which the general case will be deduced.

\begin{proposition}\label{prop: model universal case} There is a weak equivalence
$M'(2n) \overset{\sim}\to \CPL'(2n)$ of differential graded $\mathbb{Z}[a]$-algebras, which sends $x_1$ to $\Phi$.
\end{proposition}

This universal case will be deduced in turn from the case $(\mathbb{Z},0)$.

\begin{lemma} \label{lem: times a is zero on homology}
The map $a \cdot -\colon\Sigma^{0,1} \CPL'(2n) \to \CPL'(2n)$ is chain nullhomotopic for $2n \geq 2$; in particular multiplication by $a$ is zero on $H_{*,*}(\CPL'(2n))$. Similarly, multiplication by $a$ is zero on $H_{*,*}(\CPL'(2n+2), \CPL'(2n))$.
\end{lemma}
\begin{proof}
We have $d(\Phi) = a$ with $\Phi$ as in \cref{definition:Phi}, so $- \cdot \Phi$ gives the required chain nullhomotopy. For the second statement we use that $\CPL'(2n+2)/\CPL'(2n)$ is a $\CPL'(2n)$-module, and $a=0 \in H_{0,1}(\CPL'(2n))$.
\end{proof}

We have that $\CPL'(2n) \otimes_{\mathbb{Z}[a]} \mathbb{Z} = \CPL(2n)$ and so on. From \cref{eq:HomologyOfF} we have that
$$H_{d,*}(\CPL(2n+2)/\CPL(2n)) = \begin{cases}
0 & d < 2n+1 \\
\mathbb{Z}\{x_{2(n+1)-1}\} & d=2n+1.
\end{cases}$$
There is a short exact sequence of chain complexes
$$0 \to \CPL'(2n+2)/\CPL'(2n) \overset{a \cdot -} \to \CPL'(2n+2)/\CPL'(2n) \to (\CPL'(2n+2)/\CPL'(2n))\otimes_{\mathbb{Z}[a]} \mathbb{Z} \to 0$$
and the latter is $(\CPL'(2n+2)/\CPL'(2n))\otimes_{\mathbb{Z}[a]} \mathbb{Z} = \CPL(2n+2)/\CPL(2n)$. Using \cref{lem: times a is zero on homology}, the long exact sequence on homology splits,
and we deduce by induction over homological degree that
$$H_{d,*}(\CPL'(2n+2)/ \CPL'(2n)) = \begin{cases}
0 & d < 2n+1 \\
\mathbb{Z}[a]/(a)\{x'_{2(n+1)-1}\} & d=2n+1
\end{cases}$$
for some class $x'_{2(n+1)-1}$ of weight $n+1$ which reduces to $x_{2(n+1)-1}$ modulo $a$.

\begin{lemma} \label{lemma: single cell universal case}
The dga $\CPL'(2n+2)$ is obtained from $\CPL'(2n)$ by attaching a single cell along $\partial(x'_{2(n+1)-1}) \in H_{2n,n+1}(\CPL'(2n))$.
\end{lemma}
\begin{proof}
The above calculation gives a map of dgas
\begin{equation}\label{eq:CellAtt}
\CPL'(2n) * T_{\mathbb{Z}[a]}[x'_{2(n+1)-1}] \to \CPL'(2n+2) 
\end{equation} 
from the cell attachment. On applying $- \otimes_{\mathbb{Z}[a]} \mathbb{Z}$ this gives the map of dgas
\begin{equation}\label{eq:CellAtt2}
\CPL(2n) * T_{\mathbb{Z}}[x_{2(n+1)-1}] \to \CPL(2n+2)
\end{equation}
given by attaching a cell along $\partial(x_{2(n+1)-1})$, which is an equivalence by \cref{prop:Indec}. 

To complete the argument we make two observations. Firstly, observe that if $M$ is a finitely-generated $\mathbb{Z}[a]$-module with additional grading such that $M \otimes_{\mathbb{Z}[a]}\mathbb{Z}=0$, then $M=0$. This is because the assumption means that all elements of $M$ are infinitely divisible by $a$, but by finite generation there exists $k$ such that $M_i=0$ for $i<k$. Secondly, observe that if $C$ is a chain complex of finitely-generated free $\mathbb{Z}[a]$-modules with additional grading such that $C  \otimes_{\mathbb{Z}[a]}\mathbb{Z}$ is acyclic, then $C$ is acyclic. This is because the homology of such a chain complex consists of finitely-generated $\mathbb{Z}[a]$-modules with additional grading (as $\mathbb{Z}[a]$ is Noetherian), and the short exact sequence $0 \to C \overset{a \cdot -}\to C \to C \otimes_{\mathbb{Z}[a]} \mathbb{Z} \to 0$ of chain complexes gives an injection on homology $H_*(C)  \otimes_{\mathbb{Z}[a]}\mathbb{Z} = \coker(H_*(C) \xrightarrow{a \cdot -} H_*(C)) \hookrightarrow H_*(C \otimes_{\mathbb{Z}[a]}\mathbb{Z} )=0$, so by the first observation $H_*(C)=0$.

Apply the latter observation with $C$ the mapping cone of \cref{eq:CellAtt}. Then $C\otimes_{\mathbb{Z}[a]} \mathbb{Z}$ is the mapping cone of \cref{eq:CellAtt2} and hence is acyclic, so $C$ is acyclic too.
\end{proof}

\begin{proof}[Proof of \cref{prop: model universal case}] It follows from \cref{lemma: single cell universal case} and induction that there is a model
$$(T_{\mathbb{Z}[a]}(x'_1, x'_3, \ldots, x'_{2n-1}), d') \overset{\sim}\to \CPL'(2n),$$
reducing modulo $a$ to the model
$$M(2n) := (T_{\mathbb{Z}}[x_1, x_3, \ldots, x_{2n-1}], d_\text{model}) \overset{\sim}\to \CPL(2n)$$
established earlier. In particular, $\CPL'(2n) \to \CPL(2n)$ is an isomorphism in bidegree $(1,1)$, since there can be no $a$-multiples in this bidegree, so we must have $x'_1 \mapsto \Phi$ and so $d'(x'_1) = a$. A priori, for $2i-1 > 1$ we only have the equality $d'(x'_{2i-1}) = d_\text{model}(x_{2i-1})$ mod $a$. We wish to show that this equality also holds before modding out by $a$, so that the above model is $M'(2n)$.

This will be a simple consequence of the bigrading. If an element $a^j \cdot X$ of $T_{\mathbb{Z}[a]}(x'_1, x'_3, \ldots, x'_{2n-1})$ with $X$ a word in $x'_1, x'_3, \ldots, x'_{2n-1}$ has bidegree $(d, w)$, then $X$ has bidegree $(d, w-j)$ and so has word length $2(w-j)-d$. Writing
$$d'(x'_{2i-1}) = d_\text{model}(x_{2i-1}) + a \cdot \epsilon_i$$
it follows that $\epsilon_i$ has bidegree $(2i-2, i-1)$. It is a sum of $a^j \cdot X_j$'s with $X_j$ having word length $2(i-1-j)-(2i-2) = -2j$. This is impossible unless $j=0$ and $X_j$ is (a $\mathbb{Z}$-multiple of) the empty word. Thus $\epsilon_i$ must have homological degree 0. If $2i-1 > 1$ then for homological degree reasons we must have $\epsilon_i=0$. If $2i-1=1$ then this need not be the case, but we have arranged that $d'_\text{model}(x'_1) = a = d'(x'_1)$. \end{proof}

\subsection{Proof of \cref{thm:Main}} Before giving the proof, we recall the statement.

\begin{restate}{Theorem}{thm:Main}
    There is a weak equivalence 
$$M(2n;R,a):=(T_R[x_1, x_3, \ldots, x_{2 n-1}], d) \overset{\sim}\to \CPL(2n; R,a)$$
of differential graded $R$-algebras, where the left-hand side is the tensor algebra on the generators $x_i$ of degree $i$ equipped with the differential given by 
$$d(x_1)=a \text{ and } d(x_{2i-1}) = \sum_{\substack{j+k=i \\ j,k>0}} \binom{i}{j} x_{2j-1}x_{2k-1} \text{ for } 2i-1>1,$$ and which sends $x_1$ to $\Phi$.
\end{restate}
\begin{proof} Applying the derived tensor product $(-)\otimes^{\mathbb{L}}_{\mathbb{Z}[a]} R$ to the weak equivalence of \cref{prop: model universal case} gives a weak equivalence $$M'(2n) \otimes^{\mathbb{L}}_{\mathbb{Z}[a]} R \overset{\sim} \to \CPL'(2n) \otimes^{\mathbb{L}}_{\mathbb{Z}[a]} R.$$
Since $M'(2n)$ and $\CPL'(2n)$ are both complexes of free $\mathbb{Z}[a]$-modules, we may replace both derived tensor products with underived ones. The resulting quasi-isomorphism is the desired one. Note that under this equivalence $x_1$ is mapped to $\Phi$, following through from our normalisation in \cref{exampe:model 2n=2} when proving \cref{proposition: equivalence of dga with model}.
\end{proof}

\section{Calculations using small models}
\label{sec:Calc}
In this section we prove \cref{cor:MainCalculations}, which lists algebraic consequences of the model in \cref{thm:Main}. It has parts (a)--(e), which will be proved in turn.  

\subsection{Part (a) of \cref{cor:MainCalculations}}

\begin{restate}{Corollary}{cor:MainCalculations}
\mbox{}
\begin{enumerate}
    \item[(a)]
     $H_0(\CPL(2n;R,a)) = R/a$, so the $R$-module structure on $H_i(\CPL(2n;R,a))$ descends to $R/a$, i.e.~the homology groups are all annihilated by $a$.
\end{enumerate}
\end{restate}
\begin{proof}
    We have $a = d(\Phi)$, so multiplication by $a$ is chain nullhomotopic.
\end{proof}

\subsection{Parts (b) and (c) of \cref{cor:MainCalculations}: the case $a=0$}
\label{subsection:calculations-the-parameter-zero-case}

When $a=0$ the model is
$$M(2n;R,0) = (T_R[x_1, x_3, \ldots, x_{2n-1}], d_\text{model})$$
where the differential is given by 
$$d_\text{model}(x_{2i-1}) = \sum_{\substack{j+k = i \\ j,k>0}} \binom{i}{j} x_{2j-1}x_{2k-1}.$$
Consider the bigraded free $R$-module
$$C_n := R\{1, x_1, x_3, \ldots, x_{2n-1}\}$$
with $|x_{2i-1}|=(2i,i)$, and endow it with a coproduct by the formula
$$\psi(x_{2i-1}) := \sum_{j+k=i} \binom{i}{j} x_{2j-1} \otimes x_{2k-1},$$
a counit by $1 \mapsto 1, x_{2i-1} \mapsto 0 : C_n \to R$, and a coaugmentation by $1 \mapsto 1 : R \to C_n$. Then the cobar complex $\mathrm{Cobar}(R, C_n, R)$ of this coaugmented coalgebra is tautologically the dga $M(2n;R,0)$. The linear dual of the coalgebra $C_n$ is none other than the truncated divided power algebra
$$A_n := C_n^\vee \cong \Gamma_R[y]/(y^{[p]} : p \geq n+1)$$
where $y^{[i]}$, of bidegree $(-2i,i)$, is dual to $x_{2i-1}$. 

\begin{proposition}\label{prop: homology is Ext}
There is a bigraded algebra isomorphism
$$H_{*,*}(\CPL(2n; R, 0)) \cong \mathrm{Ext}_{A_n}^*(R, R).$$    
\end{proposition}
\begin{proof}
Using that $A_n$ is free as an $R$-module, so $\mathrm{Bar}(R,A_n,R)$ is a chain complex of free $R$-modules, there are equivalences
\begin{eqnarray*}
    \mathrm{Cobar}(R, C_n, R) &=& \mathrm{Hom}_R(\mathrm{Bar}(R,A_n,R), R) \\ &\simeq& \mathbb{R}\mathrm{Hom}_R(R \otimes_{A_n}^\mathbb{L} R, R) \simeq \mathbb{R}\mathrm{Hom}_{A_n}(R, R).
\end{eqnarray*}
Recall that $\mathrm{Cobar}(R, C_n, R)=M(2n;R,0)$ and apply \cref{thm:Main}.
\end{proof}

For general $R$ this is probably the best description of the homology of the dga of planar loops. However, in specific cases we can obtain a more concrete description.

\begin{rerestate}{Corollary}{cor:MainCalculations}
{restating:cor:MainCalculations(b)}
\mbox{}
\begin{enumerate}
    \item[(b)]
     For $2n=2$ there is an isomorphism
$$H_{*,*}(\CPL(2; R, 0)) \cong R[\Phi]$$
of bigraded algebras, with $|\Phi| = (1,1)$.
\end{enumerate}
\end{rerestate}
\begin{proof} When $2n=2$ the model of \cref{thm:Main} is $R[x_1]$, with zero differential, and $x_1$ maps to $\Phi \in \CPL(2)$.
\end{proof}

Almost all parts of \cref{cor:MainCalculations} part (c) will follow from \cref{prop: homology is Ext}. The exception is the identification of $\alpha$ as a Massey product, which needs some preparatory work in the model.

For our conventions on Massey products, we follow May \cite{MayMatricMassey}. Consider a dga $(A_*, d)$. Write $\overline{a} := (-1)^{\deg(a) + 1} a$. Given elements $\zeta_1, \dots , \zeta_i$ in a dga $(A_*, d)$, a \emph{defining system} \cite[Definition 1.2]{MayMatricMassey} consists of elements $a_{j,k}$ of $A$, for $0 \leq j < k \leq i$ and $(j,k) \neq (0,i)$, thought of as a matrix with only above-diagonal entries, and no top-right entry, such that
\begin{enumerate}
    \item for each $j$, $a_{j,j+1}$ is a cycle representing $\zeta_j$, and
    \item for $k \geq j+2$, we have $d(a_{j,k}) = \sum_{\ell=j+1}^{k-1} \overline{a}_{j,\ell} a_{\ell,k}$.
\end{enumerate}

The \emph{Massey product} $\langle \zeta_1, \dots , \zeta_i \rangle$ is the set of homology classes represented by elements $\sum_{\ell=1}^{i-1} \overline{a}_{0,\ell} a_{\ell,i}$ for some defining system $(a_{j,k})$. The Massey product is \emph{strictly defined} if every partial defining system (for some $m$, the data of only those $a_{j,k}$ for which $k-j \leq m$) can be completed to a defining system. It is \emph{uniquely defined} if (it is strictly defined and) consists of only a single element.

Since our dga $M(2n)$ is bigraded, our Massey products will be taken in a bigraded sense, meaning that if $|\zeta_j|=(d_j,w_j)$, then we insist that $|\langle \zeta_1, \dots \zeta_i \rangle| = ((i-2) + \sum_{j=1}^i d_j, \sum_{j=1}^i w_j)$.

\begin{lemma} \label{lemma: Massey} 
Assume $n! \in R^\times$. In the dga $M(2n) = (T_R[x_1, x_3, \ldots, x_{2n-1}], d_\text{model}),$ write $\Phi$ for the homology class $[x_1]$. For $3 \leq i \leq n+1$, the $i$-fold Massey product $m_{i} = \langle \Phi , \dots , \Phi \rangle$ is strictly and uniquely defined, consisting of the single element $$\frac{1}{i!} \sum_{\substack{j+k=i\\ j,k>0}} { i \choose j} x_{2j-1}x_{2k-1}.$$
\end{lemma}

The important part is the case $i=n+1$: for smaller $i$, $m_i$ contains $[d(\frac{1}{i!} x_{2i-1})]=[0]$.

\begin{proof} Fix $n$. We work by strong induction on $i$, and will show for each $i$ that there is precisely one defining system, namely $a_{j,k} = \frac{1}{(k-j)!} x_{2(k-j)-1}$, and that each partial defining system must also have these entries. This suffices, since the Massey product is therefore strictly and uniquely defined, and its single element is given by
$$\sum_{\ell=1}^{i-1} \overline{a}_{0,\ell} a_{\ell,i} = \sum_{\ell=1}^{i-1}  \frac{1}{(\ell)!(i-\ell)!}x_{2\ell-1} x_{2(i - \ell)-1} = \frac{1}{i!}\sum_{\substack{j+k = i \\ j,k > 0}}  {i \choose j} x_{2j-1} x_{2k-1}.$$

We begin with the (illustrative) base case $i=3$. The generator $x_1$ is the only cycle representing $\Phi$, since  $M(2n)_{1,1} = R \{x_1\}$. A defining system for $\langle \Phi , \Phi , \Phi \rangle$ must therefore take the form
$$\begin{pmatrix}
0 & x_1 & a_{0,2} & * \\
& 0 &  x_1 & a_{1,3} \\
& & 0 & x_1 \\
& & & 0
\end{pmatrix}.$$
The classes $a_{0,2}$ and $a_{1,3}$ must both have boundary $x_1^2$, hence must have bidegree $(3,2)$, and in particular word length 1. This bidegree is one-dimensional: $M(2n)_{3,2} = R \{x_{3}\}$, and $d(x_3) = 2 x_1^2$, so the only option is to set $a_{0,2}=a_{1,3}=\frac{1}{2} x_3$. The Massey product $\langle \Phi , \Phi , \Phi \rangle$ is thus the singleton $\{\frac{1}{2} x_1 x_3 + \frac{1}{2} x_3 x_1 \} ,$ which is the required formula.

For the inductive step, let $i \geq 4$. Fix $(j_0, k_0) \neq (0,i)$, and consider the part of the defining system for which $j \geq j_0$, $k \leq k_0$, and $(j,k) \neq (j_0,k_0)$. By definition, this is a defining system for a shorter (length $k_0 - j_0$) iterated Massey product $\langle \Phi , \dots, \Phi \rangle$, so by induction we have $a_{j,k} = \frac{1}{(k-j) !} x_{2 (k-j) -1}$ for each such $j$ and $k$. This excludes $(j_0,k_0)$ itself, so we may fill in all of the defining system apart from the last two entries, $a_{0,i-1}$ and $a_{1,i}$. As in the base case, each of these must be of bidegree $(2i-2,i)$ (so of word length 1) with boundary $\frac{1}{(i-1)!}\sum_{\substack{j+k = i - 1 \\ j,k > 0}}  {i -1 \choose j} x_{2j-1} x_{2k-1}$, so must in fact be $\frac{1}{(i-1)!} x_{2(i-1)-1}$ (this generator exists because we assumed $i\leq n + 1$). This establishes the existence and uniqueness of the defining system. \end{proof}

\begin{rerestate}{Corollary}{cor:MainCalculations}
{restating:cor:MainCalculations(c)} \mbox{}
\begin{enumerate}
    \item[(c)] If $n! \in R^\times$ and $2n \geq 4$ then there is an isomorphism
$$H_{*,*}(\CPL(2n; R, 0)) \cong T_R[\Phi,  \alpha]/(\Phi^2, \Phi\alpha - \alpha\Phi) = R[\Phi, \alpha]/(\Phi^2)$$
of bigraded algebras, with $|\Phi| = (1,1)$ and $|\alpha| = (2n, n+1)$. The class $\alpha$ is the unique element of the $(n + 1)$-fold Massey product $\langle \Phi, \Phi, \ldots, \Phi \rangle$, which is strictly defined.
\end{enumerate}

\end{rerestate}

\begin{proof}
Under the assumption that $n! \in R^\times$ the canonical map
$$R[y]/(y^{n+1}) \to \Gamma_R[y]/(y^{[p]} : p \geq n+1) =: A_n$$
is an isomorphism, and so
$$H_{*,*}(\CPL(2n, R, 0)) \cong \mathrm{Ext}^*_{R[y]/(y^{n+1})}(R,R).$$
Using the periodic resolution
\begin{equation*}
\begin{tikzcd}[column sep=2.5ex]
   \cdots \arrow[rr, "y"] & {} \ar[draw=none]{d}[name=Y, anchor=center]{} &  \Sigma^{-4n-4,2n+2}R[y]/(y^{n+1})\{z_{4n+1}\} \ar[rounded corners,
            to path={ -- ([xshift=2ex]\tikztostart.east)
                      |- (Y.center) \tikztonodes
                      -| ([xshift=-2ex]\tikztotarget.west)
                      -- (\tikztotarget)}]{dll}[swap, at end]{y^n} \\  
   \Sigma^{-2n-4, n+2}R[y]/(y^{n+1})\{z_{2n+1}\} \arrow[rr, "y"] & {} \ar[draw=none]{d}[name=X, anchor=center]{} &  \Sigma^{-2n-2,n+1}R[y]/(y^{n+1})\{z_{2n}\} \ar[rounded corners,
            to path={ -- ([xshift=2ex]\tikztostart.east)
                      |- (X.center) \tikztonodes
                      -| ([xshift=-2ex]\tikztotarget.west)
                      -- (\tikztotarget)}]{dll}[swap, at end]{y^n} \\      
  \Sigma^{-2,1}R[y]/(y^{n+1})\{z_1\} \arrow[rr, "y"] & {} & R[y]/(y^{n+1})\{1\} \rar{\epsilon} & R 
\end{tikzcd}    
\end{equation*}
to calculate this gives
$$\mathrm{Ext}^*_{R[y]/(y^{n+1})}(R,R) = R\{1, \zeta_{1,1}, \zeta_{2n, n+1}, \zeta_{2n+1, n+2}, \ldots\},$$
where $\zeta_{i,j}$ is dual to $z_i$, and has bidegree $(i, j)$. For $n \geq 2$ it is standard to compute the Yoneda product here (for example by constructing a representing map explicitly using the resolution recipe for the product, which can be found e.g.~in \cite[Section 9.2]{McCleary}), giving the presentation $T_R[\Phi,  \alpha]/(\Phi^2, \Phi\alpha - \alpha\Phi)$ of the homology, with $\Phi = \zeta_{1,1}$ and $\alpha = \zeta_{2n,n+1}$.

It remains only to identify $\alpha$ as a Massey product. Note first that in the model $M(2n)$, the differential into bidegree $(2n,n+1)$ comes from bidegree $(2n+1,n+1)$, which is spanned by words of length 1, i.e.~by generators. None of the generators have this bidegree, so $M(2n)_{2n+1,n+1} = 0$, and the homology in bidegree $(2n,n+1)$ is equivalently the cycles.

By \cref{lemma: Massey}, the $(n+1)$-fold Massey product $\langle \Phi , \dots , \Phi \rangle$ is strictly and uniquely defined in the model $M(2n)$, and its unique element is the homology class of the (nonzero) cycle $$\mu := \frac{1}{(n+1)!} \sum_{\substack{j+k=n+1\\ j,k>0}} { n+1 \choose j} x_{2j-1}x_{2k-1}$$
in bidegree $(2n,n+1)$. The coefficients of $\mu$ have no non-unit divisors, hence no non-unit common divisor, so $\mu$ is not divisible by any non-unit. It follows that $\mu$ generates $H_{2n,n+1}(M(2n))$, so $\mu = u \alpha$ for $u \in R^\times$, and replacing $\alpha$ by $u \alpha$ does not change the presentation.
\end{proof}

In the following example the assumption of the previous Corollary fails only slightly.

\begin{example} \label{example: 2n=2p}
Consider the data $(\mathbb{Z}/p, 0)$ with $p$ a prime number and $2n=2p$, in which case $n!$ is not a unit in $\mathbb{Z}/p$. Write $M(2p)=M(2p;\mathbb{Z}/p, 0)$. In this case $d(x_{2p-1})=0$ and so we have a coproduct
$$M(2p) = M(2p-2) * (T_{\mathbb{Z}/p}[x_{2p-1}], 0)$$
of dgas. As a chain complex we can write this as the sum 
$$M(2p-2) * (T_{\mathbb{Z}/p}[x_{2p-1}], 0) = \bigoplus_{i=0}^\infty M(2p-2) \cdot x_{2p-1} \cdot M(2p-2) \cdots x_{2p-1} \cdot M(2p-2),$$
where there are $i$ instances of $x_{2p-1}$ in the $i$th summand. This expresses $M(2p)$ as a sum of tensor powers over $\mathbb{Z}/p$ of copies of $M(2p-2)$, which by the K{\"u}nneth theorem shows that $H_{*,*}(\CPL(2p ; \mathbb{Z}/p, 0)) \cong H_{*,*}(\CPL(2p-2 ; \mathbb{Z}/p, 0)) * \mathbb{Z}/p[x_{2p-1}]$. As \cref{cor:MainCalculations} part (c) applies to the first factor, we obtain an isomorphism
$$T_{\mathbb{Z}/p}[\Phi, \alpha, x_{2p-1}]/(\Phi^2, \Phi\alpha - \alpha \Phi) \overset{\sim}\to H_{*,*}(\CPL(2p; \mathbb{Z}/p, 0)),$$
where $\alpha := [\tfrac{1}{p}d_\text{model}(x_{2p-1})]$.
\end{example}

\subsection{Part (d) of \cref{cor:MainCalculations}: the case $(R, a)$, $a$ not a zerodivisor}\label{sec:aNotZerodivisor}

\begin{rerestate}{Corollary}{cor:MainCalculations}
{restating:cor:MainCalculations(d)}\mbox{}
\begin{enumerate}
    \item[(d)] If $a \in R$ is not a zerodivisor and $n! \in (R/a)^\times$ then there is an isomorphism 
$$H_{*}(\CPL(2n; R, a)) \cong R/a[\beta]$$
of graded algebras, with $|\beta|=2n$. The class $\beta$ is uniquely determined by the property that the change-of-rings map $(R,a) \to (R/a,0)$ sends it to the class $\alpha \in H_{2n,n+1}(\CPL(2n ; R/a, 0))$ of part (c) (or to $\Phi^2$ if $2n=2$). It follows in particular that $H_{*}(\CPL(2n; R, a))$ admits a bigrading.
\end{enumerate}

\end{rerestate}

\begin{proof} Since $a$ is not a zerodivisor, \cref{sec:ChangingCoeff} gives a short exact sequence (which a priori is only singly graded) for every $i$:
$$0 \to H_i(\CPL(2n;R,a)) \overset{q}\to H_i(\CPL(2n;R/a,0)) \overset{\partial}\to H_{i-1}(\CPL(2n;R,a)) \to 0.$$
We know the middle term: by (c) it is $R/a[\Phi, \alpha]/(\Phi^2)$ if $2n \geq 4$, and by (b) it is $R/a[\Phi]$ if $2n=2$. We also know that $H_0(\CPL(2n;R,a)) \cong R/a$. It follows by induction on degree that $H_{*}(\CPL(2n; R, a))$ is as claimed as a graded $R$-module. Since the map $q$ above is an injection, the algebra structure must be as claimed.
\end{proof}

\begin{remark}
It is not so easy to find interesting (i.e.~not formal) dgas whose homology is a polynomial ring: see Dwyer--Greenlees--Iyengar \cite[Section 6]{DwyerGreenleesIyengar} and Bay{\i}nd{\i}r \cite{Bayindir}. The above with data $(\mathbb{Z}, p)$ for $p$ a prime number and $n < p$  gives a dga $\CPL(2n;\mathbb{Z},p)$ having homology $\mathbb{Z}/p[\beta]$ with $\beta$ of degree $2n$. If $\CPL(2n;\mathbb{Z},p)$ were formal then it would contain the dga $\mathbb{Z}/p$ as a retract (in the homotopy category of dgas) and so the dga $\CPL(2n;\mathbb{Z}/p,0) \simeq \CPL(2n;\mathbb{Z},p) \otimes_\mathbb{Z}^\mathbb{L} \mathbb{Z}/p$ would contain the dga $\mathbb{Z}/p \otimes_\mathbb{Z}^\mathbb{L} \mathbb{Z}/p$ as a retract. The latter has homology $\mathbb{Z}/p[\Phi]/(\Phi^2)$, which is indeed a retract of the homology algebra $\mathbb{Z}/p[\Phi, \alpha]/(\Phi^2)$ of $\CPL(2n;\mathbb{Z}/p,0)$, but if it was a retract of dgas then the Massey product $\alpha = \langle \Phi, \ldots, \Phi \rangle$ would lie in the subalgebra generated by $\Phi$, and it does not. So $\CPL(2n;\mathbb{Z},p)$ is not formal, and even its $2n$-truncation is not formal. When $n=p-1$ it must therefore be the example found by Bay{\i}nd{\i}r \cite[Theorem 1.8]{Bayindir}. We are grateful to Markus Land for an interesting discussion on this topic.
\end{remark}

\subsection{Part (e) of \cref{cor:MainCalculations}: the case $2n=4$}

In this case a fairly complete analysis can be made, as the dga $(T_R[x_1, x_3], d)$ with $d(x_1)=a$ and $d(x_3)=2x_1^2$ is not too complicated. The resulting homology is quite rich, which is why we discuss it in detail. We hope the methods we describe will be useful for further calculations.

We will first discuss $H_{*,*}(\CPL(4 ; \mathbb{Z}, 0))$, from which the method of \cref{sec:ChangingCoeff} can then be used to understand the universal case $H_{*,*}(\CPL(4 ; \mathbb{Z}[a], a))$. To begin, \cref{example: 2n=2p} shows that $H_{*,*}(\CPL(4;\mathbb{Z}/2, 0)) \cong T_{\mathbb{Z}/2}[x_1, x_3]$. The short exact sequence 
$$\CPL(4;\mathbb{Z}/2, 0) \overset{2}\to \CPL(4;\mathbb{Z}/4, 0)\to \CPL(4;\mathbb{Z}/2, 0)$$
yields a Bockstein operator $\beta$ on $H_{*,*}(\CPL(4;\mathbb{Z}/2, 0)) \cong T_{\mathbb{Z}/2}[x_1, x_3]$, which is a derivation squaring to zero and satisfying $\beta(x_1)=0$ and $\beta(x_3)=x_1^2$. Considering it as a differential on $T_{\mathbb{Z}/2}[x_1, x_3]$, we write $\beta H_*$ for its homology (this is a graded algebra).

\begin{proposition}\label{prop:BocksteinHomology}
Writing $\gamma := x_1x_3 + x_3x_1$ and $\Phi := x_1$, which are both annihilated by $\beta$, there is an algebra isomorphism $\beta H_* \cong \mathbb{Z}/2[\Phi, \gamma]/(\Phi^2)$.
\end{proposition}
\begin{proof}
As $\beta(x_1)=0$, the map $x_1 \cdot -\colon T_{\mathbb{Z}/2}[x_1, x_3] \to T_{\mathbb{Z}/2}[x_1, x_3]$ commutes with $\beta$, i.e.~is a chain map with respect to this differential. It is injective, and its cokernel may be described as $\mathbb{Z}/2\{1\} \oplus x_3 \cdot T_{\mathbb{Z}/2}[x_1, x_3]$, so we have a short exact sequence $$T_{\mathbb{Z}/2}[x_1, x_3] \to T_{\mathbb{Z}/2}[x_1, x_3] \to \mathbb{Z}/2\{1\} \oplus x_3 \cdot T_{\mathbb{Z}/2}[x_1, x_3]$$
of right $(T_{\mathbb{Z}/2}[x_1, x_3], \beta)$-modules. The differential $\beta$ on the cokernel kills 1, and as $\beta(x_3 \cdot z) = x_1^2 \cdot z + x_3 \cdot \beta(z)$ with the first summand divisible on the left by $x_1$, it makes $x_3 \cdot T_{\mathbb{Z}/2}[x_1, x_3]$ into a chain complex which is the 3-fold suspension of $(T_{\mathbb{Z}/2}[x_1, x_3], \beta)$. Therefore, the corresponding long exact sequence on $\beta$-homology groups has the form
$$\cdots \overset{\partial} \to \beta H_{i-1} \overset{\Phi \cdot -}\to \beta H_i \overset{q}\to \beta H_{i-3} \overset{\partial}\to \beta H_{i-2} \overset{\Phi \cdot -}\to \beta H_{i-1} \overset{q}\to \cdots,$$
with a small correction around $i=0$ due to the unit. A diagram chase shows that the connecting map $\partial$ is also given by $\Phi \cdot -$, using the formula for $\beta(x_3\cdot z)$ above. All the maps are right $\beta H_*$-module maps. One checks that $q(\gamma) = x_1 = \Phi$.

Using this it is simple to show by induction on degree that the canonical map
$$\mathbb{Z}/2\{1, \Phi, \gamma, \Phi\gamma, \gamma^2, \Phi\gamma^2, \gamma^3, \Phi\gamma^3, \ldots\} \to \beta H_*$$
is an isomorphism. In the algebra structure on the target $\Phi^2=0$, because $\beta H_2=0$. To obtain the claimed algebra structure it remains to show that $\Phi\gamma = \gamma\Phi$. At the chain level these are represented by $x_1(x_1x_3 + x_3x_1)$ and $(x_1x_3 + x_3x_1)x_1$ respectively, whose difference is $x_1^2 x_3 + x_3 x_1^2 = \beta(x_3^2)$, so is trivial in $\beta$-homology.
\end{proof}

\begin{rerestate}{Corollary}{cor:MainCalculations}
{restating:cor:MainCalculations(e)}\mbox{}
\begin{enumerate}
\item[(e)]
\begin{enumerate}
    \item[(i)] $H_{*,*}(\CPL(4 ; \mathbb{Z}, 0))$ only has 2-power torsion, and it is entirely simple 2-torsion.
    \item[(ii)] Letting $\gamma := [x_1x_3 + x_3x_1] \in H_{4,3}(\CPL(4 ; \mathbb{Z}, 0))$, which is also the Massey product $\langle \Phi, 2\Phi, \Phi\rangle$, there is an isomorphism
    $$H_{*,*}(\CPL(4 ; \mathbb{Z}, 0))/\text{$2$-torsion} \cong \mathbb{Z}[\Phi, \gamma]/(\Phi^2)$$
    of bigraded algebras.
    \item[(iii)] The Poincar{\'e} series $f(t,s) := \sum_{i,j=0}^\infty \mathrm{dim}_{\mathbb{Z}/2} \text{tors} H_{i,j}(\CPL(4 ; \mathbb{Z}, 0)) t^i s^j$ for the ranks of the torsion subgroups is given by
    $$\frac{s^2 t^2}{(1-st - s^2 t^3 ) (1 - s^3 t^4)}$$
    $$\quad\quad = s^2 t^2 + s^3 t^3 + s^4 t^4 + s^4 (s + 1) t^5 + s^5 (s + 3) t^6 + s^6 (s + 4) t^7 + \cdots.$$
\end{enumerate}
\end{enumerate}
\end{rerestate}
\begin{proof}
We have $H_{*,*}(\CPL(4 ; \mathbb{Z}, 0)) \otimes \mathbb{Z}[\tfrac{1}{2}] = H_{*,*}(\CPL(4 ; \mathbb{Z}[\tfrac{1}{2}], 0))$  and \cref{cor:MainCalculations} part (c) applies to show that this is $\mathbb{Z}[\tfrac{1}{2}][\Phi, \alpha]/(\Phi^2)$. This is torsion-free, so $H_{*,*}(\CPL(4 ; \mathbb{Z}, 0))$ has no odd prime torsion.

If $C$ is a chain complex of finitely-generated free $\mathbb{Z}$-modules, and $\beta$ is the Bockstein operator on $H_*(C \otimes \mathbb{Z}/2)$, then $H_*(C)$ has only simple 2-torsion if and only if $\mathrm{dim}_{\mathbb{Z}/2} \beta H_n(C) = \mathrm{rk}_{\mathbb{Z}}  H_n(C)$ for all $n$ \cite[Corollary 3E.4]{HatcherAT}. The above calculation with $\mathbb{Z}[\tfrac{1}{2}]$-coefficients together with \cref{prop:BocksteinHomology} shows this is the case for $C=\CPL(4;\mathbb{Z},0)$. This proves (i).

As $d(x_3)=2x_1^2$, and 
$$d(x_3^2) = 2x_1^2 x_3 - 2x_3 x_1^2 = 2(x_1(x_1x_3 + x_3 x_1) - (x_1x_3 + x_3 x_1)x_1)=2(\Phi\gamma-\gamma\Phi),$$
it follows that $\Phi^2=0$ and $\Phi\gamma = \gamma\Phi$ modulo (2-)torsion. This defines a ring homomorphism $\rho\colon\mathbb{Z}[\Phi, \gamma]/(\Phi^2) \to H_{*,*}(\CPL(4 ; \mathbb{Z}, 0))/\text{$2$-torsion}$. On applying $- \otimes \mathbb{Z}[\tfrac{1}{2}]$ it becomes an isomorphism, using the calculation in the first paragraph and the fact that $2\alpha = \gamma$ (applying the case $3 = i = n+1$ of \cref{lemma: Massey} giving a formula for $m_{3} = \langle \Phi , \Phi , \Phi \rangle = \alpha$). To check that $\rho$ is an isomorphism, it remains to show that $\rho(\Phi\gamma^n)$ is not divisible by 2. For this we use that to a chain complex $C$ of finitely-generated free $\mathbb{Z}$-modules there is associated a map
$$b\colon H_*(C)/\text{$2$-torsion} \to \beta H_*(C)$$
by sending an integral homology class to its reduction mod 2, which is annihilated by $\beta$: it is elementary to verify that it is well-defined, and is an algebra map if $C$ is a dga. If $x \in H_n(C)$ is divisible by 2 modulo $2$-torsion, so $x = 2y \in H_n(C)/\text{$2$-torsion}$, then it follows that $b(x)= b(2y) = 2 b(y) = 0$. In our case the composition
$$\mathbb{Z}[\Phi, \gamma]/(\Phi^2) \overset{\rho}\to H_{*,*}(\CPL(4 ; \mathbb{Z}, 0))/\text{$2$-torsion} \overset{b}\to \beta H_* \cong \mathbb{Z}/2[\Phi, \gamma]/(\Phi^2)$$
sends $\Phi$ and $\gamma$ to elements of the same name, so $b\rho(\Phi \gamma^n) = \Phi \gamma^n \neq 0$ and so $\rho(\Phi \gamma^n)$ is not divisible by 2. This proves (ii).

Define the Poincar{\'e} series
\begin{align*}
g(t,s) &: = \sum_{i.j=0}^\infty \mathrm{rk}_{\mathbb{Z}} H_{i,j}(\CPL(4 ; \mathbb{Z}, 0)) t^i s^j\\
h(t,s) &:= \sum_{i,j=0}^\infty \mathrm{dim}_{\mathbb{Z}/2} H_{i,j}(\CPL(4 ; \mathbb{Z}/2, 0)) t^i s^j.
\end{align*}
The first may be determined by the ranks after inverting 2, so by the first paragraph it is
$g(t,s) = (1 + ts)/(1-t^4s^3)$. The second may be determined by $H_{*,*}(\CPL(4 ; \mathbb{Z}/2, 0)) = T_{\mathbb{Z}}[x_1, x_3]$ and the standard Poincar{\'e} series for a tensor algebra to be $h(t,s) = 1/(1- ts - t^3s^2)$.
It follows from the Universal Coefficient Theorem for $\Z/2$ that these Poincar{\'e} series are related to the one $f(t,s)$ for $\mathrm{dim}_{\mathbb{Z}/2} \text{tors}H_{*,*}(\CPL(4;\mathbb{Z},0))$ by $h(t,s) = g(t,s) + (1+t)f(t,s)$, as the 2-torsion is simple. Making these substitutions and rearranging proves (iii).
\end{proof}

These calculations have immediate consequences for $H_{*,*}(\CPL(4 ; \mathbb{Z}[a], a))$ using the method of \cref{sec:aNotZerodivisor}. We leave the spelling out of this to the interested reader.

\section{Cell modules and the complexes of innermost and outermost cups}\label{section: complexes}
The goal of the rest of the paper is to prove \cref{theorem: shift} and \cref{theorem:derived-outermost-cup-complex-resolution}.
To this end, in this section we will introduce three highly connected combinatorial complexes, and in the next section we will prove technical lemmas on the chain-level. 
In the final two sections we will use these complexes to prove the two theorems respectively: \cref{theorem: shift} is proved via a straightforward algebraic topology argument, and the proof of \cref{theorem:derived-outermost-cup-complex-resolution} is a more involved argument by building and manipulating derived versions of our combinatorial complexes.

We start by defining the modules which we would like to resolve.

\begin{definition}[Cell modules] 
\label{definition:cell-modules}
    The \emph{cell module} $\cell(2n,2i)$ is the left $\TL_{2n}$-module
	$$\cell(2n,2i)  := \TL(2n, 2i)/\{\text{diagrams having a right cup}\}.$$	
\end{definition}

That is, $\cell(2n,2i)$ has an $R$-module basis consisting of diagrams in which every right-hand node is connected to some left-hand node.

\begin{remark}
    \label{remark:cellmodule-0-i-is-trivial}
		If $i>n$ then any diagram in $\TL(2n, 2i)$ must have a right cup, so $\cell(2n,2i)=0$ in this case.
\end{remark}

\begin{remark}
    \label{remark:cell modules in rep theory}
    Temperley--Lieb algebras are examples of cellular algebras \cite{GrahamLehrer}, and the cell modules in \cref{definition:cell-modules} play an important role in the representation theory of $\TL_{2n}$. For example, it is known that if $R = \mathbb{C}$ and the parameter is $a=q+q^{-1}$ for $q$ not a root of unity, then $\TL_{2n}$ is semisimple and the collection
    $$\{ \cell(2n,2i) \ | \ 0 \leq i \leq n\}$$
    is a complete set of irreducibles \cite[Corollary 4.6]{RidoutStAubin}. Similar results hold for Temperley--Lieb algebras on an odd number of strands \cite{RidoutStAubin}.
\end{remark}

\subsection{Cup complexes}

In this section we introduce two complexes. We will describe them in terms of diagrams with additional data: certain arcs will be distinguished. We will call these arcs \emph{dashed}, and we draw them dashed.

\subsubsection{Innermost cups}
Our first complex is a resolution for the cell module $\cell(2n,2i)$, which in the case $i=n$ recovers \cite[Definition 8]{Sroka}. By an \emph{innermost cup} in a diagram refers we mean an arc which joins two consecutive nodes on the same side: we say \emph{left} innermost cup or \emph{right} innermost cup to distinguish the sides.

\begin{definition}[Complex of innermost cups]\label{defn:complex of innermost cups algebraic}
Let $\InnermostCupComplex_q(2n,2i)$ be the free $R$-module with basis the diagrams in $\TL(2n,2i)$ equipped with the additional data of $q$ dashed innermost right cups. This has the structure of a left $\TL_{2n}$-module by concatenating  diagrams, and keeping the same dashed innermost cups. The collection of these form a semi-simplicial $\TL_{2n}$-module, where the $i$-th face map un-dashes the $i$-th dashed innermost cup (the dashed cups are ordered from bottom to top).
The alternating sum of these face maps gives a differential
$$\delta_{q}\colon \InnermostCupComplex_{q}(2n, 2i) \to \InnermostCupComplex_{q-1}(2n, 2i).$$
This is the \emph{complex of innermost cups}, denoted $\InnermostCupComplex_*(2n,2i)$.

We let $\AugmentedInnermostCupComplex_*(2n,2i)$ be the augmented form of this chain complex obtained by setting 
$$\AugmentedInnermostCupComplex_{-1}(2n,2i) := \cell(2n, 2i),$$
and letting $\delta_0\colon \AugmentedInnermostCupComplex_{0}(2n, 2i) = \TL(2n, 2i) \to \AugmentedInnermostCupComplex_{-1}(2n, 2i) =  \cell(2n, 2i)$ be the quotient map used to define $\cell(2n, 2i)$ in \cref{definition:cell-modules}.
\end{definition}

\begin{figure}[ht]
    \begin{tikzpicture}[scale=0.4, baseline=(base)]
               \coordinate (base) at (0,3.25);
        \draw[gray,line width = 1](0,0.5)--(0,8.5);
        \draw[gray,line width = 1](6,0.5)--(6,8.5);

        \foreach \x in {1,2, 3,4,5,6,7,8}{
            \draw[fill=black]  (0,\x) circle [radius=0.15];
   
        } 

        \foreach \y in {1,2,3,4,5,6}{
            \draw[fill=black]  (6,\y + 1) circle [radius=0.15];
         
        } 

        \draw (0,1) to[out=0,in=180] (6,2);
        \draw (0,6) to[out=0,in=180] (6,5);

        \draw (0,3) to[out=0,in=-90] (1,3.5) to[out=90,in=0] (0,4);
        \draw (0,7) to[out=0,in=-90] (1,7.5) to[out=90,in=0] (0,8);
        \draw (0,2) to[out=0,in=-90] (2,3.5) to[out=90,in=0] (0,5);

        \draw[dashed] (6,3) to[out=180,in=-90] (5,3.5) to[out=90,in=180] (6,4);
        \draw[dashed] (6,6) to[out=180,in=-90] (5,6.5) to[out=90,in=180] (6,7);
    \end{tikzpicture}
    \quad
    $\mapsto$
    \quad
    \begin{tikzpicture}[scale=0.4, baseline=(base)]
               \coordinate (base) at (0,3.25);
        \draw[gray,line width = 1](0,0.5)--(0,8.5);
        \draw[gray,line width = 1](6,0.5)--(6,8.5);

        \foreach \x in {1,2, 3,4,5,6,7,8}{
            \draw[fill=black]  (0,\x) circle [radius=0.15];
   
        } 

        \foreach \y in {1,2,3,4,5,6}{
            \draw[fill=black]  (6,\y + 1) circle [radius=0.15];
         
        } 

        \draw (0,1) to[out=0,in=180] (6,2);
        \draw (0,6) to[out=0,in=180] (6,5);

        \draw (0,3) to[out=0,in=-90] (1,3.5) to[out=90,in=0] (0,4);
        \draw (0,7) to[out=0,in=-90] (1,7.5) to[out=90,in=0] (0,8);
        \draw (0,2) to[out=0,in=-90] (2,3.5) to[out=90,in=0] (0,5);

        \draw(6,3) to[out=180,in=-90] (5,3.5) to[out=90,in=180] (6,4);
        \draw[dashed] (6,6) to[out=180,in=-90] (5,6.5) to[out=90,in=180] (6,7);
    \end{tikzpicture}
    \quad
    $-$
    \quad
    \begin{tikzpicture}[scale=0.4, baseline=(base)]
               \coordinate (base) at (0,3.25);
        \draw[gray,line width = 1](0,0.5)--(0,8.5);
        \draw[gray,line width = 1](6,0.5)--(6,8.5);

        \foreach \x in {1,2, 3,4,5,6,7,8}{
            \draw[fill=black]  (0,\x) circle [radius=0.15];
   
        } 

        \foreach \y in {1,2,3,4,5,6}{
            \draw[fill=black]  (6,\y + 1) circle [radius=0.15];
         
        } 

        \draw (0,1) to[out=0,in=180] (6,2);
        \draw (0,6) to[out=0,in=180] (6,5);

        \draw (0,3) to[out=0,in=-90] (1,3.5) to[out=90,in=0] (0,4);
        \draw (0,7) to[out=0,in=-90] (1,7.5) to[out=90,in=0] (0,8);
        \draw (0,2) to[out=0,in=-90] (2,3.5) to[out=90,in=0] (0,5);

        \draw[dashed] (6,3) to[out=180,in=-90] (5,3.5) to[out=90,in=180] (6,4);
        \draw(6,6) to[out=180,in=-90] (5,6.5) to[out=90,in=180] (6,7);
    \end{tikzpicture}
    .
    \caption{The differential $\delta_2\colon\InnermostCupComplex_2(8,6) \to \InnermostCupComplex_1(8,6)$ of the complex of innermost cups; each face maps un-dashes a dashed cup.}  \label{fig: innermost cup complex}
\end{figure}

We will now explain how this complex is close to a projective resolution of $\cell(2n, 2i)$. To do so, we observe that the $\TL_{2n}$-modules $\InnermostCupComplex_{q}(2n, 2i)$ decompose as a direct sum indexed over the locations of the dashed innermost right cups. More precisely, for each collection $F$ of $q$ dashed innermost right cups, there is a summand $\operatorname{Cup}(F)(2n,2i)$
of $\InnermostCupComplex_{q}(2n, 2i)$ consisting of those diagrams whose associated dashed innermost right cups are precisely $F$. Writing $\mathcal{I}_{2i}^q$ for the set of such $F$'s, there is a $\TL_{2n}$-module decomposition
\begin{equation}\label{eq:InnDecomposedIntoCup}
\InnermostCupComplex_{q}(2n, 2i) = \bigoplus_{F \in \mathcal{I}^q_{2i}} \operatorname{Cup}(F)(2n,2i).
\end{equation}
The following two lemmas show that $\InnermostCupComplex_{*}(2n, 2i)$ is close to being a complex of projective left $\TL_{2n}$-modules.

\begin{lemma} \label{lemma:forgetting cups}
    The map 
		$$\operatorname{Cup}(F)(2n,2i) \to \TL(2n,2i-2|F|)$$ 
		which deletes all dashed innermost right cups is an isomorphism of left $\TL_{2n}$-modules.
\end{lemma}

\begin{proof} The deletion map is by construction a map of $R$-modules, so we must verify that it is a $\TL_{2n}$-linear bijection. Multiplying by the diagram $R_{k}$ of \cref{defn: diagrams Lk Rk and Lmax} (with its single cup dashed) is inverse to the map which deletes the cup from $k$ to $k+1$. The map $-\cdot R_{k}$ is a right multiplication map and hence automatically left $\TL_{2n}$-linear. Deletion of all cups is a composite of deletions of single cups, and is therefore an isomorphism as required.\end{proof}

\begin{lemma}
    \label{lemma:left-projective-TL-modules}
    If $0 < i \leq n$ or $n = i =0$, the $\TL_{2n}$-module $\TL(2n,2i)$ can be identified with a retract of $\TL(2n,2n) = \TL_{2n}$ and is hence projective.
\end{lemma}

\begin{proof} If $n=i=0$, the result is trivial. Using the forgetting cups map of \cref{lemma:forgetting cups}, $\TL(2n,2i)$ is identified with the cup module $\operatorname{Cup}(F)(2n,2n)$ for $F$ the set of innermost left cups joining each of $2i+1,2i+3, \dots, 2n-1$ to the node directly above them. When $i>0$, this latter cup module was shown to be a retract of $\TL_{2n}$ by right multiplication $- \cdot E\colon \TL_{2n} \twoheadrightarrow \operatorname{Cup}(F)(2n,2n)$ with a specific (idempotent) Temperley--Lieb diagram $E \in \TL_{2n}$ \cite[Proposition 13 and Corollary 15]{Sroka}.
\end{proof}

Putting these lemmas together gives the following corollary.

\begin{corollary}\label{cor:InnIsProj}
The left $\TL_{2n}$-module $\InnermostCupComplex_{q}(2n, 2i)$ is projective for $0 \leq q < i$.\qed
\end{corollary}

\subsubsection{Submaximal cups}

Our second complex is a subcomplex of the complex of innermost cups, defined using the unique element $M \in \mathcal{I}_{2i}^i$, i.e.~the unique collection of dashed innermost cups of maximal cardinality $i$. We call a collection $F$ of dashed innermost cups \emph{submaximal} if it is obtained from $M$ by omitting some cups, or equivalently if every dashed innermost cup connecting $j<j+1$ has $j$ odd (and $j+1$ even).

\begin{definition}[Complex of submaximal cups] 
    \label{definition: complex of maximal innermost cups}
    Let $n \geq i$.
    The \emph{complex of submaximal cups} for the cell module $\cell(2n,2i)$ is the sub-chain-complex $\SubmaxInnermostCupComplex_*(2n,2i)$ of the complex of innermost cups $\InnermostCupComplex_*(2n,2i)$ on those diagrams having a submaximal set of dashed cups. It is therefore given in degree $q \geq 1$ by
    $$\SubmaxInnermostCupComplex_q(2n,2i) \coloneqq \bigoplus_{\substack{F \in \mathcal{I}_{2i}^q\\ \text{ submaximal}}} \operatorname{Cup}(F)(2m,2i)$$
    and in degree $q = 0$ it is given by
    $$\SubmaxInnermostCupComplex_0(2n,2 i) \coloneqq \im(\delta_1: \SubmaxInnermostCupComplex_1(2n, 2i) \to \InnermostCupComplex_0(n, 2i) = \TL(2n,2i))$$
    i.e.~$\SubmaxInnermostCupComplex_0(2n,2 i)$ is the submodule of $\TL(2n,2i)$ consisting of diagrams that share an innermost right cup with $M$.
\end{definition}

\subsubsection{Outermost cups}
The final combinatorial complex we define is built from the cell modules $S(2n,2i)$, and requires us to set the Temperley--Lieb parameter $a$ to be $0$. An \emph{outermost left cup} in a diagram in $\TL(2n,0)$ refers to an arc connecting nodes on the left hand side, which is not nested inside any other arc connecting nodes on the left hand side. In formulas, an arc connecting $i < j$ is outermost if there is no arc connecting $i' < j'$ for $i' < i$ and $j < j'$. Outermost left cups in a given diagram cannot be nested, so they totally ordered by the position of (either of) the nodes they join.

\begin{definition}[Complex of outermost cups]\label{defn:outermostcups}
Let $a=0$. Let $\OutermostCupComplex_q(2n)$ be the free $R$-module with basis the diagrams in $\TL(2n,0)$ equipped with $q$ dashed outermost left cups. This has the structure of a left $\TL_{2n}$-module by concatenating  diagrams, dashing every arc which has a dashed portion, and setting the result to zero if there are now fewer than $q$ dashed outermost left cups. The collection of these form a semi-simplicial $\TL_{2n}$-module, where the $i$-th face map un-dashes the $i$-th dashed outermost cup (the dashed cups are ordered from bottom to top). The alternating sum of these face maps defines a differential
$$\partial_{q}\colon \OutermostCupComplex_q(2n) \to \OutermostCupComplex_{q-1}(2n).$$
This is the	\emph{complex of outermost cups}, denoted $\OutermostCupComplex_*(2n)$.
\end{definition}

\begin{remark}
    The assumption that the parameter $a$ be zero in \cref{defn:outermostcups} is used only in the claim that $\OutermostCupComplex_*(2n)$ is a semi-simplicial $\TL_{2n}$-module: un-dashing the $i$-th dashed outermost cup is not a map of $\TL_{2n}$-modules if $a \neq 0$ (see \cref{remark:outermostcups-assumption-on-a} for additional details). In contrast, the $\TL_{2n}$-modules $\OutermostCupComplex_q(2n)$ can be defined for any parameter: our next \cref{lem:OutIsCell} shows that it is the cell-module $S(2n,2q)$ in \cref{definition:cell-modules}.
\end{remark}

\begin{figure}[ht]
    \begin{tikzpicture}[scale=0.4, baseline=(base)]
               \coordinate (base) at (0,3.25);
        \draw[gray,line width = 1](0,0.5)--(0,8.5);

        \foreach \x in {1,2, 3,4,5,6,7,8}{
            \draw[fill=black]  (0,\x) circle [radius=0.15];
   
        } 

        \draw (0,2) to[out=0,in=-90] (1,2.5) to[out=90,in=0] (0,3);
        \draw (0,4) to[out=0,in=-90] (1,4.5) to[out=90,in=0] (0,5);
        \draw[dashed] (0,7) to[out=0,in=-90] (1,7.5) to[out=90,in=0] (0,8);
        \draw[dashed] (0,1) to[out=0,in=-90] (3,3.5) to[out=90,in=0] (0,6);
    \end{tikzpicture}
    \quad
    $\mapsto$
    \quad
    \begin{tikzpicture}[scale=0.4, baseline=(base)]
               \coordinate (base) at (0,3.25);
        \draw[gray,line width = 1](0,0.5)--(0,8.5);

        \foreach \x in {1,2, 3,4,5,6,7,8}{
            \draw[fill=black]  (0,\x) circle [radius=0.15];
        } 

        \draw (0,2) to[out=0,in=-90] (1,2.5) to[out=90,in=0] (0,3);
        \draw (0,4) to[out=0,in=-90] (1,4.5) to[out=90,in=0] (0,5);
        \draw[dashed] (0,7) to[out=0,in=-90] (1,7.5) to[out=90,in=0] (0,8);
        \draw(0,1) to[out=0,in=-90] (3,3.5) to[out=90,in=0] (0,6);
    \end{tikzpicture}
    \quad
    $-$
    \quad
    \begin{tikzpicture}[scale=0.4, baseline=(base)]
               \coordinate (base) at (0,3.25);
        \draw[gray,line width = 1](0,0.5)--(0,8.5);
        
        \foreach \x in {1,2, 3,4,5,6,7,8}{
            \draw[fill=black]  (0,\x) circle [radius=0.15];
        } 

        \draw (0,2) to[out=0,in=-90] (1,2.5) to[out=90,in=0] (0,3);
        \draw (0,4) to[out=0,in=-90] (1,4.5) to[out=90,in=0] (0,5);
        \draw(0,7) to[out=0,in=-90] (1,7.5) to[out=90,in=0] (0,8);
        \draw[dashed] (0,1) to[out=0,in=-90] (3,3.5) to[out=90,in=0] (0,6);
    \end{tikzpicture}
    .
    \caption{The differential $\partial_2\colon\OutermostCupComplex_2(8) \to \OutermostCupComplex_1(8)$ of the complex of innermost cups; each face maps un-dashes a dashed cup.}  \label{fig: outermost cup complex}
\end{figure}

Consider the unique basis element $C_q \in \OutermostCupComplex_q(2q)$ (this is $\Lmax\in \TL(2q,0)$ with all cups dashed). Right multiplication with the diagram $C_q$ induces a ``close all cups'' left $\TL_{2n}$-module map
$$ - \cdot C_q \colon \TL(2n,2q) \to \OutermostCupComplex_q(2n).$$
\begin{lemma}\label{lem:OutIsCell}
The ``close all cups'' map $ - \cdot C_q$ induces left $\TL_{2n}$-module isomorphisms $ \cell(2n,2q)\cong \OutermostCupComplex_q(2n)$.
\end{lemma}
\begin{proof}
If a diagram $D \in \TL(2n,2q)$ has a right cup, then it has an innermost right cup. In the diagram $D \cdot C_q$ either (i) this innermost right cup connected $2i+1$ to $2i+2$ so right multiplication by $C_q$ creates a dashed circle, and hence there are $<q$ dashed outermost arcs so it gives zero, or (ii) it connected $2i$ to $2i+1$ so right multiplication by $C_q$ joins two dashed arcs, and hence there are again $<q$ dashed outermost arcs so it again gives zero. It follows that this map descends to a left $\TL_{2n}$-module map
$$\cell(2n,2q) \to \OutermostCupComplex_q(2n).$$
This induces a bijection on diagram bases: the inverse is given by cutting open the dashed outermost left cups, and connecting to nodes on the right, which always gives diagrams having no right cups.
\end{proof}

An example is shown in \cref{fig: close cups}.

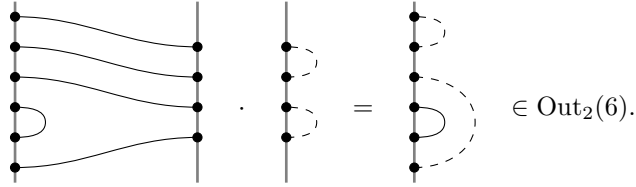
\begin{figure}[ht]
    \begin{tikzpicture}[scale=0.4, baseline=(base)]
               \coordinate (base) at (0,2.75);
        \draw[gray,line width = 1](0,0.5)--(0,6.5);
        \draw[gray,line width = 1](6,0.5)--(6,6.5);

        \foreach \x in {1,2, 3,4,5,6}{
            \draw[fill=black]  (0,\x) circle [radius=0.15];
        } 

        \foreach \y in {1,2,3,4}{
            \draw[fill=black]  (6,\y + 1) circle [radius=0.15];
    
        } 

        \draw (0,1) to[out=0,in=180] (6,2);
        \draw (0,4) to[out=0,in=180] (6,3);
        \draw (0,5) to[out=0,in=180] (6,4);
        \draw (0,6) to[out=0,in=180] (6,5);

        \draw (0,2) to[out=0,in=-90] (1,2 + 0.5) to[out=90,in=0] (0,2 + 1);
    \end{tikzpicture}
    \quad
    $\cdot$
    \quad
        \begin{tikzpicture}[scale=0.4, baseline=(base)]
              \coordinate (base) at (0,2.75);
        \draw[gray,line width = 1](0,0.5)--(0,6.5);

        \foreach \x in {2, 3,4,5}{
            \draw[fill=black]  (0,\x) circle [radius=0.15];
        } 

        \draw[dashed] (0,2) to[out=0,in=-90] (1,2.5) to[out=90,in=0] (0,3);
        \draw[dashed] (0,4) to[out=0,in=-90] (1,4.5) to[out=90,in=0] (0,5);
        
    \end{tikzpicture}
    \quad
    $=$
    \quad
    \begin{tikzpicture}[scale=0.4, baseline=(base)]
              \coordinate (base) at (0,2.75);
        \draw[gray,line width = 1](0,0.5)--(0,6.5);

        \foreach \x in {1,2, 3,4,5,6}{
            \draw[fill=black]  (0,\x) circle [radius=0.15];
        }

        \draw[dashed] (0,5) to[out=0,in=-90] (1,5.5) to[out=90,in=0] (0,6);
        \draw[dashed] (0,1) to[out=0,in=-90] (2,2.5) to[out=90,in=0] (0,4);

        \draw (0,2) to[out=0,in=-90] (1,2 + 0.5) to[out=90,in=0] (0,2 + 1);
    \end{tikzpicture}
    \quad
    $\in \OutermostCupComplex_2(6)$.
    \caption{The ``close all cups'' map $\cell(6,4) \to \OutermostCupComplex_2(6)$, which is multiplication by the diagram $C$ (dashed $\Lmax$) of \cref{lem:OutIsCell}.}  \label{fig: close cups}
\end{figure}

\begin{remark} \label{remark:outermost in zero row}
    $\OutermostCupComplex_*(0)$ vanishes in positive degrees, and in degree 0 it is a free $R$-module of rank 1: $S(0,0) \cong \TL(0,0) = R \{\emptyset\}$.
\end{remark}

\begin{remark}
    \label{remark:outermostcups-assumption-on-a}
    For any parameter $a$, the morphism modules $ \TL(2n, 2q)$ for $0 \leq q \leq n$ assemble to a semi-simplicial $\TL_{2n}$-module $\TL(2n,2*)$ with $i$-th face map given by $d_i\colon \TL(2n,2q) \to \TL(2n,2q-2)\colon A \mapsto A \cdot L_{2i+1}$. We have seen in \cref{lem:OutIsCell} that the map $- \cdot C_q \colon \TL(2n,2q) \twoheadrightarrow \OutermostCupComplex_q(2n)$ is a surjection with kernel the submodule spanned by diagrams that have a cup on the right. If $a = 0$ (as in \cref{defn:outermostcups}), then this submodule is closed under the face maps, so the semisimplicial structure descends to give the desired semisimplicial structure on $\OutermostCupComplex_q(2n)$.
\end{remark}

\subsection{Connectivity results} \label{subsection: MV complexes}
 In this section we will prove that the complexes introduced above are acyclic, using the `Mayer--Vietoris' language of \cite[Section 2]{BoydeTwo}. We will briefly recall the general set up and then apply it to our complexes.

Let $A$ be an $R$-algebra, and let $N \subset M$ be $A$-modules. Let $N_1, \dots, N_w$ be submodules of $N$ which cover $N$ in the sense that $N_1 + \dots + N_w = N$. To this data we may associate a (\emph{Mayer--Vietoris}) chain complex $\widetilde{\MayerVie}_*$, by setting $\widetilde{\MayerVie}_{-1} = M/N$, $\widetilde{\MayerVie}_0 = M$, and
$$\widetilde{\MayerVie}_p = \bigoplus_{\substack{S \subset \underline{w} \\ |S| = p}} \bigcap_{i \in S} N_i$$
for $p = 1, \dots , w$. The differential $d_p\colon\widetilde{\MayerVie}_p \to \widetilde{\MayerVie}_{p-1}$ is a sum of inclusions with signs: on $\bigcap_{i \in S} N_i \subset \MayerVie_p$ it is given by the formula
\begin{align*}
    \bigcap_{i \in S} N_i & \to \bigoplus_{j \in S} \bigcap_{i \in S \setminus j} N_i \\
    x & \mapsto \oplus_{j \in S} (-1)^{\gamma_S(j)} x,
\end{align*}
where $\gamma_S(j) = |\{i \in S \mid i < j \}|$. That this defines a chain complex ($d^2 = 0$) is verified as \cite[Lemma 2.1]{BoydeTwo}. The key result is the following.

\begin{lemma}[{\cite[Lemma 2.2]{BoydeTwo}}]
    \label{lemma: MV acyclicity} If $N$ is free as an $R$-module on some basis $\mathcal{B}$, such that the $N_i$ are free submodules on subsets of $\mathcal{B}$, then the Mayer--Vietoris complex $\widetilde{\MayerVie}_*$ is acyclic. \qed
\end{lemma}

We now apply this set-up to our complexes.

\begin{proposition}[{c.f.~\cite[Theorem 10]{Sroka}}]
    \label{proposition:complexes-of-innermost-cups-are-acyclic}
     $\AugmentedInnermostCupComplex_*(2n, 2i)$ is acyclic.
\end{proposition}

\begin{proof}
This complex is the Mayer--Vietoris complex associated to the data given by setting $A = \TL_{2n}$, $M = \TL(2n,2i)$, $N = \im(\delta_1) = \ker(\delta_0)$, and $N_k = \operatorname{Cup}(F_k)(2n,2i)$ for $1 \leq k \leq 2i-1$, where $F_k$ is the single dashed innermost right cup joining $k$ to $k+1$. It is clear that the $N_k$ cover $N$, and to see the identification on the chain level, note that an intersection $\bigcap_{k \in S} N_k$ is zero unless $S$ is the set of lowest nodes of a collection $F$ of innermost right cups, in which case $\bigcap_{k \in S} N_k = \operatorname{Cup}(F)(2n,2i)$. This complex satisfies the hypothesis of \cref{lemma: MV acyclicity} by construction, and the result follows.
\end{proof}

\begin{remark}
    \label{remark:innermostcupcomplex-0-i-is-acyclic}
    By \cref{remark:cellmodule-0-i-is-trivial}, $\AugmentedInnermostCupComplex_*(0,2i) = \InnermostCupComplex_*(0,2i)$, so \cref{proposition:complexes-of-innermost-cups-are-acyclic} also shows that $\InnermostCupComplex_*(0,2i)$ is acyclic.
\end{remark}

\begin{proposition}
    \label{lemma:outermostcupcomplex-acyclic}
    $\OutermostCupComplex_*(2n)$ is acyclic for $n \geq 1$.
\end{proposition}

\begin{proof} 
For this argument we forget the $\TL_{2n}$-module structure. For a collection $F$ of $q$ outermost left cups, let $\OutermostCupComplex_q(2n)^F$ be the $R$-submodule of $\OutermostCupComplex_q(2n)$ consisting of those diagrams whose associated dashed outermost left cups are precisely $F$. Writing $\mathcal{O}^q_{2n}$ for the set of such $F$'s, there is an $R$-module decomposition
$$\OutermostCupComplex_q(2n) = \bigoplus_{F \in \mathcal{O}^q_{2n}} \OutermostCupComplex_q(2n)^F.$$

Then $\OutermostCupComplex_*(2n)$ is the Mayer--Vietoris complex associated to $N = M = \TL(2n,0) = \cell(2n,0)$ and the collection of submodules $N_{C} = \OutermostCupComplex_1(2n)^{C}$ with $C \in \mathcal{O}^1_{2n}$ a single left cup. Since $n \geq 1$, every diagram in $\TL(2n,0)$ has at least one outermost cup, so this collection is indeed a cover. To finish the identification with the Mayer--Vietoris complex we note the intersection $\bigcap_{C \in S} N_C$ is zero unless $F := \bigcup_{C \in S} C$ can be realised as outermost cups in some planar diagram, in which case $\bigcap_{C \in S} N_C = \OutermostCupComplex_{|S|}(2n)^F$. This complex satisfies the hypothesis of \cref{lemma: MV acyclicity} by construction, and the result follows. 
 \end{proof}

 \begin{remark}
    \label{remark:submaximal-is-mayer-vietoris}
     We remark that the complex of submaximal cups $\SubmaxInnermostCupComplex_*(2n,2i)$ is also a Mayer--Vietoris complex with $M = N = \SubmaxInnermostCupComplex_0(2n, 2i) \subseteq \TL(2n,2i)$ and hence acyclic. However, we shall not need this.
 \end{remark}

\section{Interactions between our complexes}

Recall that to define the outermost cup complex, we assumed that the Temperley--Lieb parameter $a=0$. This assumption will hold for this section.

Our motivation is as follows: the complex of outermost cups $\OutermostCupComplex_*(2n)$ is an acyclic complex of left $\TL_n$-modules:
$$S(2n,0) \cong \OutermostCupComplex_{0}(2n) \twoheadleftarrow \OutermostCupComplex_{1}(2n) \leftarrow \OutermostCupComplex_{2}(2n) \leftarrow \dots \leftarrow \OutermostCupComplex_{n}(2n) \cong R.$$ We would like to use this to resolve $S(2n,0)$, but its terms are not projective. However the acyclic complex $\AugmentedInnermostCupComplex_*(2n,2i)$ resolves $S(2n,2i) \cong \OutermostCupComplex_{i}(2n)$ (\cref{lem:OutIsCell}) and by \cref{lemma:left-projective-TL-modules} the terms $\TL(2n,2i)$ in this complex are projective when $0 < i \leq n$ or $n = i =0$.

The aim of this section is to fit these complexes together by extending the differential $\partial_i\colon\OutermostCupComplex_i(2n) \to \OutermostCupComplex_{i-1}(2n)$ of $\OutermostCupComplex_*(2n)$ to a map of chain complexes $\InnermostCupComplex_*(2n, 2i) \to \InnermostCupComplex_*(2n, 2i-2).$ Since $\OutermostCupComplex_*(2n)$ is a semisimplicial module, it suffices to show that each face map $d_k\colon \OutermostCupComplex_i(2n) \to \OutermostCupComplex_{i-1}(2n)$ making up $\partial_i$ can likewise be extended to a map \[
    \widetilde{d_k} \colon \InnermostCupComplex_*(2n, 2i) \to \InnermostCupComplex_*(2n, 2i-2).
\]

Recall that $d_k$ forgets the $k$-th dashed outermost cup, or, under the identification $S(2n,2i) \cong \OutermostCupComplex_{i}(2n)$ of \cref{lem:OutIsCell}, multiplies a diagram $D$ on the right by $L_{2k+1}$.

To solve the lifting problems, we will use the semisimplicial structure on $\InnermostCupComplex_*(2n, 2i)$, forgetting the one on $\OutermostCupComplex_*(2n)$, so for the rest of this section we suggest that the reader just think of $d_k$ as a map of $\TL_{2n}$-modules, forgetting that it is part of a semisimplicial structure.

It is easy to define {some} extension; the point is to construct one that can be used for {explicit} calculations. Recall in particular that \cref{lem:OutIsCell} already gives us a commutative diagram (c.f.~\cref{remark:outermostcups-assumption-on-a}):
\begin{center}
\begin{tikzcd}
    \InnermostCupComplex_0(2n,2i) \arrow[r, equal] \arrow[d] & \TL(2n,2i) \arrow[r, two heads] \arrow[d, "d_k\colon D \mapsto D \cdot L_{2k+1}"] & \OutermostCupComplex_i(2n) \arrow[d, "d_k"]\\
    \InnermostCupComplex_0(2n,2i-2) \arrow[r, equal] & \TL(2n,2i-2) \arrow[r, two heads] & \OutermostCupComplex_{i-1}(2n)
\end{tikzcd}
\end{center}
and hence a definition of the map $\widetilde{d_k}$ in degree zero. To construct $\widetilde{d_k}$ in higher degrees, we have to think about the effect of left multiplication by $L_{2k+1}$ on diagrams with marked innermost right cups. In general this is complicated. However, for our purposes it will suffice to understand this when the the marked cups are submaximal, in which case it is not so bad.

We will therefore define an \emph{explicit} extension on the complex of submaximal cups, and then argue that this extends to the whole innermost cup complex, without being explicit about the extension.

Recall from \cref{definition: complex of maximal innermost cups} that the complex of submaximal cups $\SubmaxInnermostCupComplex_*(2n,2i)$ is the subcomplex of $\InnermostCupComplex_*(2n,2i)$ consisting of diagrams where every dashed innermost cup connecting $j<j+1$ has $j$ odd even.

The observation which tells us how to define $\widetilde{d_k}$ on this subcomplex is the following.

\begin{lemma} \label{lem: multiplying marked guys by L2k+1}
    Let $F$ be a set of dashed innermost cups, and let $D$ be a diagram in $\TL(2n, 2i)$ having all cups from $F$.
    \begin{enumerate}
        \item If $F$ contains the cup connecting $2k+1$ to $2k+2$, then $D \cdot L_{2k+1} = 0 \in \TL(2n,2i-2)$.
        \item If $F$ is submaximal and does not contain this cup, then $D \cdot L_{2k+1}$ has all cups in the set $F(2k+1)$, where $F(2k+1)$ denotes the set of innermost right cups (now in $\TL(2n,2i-2)$) obtained from $F$ by deleting the nodes $2k+1$ and $2k+2$ on the right (and renumbering accordingly).
    \end{enumerate}
\end{lemma}

\begin{proof} Part $(1)$ is because $a=0$. Part $(2)$ is because a submaximal set not including the cup connecting $2k+1$ to $2k+2$ cannot include any cups involving $2k+1$ or $2k+2$, so the marked cups are preserved by the multiplication (\cref{fig: order-preserving correspondence}). \end{proof}

\begin{definition} \label{def: face map lift to submax} Let $$\widetilde{d_k}\colon \SubmaxInnermostCupComplex_*(2n, 2i) \to \SubmaxInnermostCupComplex_*(2n, 2i-2)$$ be the map of chain complexes which is:
\begin{enumerate}
    \item zero on diagrams having a dashed cup from $2k+1$ to $2k+2$, and
    \item acts on diagrams $D$ with set of dashed cups $F$ not containing the cup from $2k+1$ to $2k+2$ as $D \mapsto D \cdot L_{2k+1}$, where the latter is regarded as having dashed cups indexed by the set $F(2k+1)$ of \cref{lem: multiplying marked guys by L2k+1} (see \cref{fig: order-preserving correspondence}).
\end{enumerate}

\end{definition}

To see that this is indeed a map of chain complexes, note that by \cref{lem: multiplying marked guys by L2k+1} it respects the face maps (which are given by forgetting dashed cups). Note that it agrees with the extension already constructed in degree 0.

\begin{figure}[ht]
    \begin{tikzpicture}[scale=0.4, baseline=(base)]
               \coordinate (base) at (0,3.25);
        \draw[gray,line width = 1](0,0.5)--(0,8.5);
        \draw[gray,line width = 1](6,0.5)--(6,8.5);

        \foreach \x in {1,2, 3,4,5,6,7,8}{
            \draw[fill=black]  (0,\x) circle [radius=0.15];
        } 

        \foreach \y in {1,2,3,4,5,6}{
            \draw[fill=black]  (6,\y + 1) circle [radius=0.15];
        } 

        \draw (0,1) to[out=0,in=180] (6,4);
        \draw (0,6) to[out=0,in=180] (6,5);

        \draw (0,7) to[out=0,in=-90] (1,7.5) to[out=90,in=0] (0,8);
        \draw (0,2) to[out=0,in=-90] (1,2.5) to[out=90,in=0] (0,3);
        \draw (0,4) to[out=0,in=-90] (1,4.5) to[out=90,in=0] (0,5);

        \draw[dashed] (6,2) to[out=180,in=-90] (5,2.5) to[out=90,in=180] (6,3);
        \draw[dashed] (6,6) to[out=180,in=-90] (5,6.5) to[out=90,in=180] (6,7);
    \end{tikzpicture}
    \quad
    $\cdot$
    \quad
    \begin{tikzpicture}[scale=0.4, baseline=(base)]
               \coordinate (base) at (0,3.25);
        \draw[gray,line width = 1](0,0.5)--(0,8.5);
        \draw[gray,line width = 1](6,0.5)--(6,8.5);

        \foreach \x in {1,2, 3,4,5,6}{
            \draw[fill=black]  (0,\x+1) circle [radius=0.15];
        } 

        \foreach \y in {1,2,3,4}{
            \draw[fill=black]  (6,\y + 2) circle [radius=0.15];
        } 

        \draw (0,2) to[out=0,in=180] (6,3);
        \draw (0,3) to[out=0,in=180] (6,4);
        \draw (0,6) to[out=0,in=180] (6,5);
        \draw (0,7) to[out=0,in=180] (6,6);
        
        \draw(0,4) to[out=0,in=-90] (1,4.5) to[out=90,in=0] (0,5);
    \end{tikzpicture}
    \quad
    $=$
    \quad
        \begin{tikzpicture}[scale=0.4, baseline=(base)]
               \coordinate (base) at (0,3.25);
        \draw[gray,line width = 1](0,0.5)--(0,8.5);
        \draw[gray,line width = 1](6,0.5)--(6,8.5);

        \foreach \x in {1,2,3,4,5,6,7,8}{
            \draw[fill=black]  (0,\x) circle [radius=0.15];
        } 

        \foreach \y in {1,2,3,4}{
            \draw[fill=black]  (6,\y + 2) circle [radius=0.15];
        }  

        \draw (0,7) to[out=0,in=-90] (1,7.5) to[out=90,in=0] (0,8);
        \draw (0,2) to[out=0,in=-90] (1,2.5) to[out=90,in=0] (0,3);
        \draw (0,4) to[out=0,in=-90] (1,4.5) to[out=90,in=0] (0,5);
        \draw (0,1) to[out=0,in=-90] (3,3.5) to[out=90,in=0] (0,6);

        \draw[dashed] (6,3) to[out=180,in=-90] (5,3.5) to[out=90,in=180] (6,4);
        \draw[dashed] (6,5) to[out=180,in=-90] (5,5.5) to[out=90,in=180] (6,6);
    \end{tikzpicture}
    .
    \caption{The map $\widetilde{d_2}\colon \SubmaxInnermostCupComplex_2(8, 6) \to \SubmaxInnermostCupComplex_2(8, 4)$. Forgetting dashings, this map is given by right-multiplication by $L_3 \in \TL(6,4)$  (just like $d_2$). On dashings, it is the map $F \mapsto F(3)$ of \cref{lem: multiplying marked guys by L2k+1}.}  \label{fig: order-preserving correspondence}
\end{figure}
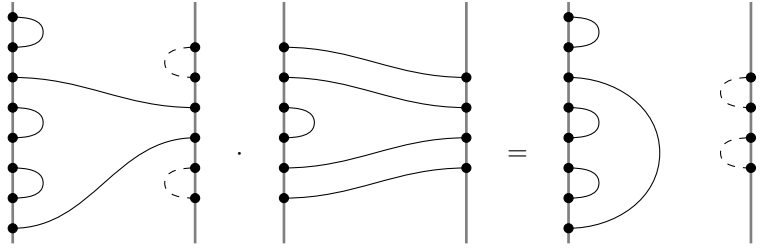

\begin{remark} The following description from the point of view of coverings may be helpful. Recall from \cref{subsection: MV complexes} that a Mayer--Vietoris complex takes as input the data of an $R$-algebra $A$, with $A$-modules $N \subset M$, and submodules $N_1, \dots, N_w$ of $N$ which form a cover of $N$ in the sense that $N_1 + \dots + N_w = N$. Suppose that we are given another such datum: $N' \subset M' $, and $N'_1, \dots , N'_{w'}$ with $N'_1 + \dots + N'_{w'} = N'$, and a map of pairs of $A$-modules $f\colon (M,N) \to (M',N')$. In order to extend $f$ to a map of Mayer--Vietoris complexes, we need to provide the data of a morphism of sets $g\colon \{1,\dots,w\} \to \{1, \dots w'\}$ so that for each $i \in \{1,\dots,w\}$, $f$ restricts to a map $N_{i} \to N'_{g(i)}$, i.e.~a lift of $f$ to a map of coverings. This is precisely what we do above for submaximal cup complexes, c.f.~\cref{remark:submaximal-is-mayer-vietoris} (note: $g$ only need be defined where $f$ is nonzero).
\end{remark}

We now argue that for abstract reasons there exists some extension of $\widetilde{d_k}$ over the whole innermost cup complex.

\begin{lemma} \label{lemma: new extension from submaximal}
    The map $\widetilde{d_k}\colon \SubmaxInnermostCupComplex_*(2n, 2i) \to \SubmaxInnermostCupComplex_*(2n, 2i-2)$ can be extended to a map $$\widetilde{d_k}\colon\InnermostCupComplex_*(2n, 2i) \to \InnermostCupComplex_*(2n, 2i-2)$$ of innermost cup complexes.
\end{lemma}

\begin{proof} 
If $q \geq 1$, the $q$-th chain module $\InnermostCupComplex_q(2n, 2i)$ can be decomposed as 
\[
    \InnermostCupComplex_q(2n, 2i) = \SubmaxInnermostCupComplex_q(2n, 2i) \oplus R_q(2n, 2i)
\]
where $R_q(2n, 2i)$ is the direct sum of $\operatorname{Cup}(F)(2n,2i)$ for innermost sets $F$ which are not submaximal. This excludes in particular the maximal innermost set $M=\{1,3, \dots, 2i-1\}$, so by \cref{lemma:forgetting cups,lemma:left-projective-TL-modules}, each term $R_q(2n, 2i)$ is a direct sum of projective $\TL_{2n}$-modules, and hence projective. We must define our extension on the complement $R_q(2n, 2i)$ of $\SubmaxInnermostCupComplex_q(2n, 2i)$ in each degree. This is an induction on degree using that each $R_q(2n, 2i)$ is projective, and that the codomain $\InnermostCupComplex_*(2n,2i-2)$ is acyclic.
\end{proof}

\section{The derived complexes of innermost and outermost cups}

In this section we explain how to set up derived versions of the complexes of innermost and outermost cups. We explore algebraic properties and consequences, including the connection between the dga of planar loops and the homology of Temperley--Lieb algebras stated in \cref{theorem: shift}.

We will use the general construction described in \cref{definition: dg modules general}.

\begin{definition} \label{def: derived innermost and cell}
The \emph{augmented derived complex of innermost cups}
$$0 \leftarrow \derivedCell(2n,2i) \leftarrow \derivedInnermost_0(2n,2i) \leftarrow \derivedInnermost_1(2n,2i) \leftarrow \cdots \leftarrow \derivedInnermost_i(2n,2i) \leftarrow 0$$
is the complex of $\CPL(2n)$-modules defined by applying the construction $D$ of \cref{definition: dg modules general} levelwise to the chain complex data $X_* = \AugmentedInnermostCupComplex_{*}(2n,2i)$, $Y_* = \AugmentedInnermostCupComplex_{*}(0,2i)$ and the maps $\phi \colon \TL(0,2n) \otimes_{\TL_{2n}} \AugmentedInnermostCupComplex_{*}(2n,2i) \to \AugmentedInnermostCupComplex_{*}(0,2i)$ and $\psi\colon \TL(2n,0) \otimes_{R} \AugmentedInnermostCupComplex_{*}(0,2i) \to \AugmentedInnermostCupComplex_{*}(2n,2i)$
 induced by multiplication of diagrams.

In particular the $(-1)$-th term in this complex is the \emph{derived cell module} $\derivedCell(2n,2i)$ obtained by applying the construction $D$ to  $\cell(2n,2i)$ and the map $\TL(0,2n) \otimes_{\TL_{2n}} \cell(2n,2i) \to \cell(0,2i)$ induced by multiplication of diagrams.
\end{definition}

This definition can be spelled out as follows. For $p \geq 0$, the $p$-th term $\derivedInnermost_p(2n,2i)$ is given in degree $q$ by
$$\derivedInnermost_p(2n,2i)_q =  \begin{cases}
    \TL(0,2n) \otimes_R \TL_{2n}^{\otimes q-1} \otimes_R \InnermostCupComplex_{p}(2n,2i) & q > 0 \\
    \InnermostCupComplex_{p}(0,2i) & q=0.
\end{cases}$$

The $(p=-1)$-term $\derivedCell(2n,2i)$ 
is given in degree $q$ by
$$\derivedCell(2n,2i)_q = \begin{cases} \TL(0,2n) \otimes_R \TL_{2n}^{\otimes q-1} \otimes_R \cell(2n,2i) & q>0\\
\cell(0,2i) & q=0.
\end{cases}$$

The role that (most of) the complexes $\derivedInnermost_p(2n,2i)$ will play is suggested by the following result.

\begin{lemma}\label{lem:DInnAcyclic}
For $0 \leq p < i$, the complex $\derivedInnermost_p(2n,2i)$ is acyclic if $n = 0$ or $i \leq n$.
\end{lemma}
\begin{proof}
For $0 \leq p < i$, the left $\TL_{2n}$-module $\InnermostCupComplex_p(2n,2i)$ is projective by \cref{cor:InnIsProj} and so the map $$\Bar(\TL(0,2n), \TL_{2n}, \InnermostCupComplex_p(2n,2i)) \to \TL(0,2n) \otimes_{\TL_{2n}} \InnermostCupComplex_p(2n,2i)$$ is a quasiisomorphism. It follows that
$$\derivedInnermost_p(2n,2i) \simeq \mathrm{Cone}(\TL(0,2n) \otimes_{\TL_{2n}} \InnermostCupComplex_p(2n,2i) \to \InnermostCupComplex_p(0,2i)).$$

To finish the argument we will show that $\TL(0,2n) \otimes_{\TL_{2n}} \InnermostCupComplex_p(2n,2i) \to \InnermostCupComplex_p(0,2i)$ is an isomorphism. Using the decomposition of the left $\TL_{2n}$-module $\InnermostCupComplex_p(2n,2i)$ into $\mathrm{Cup}(F)(2n,2i)$'s from \cref{eq:InnDecomposedIntoCup}, and the isomorphism of the latter with $\TL(2n,2j)$'s for $j = i - p > 0$ from \cref{lemma:forgetting cups}, we need to show that the map
$$A \otimes B \to A \cdot B \colon \TL(0,2n) \otimes_{\TL_{2n}} \TL(2n,2j) \to \TL(0,2j),$$
induced by composition in the Temperley--Lieb category, is an isomorphism.

If $n=0$, the statement is trivial. If $i \leq n$, it follows that $0 < j \leq n$. Under this assumption, the map in question can be identified with a retract of an isomorphism using \cref{lemma:left-projective-TL-modules}. More precisely, there is a commutative diagram
\begin{equation*}
\begin{tikzcd}
    \TL(0,2n) \otimes_{\TL_{2n}} \TL(2n,2j) \arrow[r] \arrow[d, shift right, hook] & \TL(0,2j) \arrow[d, shift right, hook]\\
    \TL(0,2n) \otimes_{\TL_{2n}} \TL(2n,2n) \arrow[r, "\cong"] \arrow[u, shift right, two heads] & \TL(0,2n) \arrow[u, shift right, two heads]
\end{tikzcd}
\end{equation*}
where the horizontal maps are given by composition and the vertical maps are retractions. Any algebra $A$ is a unit for $\otimes_A$, so the lower horizontal map is an isomorphism. The left vertical retraction is obtained from \cref{lemma:left-projective-TL-modules} using that functors, in this case $\TL(0,2n) \otimes_{\TL_{2n}}(-)$, preserve retractions. The right vertical retraction is a compatible retraction obtained from that in \cref{lemma:left-projective-TL-modules} by restriction. More precisely, the retraction in \cref{lemma:left-projective-TL-modules} restricts along any two embeddings of right $\TL_{2n}$-modules $\TL(0, 2n) \hookrightarrow \TL(2n,2n)$ and $\TL(0, 2j) \hookrightarrow \TL(2n, 2j)$ given by left multiplication with a basis diagram $D \in \TL(2n,0)$, because right multiplication by $E$ preserves left composites with $D$.
\end{proof}

When $a=0$, the modules $\cell(2n,2i)$ fit together into the complex of outermost cups (\cref{defn:outermostcups}), via $\OutermostCupComplex_q(2n) \cong \cell(2n,2q)$, and we may fit the derived cell modules together analogously. We will use that multiplication of diagrams defines maps
\begin{equation}\label{eq:SillyAction}
\TL(0,2n) \otimes_{\TL_{2n}} \OutermostCupComplex_{q}(2n) \cong \TL(0,2n) \otimes_{\TL_{2n}} \cell(2n,2q) \to \cell(0,2q),
\end{equation}
whose target is 0 for $q>0$, and $R$ for $q=0$.

\begin{definition} \label{def: derived outermost}
Assume that $a=0$. The \emph{derived complex of outermost cups}
$$0 \leftarrow  \derivedOutermost_0(2n) \xleftarrow{\partial_1} \derivedOutermost_1(2n) \xleftarrow{\partial_2}  \derivedOutermost_2(2n) \leftarrow \cdots \xleftarrow{\partial_n}  \derivedOutermost_n(2n) \leftarrow 0,$$
is the complex of $\CPL(2n)$-modules defined by applying the construction $D$ of \cref{definition: dg modules general} levelwise to the chain complexes of left $\TL_{2n}$-modules $X_*=\OutermostCupComplex_{*}(2n)$ and $Y_*=R$ concentrated in degree zero. Then the chain maps are $\phi: \TL(0,2n) \otimes_{\TL_{2n}} \OutermostCupComplex_{*}(2n) \to R$ (\cref{eq:SillyAction}) and $\psi: \TL(2n,0) \otimes_{R} R \to \OutermostCupComplex_{*}(2n)$ (use $R \cong \TL(0,0)$ and compose in $\TL$, landing in $\TL(2n,0) \cong \OutermostCupComplex_0(2n)$). It is then augmented by the usual augmentation
$$\epsilon \colon \derivedOutermost_0(2n) = L(2n) \to R = \TL(0,0)$$
given by multiplication of diagrams.
\end{definition}

With this definition, the isomorphisms $\OutermostCupComplex_q(2n) \cong \cell(2n,2q)$ of \cref{lem:OutIsCell} yield isomorphisms $\derivedOutermost_q(2n) \cong \derivedCell(2n,2q)$.

\begin{lemma} \label{lem: derived horizontally acyclic}
    The augmentation $\derivedInnermost_*(2n,2i) \to \derivedCell(2n,2i)$ is a quasiisomorphism. The augmentation $\derivedOutermost_*(2n) \to R$ is a quasiisomorphism (when these complexes are defined, i.e.~when $a = 0$).
\end{lemma}

\begin{proof} 
As $\TL(0,2n)$ and $\TL_{2n}$ are free $R$-modules, the functor
$$D(X, \phi) = \mathrm{Cone}(\phi_*\colon \Bar(\TL(0,2n), \TL_{2n}, X) \to Y)$$
preserves quasiisomorphisms in the data $(X,\phi)$, as the following both do
$$X \mapsto \Bar(\TL(0,2n), \TL_{2n}, X) \quad \text{ and } \quad Y \mapsto Y.$$

In the first case, note that both augmentations
$$\InnermostCupComplex_{*}(2n,2i) \to \cell(2n,2i) \quad \text{ and } \quad \InnermostCupComplex_{*}(0,2i) \to \cell(0,2i)$$
are quasiisomorphisms, by \cref{proposition:complexes-of-innermost-cups-are-acyclic} (see also \cref{remark:innermostcupcomplex-0-i-is-acyclic}).

In the second case, note that both augmentations
$$\OutermostCupComplex_{*}(2n) \to 0 \quad \text{ and } \quad R \to R$$
are quasiisomorphisms, for $n>0$ by \cref{lemma:outermostcupcomplex-acyclic} and for $n=0$ by observation.
\end{proof}

\subsection{Weight gradings and hook maps} \label{derived weights and hooks} We now discuss the interaction of the above constructions with the weight grading and hook map.

Recall, from \cref{def: weight}, that the weight grading on the module $\CPL(0,2n,2i)$ was defined by multiplying by the diagram $\Lmax \in \TL(2i,0)$ with $i$ cups between adjacent pairs of nodes, then counting loops. The derived cell module $\derivedCell(2n,2i)$ is the quotient of $\CPL(0,2n,2i)$ by those pinned diagrams having a right cup. These are weight-homogeneous, so the weight grading descends to $\derivedCell(2n,2i)$. 

We define a weight grading on each $\derivedOutermost_i(2n)$ by undashing all dashed outermost left cups and then counting loops.

\begin{lemma} \label{lem: weight on outermost}
The isomorphism $\derivedCell(2n,2i) \cong \derivedOutermost_i(2n)$ preserves the weight grading. The boundary maps $\partial_i$ in $\derivedOutermost_*(2n)$ preserve the weight grading.
\end{lemma}

\begin{proof} 
For the first part, the isomorphism $\cell(2n,2i) \overset{\sim}\to \OutermostCupComplex_i(2n)$ is induced by multiplication by $\Lmax \in \TL(2i,0)$, and dashing the newly-formed outermost cups. As the weight-grading on $\derivedCell(2n,2i)$ is defined by multiplication by $\Lmax$ and counting loops, and that on $\derivedOutermost_i(2n)$ is defined by neglecting dashings and counting loops, we see that they match up.

For the second part, the face maps in $\derivedOutermost_*(2n)$ are given by undashing one dashed outermost left cup, but we have defined the weight by undashing all dashed outermost left cups and counting loops.
\end{proof}

We define a weight grading on each $\derivedInnermost_p(2n,2i)$ by undashing all dashed innermost right cups, right multiplying by $\Lmax \in \TL(2i,0)$, then counting loops. The face maps on $\derivedInnermost_*(2n,2i)$ are given by undashing one dashed innermost right cup, so as in \cref{lem: weight on outermost} they preserve the weight grading.

A key point will be the following identification of the top term of the derived complex of innermost cups. Recall that we write bigraded suspensions with homological grading first and weight grading second.

\begin{lemma} \label{lemma: E_i E_i (B) 0}
    There is an isomorphism of left $\CPL(2n)$-modules $$\Sigma^{0,i} \CPL(2n) \cong \derivedInnermost_i(2n,2i).$$
\end{lemma}

\begin{proof}
There is an isomorphism $\TL(2n,0) \cong \InnermostCupComplex_i(2n,2i)$ of left $\TL_{2n}$-modules induced by right multiplying with the unique diagram in $\TL(0,2i)$ having $i$ innermost right cups, and dashing them all. This induces the isomorphism $\CPL(2n) \cong \derivedInnermost_i(2n,2i)$. The weight grading in $\derivedInnermost_i(2n,2i)$ is given by undashing these right innermost cups and multiplying by $\Lmax \in \TL(2i,0)$. In doing this we create $i$ loops, so the isomorphism shifts the weight grading by $i$.
\end{proof}

Recall the \emph{hook map} $$h\colon \CPL(0,2n,2i) \to \CPL(0,2n+2,2i)$$ of \cref{lemma:hook-map}. By definition, each cell module $\cell(2n,2i)$ is a quotient of $\TL(2n,2i)$ (\cref{definition:cell-modules}).
Recall that by the general construction of \cref{definition: dg modules general}, this induces a degree-wise quotient map
$$\CPL(0,2n,2i) \to \derivedCell(2n,2i),$$
(see \cref{def: dg modules L02n2i,def: derived innermost and cell} for the definitions of those modules via the general construction). It follows from the definition that the hook map $h$ descends along this quotient map to define a \emph{hook map} of derived cell modules $\derivedCell(2n,2i) \to \derivedCell(2n+2,2i)$, and that (under the identification of \cref{lem:OutIsCell}) this map respects the differentials in the derived outermost cup complexes, yielding a \emph{hook map}
$$h\colon \derivedOutermost_*(2n) \to \derivedOutermost_*(2n+2).$$

The following proposition will be one of the key ingredients in proving our main technical result, \cref{theorem:derived-outermost-cup-complex-resolution}. It also gives the connection between the dga of planar loops and the homology of Temperley--Lieb algebras stated in \cref{theorem: shift}.

\begin{proposition} \label{prop: suspension of dga is derived cell}
    There is a preferred equivalence of $\CPL(2n)$-modules
    $$\Sigma^i \CPL(2n) \simeq \derivedCell(2n,2i).$$
     In the setting where the additional weight grading is defined, the map is homogeneous of degree $i$, so we get a preferred equivalence of $\CPL(2n)$-modules
    $$\Sigma^{i,i} \CPL(2n) \simeq \derivedCell(2n,2i).$$
    These equivalences are compatible with the hook maps defined on each side.
\end{proposition}

\begin{proof}
Recall the resolution $\derivedInnermost_*(2n,2i)$ of $\derivedCell(2n,2i)$. Consider $\derivedInnermost_p(2n,2i)_q$ as a double complex, augmented by $\derivedCell(2n, 2i)$ in the $p=-1$ column. Because $\derivedInnermost_*(2n,2i) \to \derivedCell(2n, 2i)$ is an augmented complex of $L(2n)$-modules, this can be viewed as a chain complex of $L(2n)$-modules. By \cref{lem: derived horizontally acyclic} the totalisation of this double complex is acyclic. That is, there is a quasiisomorphism
$$\derivedCell(2n, 2i) \overset{\sim}\longleftarrow [\derivedInnermost_0(2n,2i) \leftarrow \derivedInnermost_1(2n,2i) \leftarrow \cdots \leftarrow \derivedInnermost_i(2n,2i)],$$
where we use bracket notation for totalisation as in \cref{section: small models}.

For $0 \leq p < i$, $\derivedInnermost_p(2n,2i)$ is acyclic by \cref{lem:DInnAcyclic}. For $p=i$ we have $\derivedInnermost_i(2n,2i) \cong \Sigma^{0,i} \CPL(2n)$ by \cref{lemma: E_i E_i (B) 0} (the shift of weight to be ignored if we are not working in a setting where weight is defined). The map of double complexes
\begin{equation*}
\begin{tikzcd}[column sep=2.2ex]
\derivedInnermost_0(2n,2i) \dar{\simeq}& \lar \derivedInnermost_1(2n,2i) \dar{\simeq} & \lar \cdots & \derivedInnermost_{i-1}(2n,2i) \lar \dar{\simeq}& \lar \derivedInnermost_i(2n,2i) \arrow[d, equals]\\
0 & \lar 0 & \lar \cdots & 0 \lar & \lar \Sigma^{0,i} \CPL(2n)
\end{tikzcd}
\end{equation*}
is therefore a quasiisomorphism on totalisations. The totalisation of the latter double complex is the $i$-fold suspension of $\CPL(2n)$. We obtain a zig-zag of quasiisomorphisms of $\CPL(2n)$-modules between $\derivedCell(2n, 2i)$ and $\Sigma^{i,i} \CPL(2n)$, as required.

To see that this equivalence is compatible with the hook map, note that we can define a hook map on the (derived) innermost cup complexes simply by noticing that $\derivedInnermost_0(2n,2i) \cong \CPL(0,2n,2i)$, and that (dashed) right innermost cups are preserved by the hook map. In other words, the hook map works just like before, not interacting with dashed cups, which it must always fix. The second part of the above equivalence comes from the identification $\derivedInnermost_i(2n,2i) \cong \Sigma^{0,i} \CPL(2n)$ from \cref{lemma: E_i E_i (B) 0}, which forgets the unique maximal set of dashed cups. This again commutes with the hook map, since the hook map fixes dashed cups.
\end{proof}

\begin{corollary} \label{cor: second shift on Tor} For $n \geq 1$,
$$\Tor_q^{\TL_{2n}}(R,\cell(2n,0)) \cong \begin{cases}
    0 & q < n-1 \\
    H_{q-n+1}(\CPL(2n)) & q \geq n-1.
\end{cases}$$
\end{corollary}

\begin{proof} We apply \cref{prop: suspension of dga is derived cell} in the case $i=n$. Here $\cell(2n,2n)=R$, the trivial module, and $\cell(0,2n) = 0$ by definition. The $L(2n)$-module $\derivedCell(2n,2n)$ therefore takes the form 
$$\mathrm{Cone}(\mathrm{Bar}(\TL(0,2n),\TL_{2n},R) \to 0) \cong \Sigma \mathrm{Bar}(\TL(0,2n),\TL_{2n},R),$$
so 
$$H_q(\derivedCell(2n,2n)) \cong \Tor_{q-1}^{\TL_{2n}}(\TL(0,2n),R).$$ The Temperley--Lieb category admits an (anti-)involution, given by reflecting diagrams: $\TL(i,j) \cong \TL(j,i)$, so $\Tor_{q-1}^{\TL_{2n}}(\TL(0,2n),R) \cong \Tor_{q-1}^{\TL_{2n}}(R, \TL(2n,0))$. As $\TL(2n,0) = \cell(2n,0)$, combing this with $H_q(\derivedCell(2n,2n)) \cong H_{q-n}(L(2n))$ from \cref{prop: suspension of dga is derived cell} completes the proof. 
\end{proof}

Combining this with a result of the last author, we may now deduce \cref{theorem: shift}.

\begin{proof}[Proof of \cref{theorem: shift}] In \cite[Theorem B]{Sroka}, it was shown that
$$\Tor_q^{\TL_{2n}}(R,R) \cong \begin{cases}
    R & q = 0 \\
    0 & 0 < q < n \\
    \Tor_{q-n}^{\TL_{2n}}(R,\cell(2n,0)) & q \geq n.
\end{cases}$$
Combining the final identification with \cref{cor: second shift on Tor} gives the result.
\end{proof}

\section{Resolving the dga of planar loops}\label{section:proof of main technical thm}

In this section we piece together our results so far to prove \cref{theorem:derived-outermost-cup-complex-resolution}. Throughout, we assume $a=0$ so that the derived complex of outermost cups $(\derivedOutermost_{*}(2n), \partial)$ is defined as in \cref{def: derived outermost}, and has an additional weight grading (\cref{lem: weight on outermost}).

\begin{lemma} \label{lemma: outer boundary deg ii}
    Let $i \geq 1$. The $\CPL(2n)$-module map
    $$\Sigma^{i,i}\CPL(2n) \underset{\mathclap{\text{Prop.\ \ref{prop: suspension of dga is derived cell}}}}{\simeq} \derivedCell(2n,2i) \underset{\mathclap{\text{Lem.\ \ref{lem:OutIsCell}}}}{\cong} \derivedOutermost_{i}(2n) \overset{\partial_i}\to  \derivedOutermost_{i-1}(2n) \underset{\mathclap{\text{Lem.\ \ref{lem:OutIsCell}}}}{\cong} \derivedCell(2n,2(i-1)) \underset{\mathclap{\text{Prop.\ \ref{prop: suspension of dga is derived cell}}}}{\simeq} \Sigma^{i-1,i-1}\CPL(2n)$$
is given by right multiplication by $i[\Phi] \in H_{1,1}(\CPL(2n))$.
\end{lemma}

\begin{proof}
Recall from \cref{lemma: new extension from submaximal}
that the boundary map $\partial_i \colon \OutermostCupComplex_i(2n) \to \OutermostCupComplex_{i-1}(2n)$ can be extended to a map $\widetilde{\partial_i} \colon \InnermostCupComplex_*(2n,2i) \to \InnermostCupComplex_*(2n,2i-2)$
of innermost cup complexes (that lemma shows something stronger: each face map $d_k$ comprising $\partial_i$ can be extended to a map $\widetilde{d_k}$). This induces a map of double complexes as follows.
\begin{equation*}
\begin{tikzcd}[column sep=tiny]
     \derivedOutermost_i(2n) \ar[d,"\partial_i"] & \ar[l, two heads] \ar[d,"\widetilde{\partial_i}"] \derivedInnermost_0(2n,2i)  & \ar[l] \cdots & \ar[l] \ar[d,"\widetilde{\partial_i}"] \derivedInnermost_{i-1}(2n,2i) & \ar[l] \ar[d,"\widetilde{\partial_i}"] \derivedInnermost_{i}(2n,2i) & \ar[l] 0 \\
     \derivedOutermost_{i-1}(2n) & \ar[l, two heads] \derivedInnermost_0(2n,2(i-1))  &\ar[l] \cdots & \ar[l] \derivedInnermost_{i-1}(2n,2(i-1)) & \ar[l] 0.
\end{tikzcd}
\end{equation*}

As the chain complexes $\derivedInnermost_0(2n,2i), \ldots, \derivedInnermost_{i-2}(2n,2i)$ and $\derivedInnermost_0(2n,2(i-1))$, \ldots, $\derivedInnermost_{i-2}(2n,2(i-1))$ are all acyclic, the same argument as in the proof of \cref{prop: suspension of dga is derived cell} gives compatible zig-zags of quasiisomorphisms
\begin{align*}
\derivedOutermost_i(2n) &\simeq \Sigma^{i-1}[\derivedInnermost_{i-1}(2n,2i)\leftarrow \derivedInnermost_{i}(2n,2i)]\\
\derivedOutermost_{i-1}(2n) &\simeq \Sigma^{i-1}\derivedInnermost_{i-1}(2n,2(i-1)),
\end{align*}
and therefore identifies the map in the statement of the lemma with the map of totalisations
\begin{equation*}
    \begin{tikzcd}
        {[}\derivedInnermost_{i-1}(2n,2i) \dar{\widetilde{\partial}_i} & \arrow[l, swap, "D\delta_i"] \derivedInnermost_{i}(2n,2i){]} \dar\\
        {[}\derivedInnermost_{i-1}(2n,2(i-1)) & \lar 0 {]}.
    \end{tikzcd}
\end{equation*}
As $\derivedInnermost_{i-1}(2n,2i)$ is also acyclic, we have preferred identifications
\begin{align*}
    H_1([\derivedInnermost_{i-1}(2n,2i)\leftarrow \derivedInnermost_{i}(2n,2i)]) &\cong H_0(\derivedInnermost_{i}(2n,2i)) \underset{\mathclap{\text{Lem.\ \ref{lemma: E_i E_i (B) 0}}}}{\cong} H_0(\CPL(2n)) = R\{\emptyset\}\\
    H_1(\derivedInnermost_{i-1}(2n,2(i-1))) &\underset{\mathclap{\text{Lem.\ \ref{lemma: E_i E_i (B) 0}}}}{\cong} H_1(\CPL(2n)).
\end{align*}
In \cref{sec:Phi} we discussed the diagram $\Phi \in \CPL(2n)_1$, which is a cycle when $a=0$ and so gives a homology class $[\Phi] \in H_1(\CPL(2n))$. As the map we are studying is one of $\CPL(2n)$-modules, and the domain is a free module, the map is determined by where it sends the module generator. We need to prove that, under the identifications given above, the generator $\emptyset$ is sent to $i[\Phi] \in H_{1,1}(\CPL(2n))$.

Consider the maps between these chain complexes in low degrees:
\begin{equation*}
\begin{tikzcd}
\InnermostCupComplex_i(0,2i) \dar{\delta_i} & \TL(0,2n) \otimes  \InnermostCupComplex_i(2n,2i) \dar{\delta_i} \lar & \cdots \lar\\
\InnermostCupComplex_{i-1}(0,2i) \dar{\widetilde{\partial}_i} &  \TL(0,2n)  \otimes \InnermostCupComplex_{i-1}(2n,2i) \dar{\widetilde{\partial}_i} \arrow[l, swap, "d^{i-1}"] & \cdots \lar\\
\InnermostCupComplex_{i-1}(0,2(i-1)) & \TL(0,2n) \otimes \InnermostCupComplex_{i-1}(2n,2(i-1)) \lar & \cdots \lar
\end{tikzcd}
\end{equation*}
To evaluate the map in question, first note that the canonical generator $\emptyset \in \CPL(2n)$ given by the empty diagram corresponds, under the isomorphism of \cref{lemma: E_i E_i (B) 0}, to the diagram $R_{[2i]}^\text{dashed} \in \InnermostCupComplex_i(0,2i)$ consisting of $i$ innermost right cups, all dashed. To evaluate the map we must form $\delta_i(R_{[2i]}^\text{dashed})$, use that the middle chain complex $\derivedInnermost_{i-1}(2n,2i)$ is acyclic to choose a lift $z \in \TL(0,2n)  \otimes \InnermostCupComplex_{i-1}(2n,2i)$ along $d^{i-1}$, and then $\widetilde{\partial}_i(z) \in \TL(0,2n) \otimes \InnermostCupComplex_{i-1}(2n,2(i-1))$ will be a cycle representing the desired first homology class of $\derivedInnermost_{i-1}(2n,2(i-1))$. (Any two choices of lifts $z$ differ by a boundary in the middle chain complex, so the result is a well-defined homology class.)

This calculation is displayed the next diagram; details for each step are given below.
\begin{equation*}
\begin{tikzcd}
R_{[2i]}^\text{dashed} \ar[d, "\delta_i", mapsto] &  & \\
\sum_{j = 1}^i (-1)^{j+1} R_{[2i]}^{\text{dashed\textbackslash}j} &  z \coloneqq \sum_{j = 1}^i (-1)^{j+1} \Phi_l \otimes \Phi_r^j \ar[d, "\widetilde{\partial}_i", mapsto] \arrow[l, swap, "d^{i-1}", mapsto] & \\
 & i \cdot \Phi_l \otimes  \Phi_r \cdot R_{[2i-2]}^\text{dashed} &
\end{tikzcd}
\end{equation*}
For first step, we apply \cref{defn:complex of innermost cups algebraic} and denote by $R_{[2i]}^{\text{dashed\textbackslash}j}$ the diagram obtained from $R_{[2i]}^\text{dashed}$ by undashing its $j$th innermost right cup. For the second step, we observe that $R_{[2i]}^{\text{dashed\textbackslash}j}$ can be written as a product $\Phi_l \cdot \Phi_r^j$ by ``pinning its undashed right cup to a new vertical bar on the left'' (as on the left side of \cref{{fig: bubbleWriting1}}). More precisely, if $\Phi_l \in \TL(0, 2n)$ and $\Phi_r' \in \TL(2n, 2)$ are as in \cref{defn:phir and phil} and $C_{[2i]}^{\text{dashed\textbackslash}j} \in \TL(2,2i)$ denotes the diagram obtained from $R_{[2i]}^{\text{dashed\textbackslash}j}$ by cutting the undashed $j$th cup open and connecting to nodes on the left, then $R_{[2i]}^{\text{dashed\textbackslash}j} = \Phi_l \cdot \Phi_r' \cdot C_{[2i]}^{\text{dashed\textbackslash}j}$, we set $\Phi_r^j \coloneqq \Phi_r' \cdot C_{[2i]}^{\text{dashed\textbackslash}j} \in \InnermostCupComplex_{i-1}(2n,2i)$ so that $d^{i-1}(\Phi_l \otimes \Phi_r^j) = R_{[2i]}^{\text{dashed\textbackslash}j}$. For the final step, we note that $\Phi_r^j$ lies in the complex of submaximal cups and that we may hence evaluate the map $\widetilde{\partial}_i$ constructed in \cref{lemma: new extension from submaximal} on $\Phi_r^j$ using \cref{def: face map lift to submax}. Since each $\Phi_r^j$ has all dashed cups from $R_{[2i]}^\text{dashed}$ except the one from $2j-1$ to $2j$, we obtain
$$\widetilde{d_k}(\Phi_r^j) = \begin{cases}
    \Phi_r \cdot R_{[2i-2]}^\text{dashed} & k=j-1\\
    0 & \text{else},
\end{cases} $$
so that $\widetilde{\partial}_i(\Phi_r^j) = (-1)^{j-1} \Phi_r \cdot R_{[2i-2]}^\text{dashed}$, giving
$$\widetilde{\partial}_i(z) = \sum_{j=1}^{i} (-1)^{j+1} \Phi_l \otimes (-1)^{j-1} \Phi_r \cdot R_{[2i-2]}^\text{dashed}  = i \cdot \Phi_l \otimes  \Phi_r \cdot R_{[2i-2]}^\text{dashed} \text{ as claimed.}$$

To finish this proof, we recall that the isomorphism $\CPL(2n) \simeq \derivedInnermost_{i-1}(2n,2(i-1))$ is given by right multiplication by the diagram $R_{[2i-2]}^\text{dashed}$, so the above corresponds to the 1-cycle $i \cdot \Phi_l \otimes \Phi_r$ in $\CPL(2n)$, which on homology is $i \cdot \Phi$.
\end{proof}

We are now ready to prove the main technical result of the paper, \cref{theorem:derived-outermost-cup-complex-resolution}, which we restate here.
\begin{restate}{Theorem}{theorem:derived-outermost-cup-complex-resolution}
         Let $R$ be concentrated in rank zero, let $a = 0$, and consider the differential bigraded algebra $\CPL(2n)=\CPL(2n; R, 0)$. Then the trivial module $R$ of $\CPL(2n)$ $($concentrated in bidegree $(0,0))$ admits a resolution
    $$0 \leftarrow R \leftarrow \derivedOutermost _0(2n) \leftarrow \derivedOutermost _1(2n) \leftarrow \derivedOutermost _2(2n) \leftarrow \cdots \leftarrow \derivedOutermost _n(2n) \leftarrow 0$$
    of left modules $\derivedOutermost_i(2n)$ over $\CPL(2n)$ with the following properties:
    \begin{enumerate}
        \item There are preferred equivalences $w_i\colon \Sigma^{i,i} \CPL(2n) \xrightarrow{\simeq} \derivedOutermost_i(2n)$ of differential bigraded modules over $\CPL(2n)$ such that 
        $$\Sigma^{i-1,i-1} \CPL(2n) \simeq \derivedOutermost_{i-1}(2n) \leftarrow \derivedOutermost_i(2n) \simeq \Sigma^{i,i} \CPL(2n)$$
        is given by right multiplication by $i\Phi$.
        \item The hook map $h\colon \CPL(2n) \to \CPL(2n+2)$ extends to a map of resolutions $\derivedOutermost_*(2n) \to \derivedOutermost_*(2n+2)$. The $i$-th term of this map corresponds to $\Sigma^{i,i}h$ under the equivalences $w_i$.
    \end{enumerate}
\end{restate}

\begin{proof} We show that the derived complex of outermost cups $\derivedOutermost_*(2n)$ is acyclic (i.e.~a resolution), and that Properties $(a)$ and $(b)$ are satisfied.

Acyclicity follows from \cref{lem: derived horizontally acyclic}.

Note that when $i \geq 1$, \cref{prop: suspension of dga is derived cell} gives a preferred equivalence of $\CPL(2n)$-modules (also for $i=0$, tautologically):
$$\Sigma^{i,i} \CPL(2n) \simeq \derivedCell(2n,2i) \cong \derivedOutermost_i(2n).$$

For Property $(a)$, note that under the above equivalences, \cref{lemma: outer boundary deg ii} shows that the $\CPL(2n)$-module map
$$\Sigma^{i,i} \CPL(2n) \simeq \derivedOutermost_i(2n) \to \derivedOutermost_{i-1}(2n) \simeq \Sigma^{i-1,i-1} \CPL(2n)$$
is given by right multiplication by $i\Phi$.

For Property $(b)$, the hook map $\derivedOutermost_*(2n) \to \derivedOutermost_*(2n+2)$ was defined in \cref{derived weights and hooks} to correspond to the hook map $\derivedCell(2n,2i) \to \derivedCell(2n+2,2i)$, which in turn corresponds to $\Sigma^{i,i} h$ by \cref{prop: suspension of dga is derived cell}.
\end{proof}

\emergencystretch=10em\emergencystretch=10em
\printbibliography
\end{document}